\documentclass[psamsfonts,reqno]{amsart}
\newtheorem{theorem}{Theorem}[section]
\newtheorem{lemma}{Lemma}[section]
\newtheorem{proposition}{Proposition}[section]
\theoremstyle{remark}
\newtheorem{remark}{Remark}[section]
\usepackage{fullpage}
\usepackage[dvips]{graphicx}
\usepackage{subfig}
\theoremstyle{definition}

\newcommand{\ve}{\varepsilon}
\numberwithin{equation}{section}
\numberwithin{table}{section}
\numberwithin{figure}{section}
\def\ontop#1#2{\setbox0\hbox{#2}\copy0\llap{\raise\ht0\hbox{#1}}}
\begin{document}
\newcommand{\shw}{\ensuremath{\overset{\scriptsize \circ}{S}}}
\title{NOTES ON ERROR ESTIMATES FOR GALERKIN APPROXIMATIONS OF
THE `CLASSICAL' BOUSSINESQ SYSTEM AND RELATED HYPERBOLIC PROBLEMS.}
\author{D.C. Antonopoulos}
\address{Department of Mathematics, University of Athens, 15784 Zographou, Greece}
\email{antonod@math.uoa.gr}
\author{V.A. Dougalis}
\address{Department of Mathematics, University of Athens, 15784 Zographou, Greece, and
Institute of Applied and Computational Mathematics, FORTH, 70013 Heraklion, Greece}
\email{doug@math.uoa.gr}
\begin{abstract}
We consider the `classical' Boussinesq system in one space dimension and its symmetric analog. These
systems model two-way propagation of nonlinear, dispersive long waves of small amplitude on the
surface of an ideal fluid in a uniform horizontal channel. We discretize an initial-boundary-value
problem for these systems in space using Galerkin-finite element methods and prove error estimates
for the resulting semidiscrete problems and also for their fully discrete analogs effected by
explicit Runge-Kutta time-stepping procedures. The theoretical orders of convergence obtained are
consistent with the results of numerical experiments that are also presented.
\end{abstract}
\maketitle
\section{Introduction}
In this paper we will analyze Galerkin approximations to the so-called `classical' Boussinesq system
\begin{equation}
\begin{aligned}
\eta_{t} & + u_{x} + (\eta u)_{x} = 0\,,\\
u_{t} & + \eta_{x} + uu_{x} - \tfrac{1}{3}u_{xxt} = 0\,,
\end{aligned}
\label{eq11}
\end{equation}
which is an approximation of the two-dimensional Euler equations of water-wave theory that
models two-way propagation of long waves of small amplitude on the surface of an ideal fluid
in a uniform horizontal channel of finite depth. The variables in $(\ref{eq11})$ are nondimensional
and unscaled; $x$ and $t$ are proportional to position along the channel and time, respectively,
and $\eta=\eta(x,t)$ and $u=u(x,t)$ are proportional to the elevation of the free surface above
a level of rest represented by $\eta=0$, and to the depth-averaged mean horizontal velocity
of the fluid. \par
The system $(\ref{eq11})$ is a member of a general family of Boussinesq systems derived in
\cite{bcs1} that are approximations to the Euler equations of the same order as $(\ref{eq11})$
and whose nonlinear and dispersive terms are of equal importance when written in scaled form.
These systems are written as
\begin{equation}
\begin{aligned}
\eta_{t} & + u_{x} + (\eta u)_{x} + au_{xxx} - b\eta_{xxt} = 0\,, \\
u_{t} & + \eta_{x} + uu_{x} + c\eta_{xxx} - du_{xxt} = 0\,,
\end{aligned}
\label{eq12}
\end{equation}
where $a$, $b$, $c$, $d$ are real parameters satisfying $a+b=\tfrac{1}{2}(\theta^{2}-1/3)$,
$c+d=\tfrac{1}{2}(1-\theta^{2})$, where $0\leq \theta\leq 1$.
The specific system $(\ref{eq11})$
has been previously formally derived from the Euler equations, in the appropriate parameter
regime, in \cite{per1}, \cite{per3}, and \cite{wh}. It has been widely used in the engineering fluid
mechanics literature for computations of long, nonlinear dispersive waves, starting with \cite{per1}
and \cite{per2}. \par
Existence and uniqueness of the initial-value problem for $(\ref{eq11})$ posed for $x \in \mathbb{R}$
and $t\geq 0$ and supplemented by initial conditions of the form
\begin{equation}
\eta(x,0) = \eta_{0}(x)\,, \quad u(x,0)=u_{0}(x)\,, \quad x\in \mathbb{R}\,,
\label{eq13}
\end{equation}
was studied by Schonbeck, \cite{sch} and Amick, \cite{am}. In these papers global existence and
uniqueness was established for infinitely differentiable initial data of compact support such that
$\eta_{0}(x) > -1$, $x \in \mathbb{R}$. In \cite{bcs2} the theory of \cite{sch} and \cite{am}
was used to prove that given initial data
$(\eta_{0},u_{0}) \in H^{s}(\mathbb{R})\times H^{s+1}(\mathbb{R})$ for $s\geq 1$, such that
$\inf_{x\in\mathbb{R}}\eta_{0}(x) > -1$, there is a unique solution $(\eta,u)$ which, for any
$T>0$, lies in $C(0,T;H^{s}(\mathbb{R}))\times C(0,T;H^{s+1}(\mathbb{R}))$. (Here, $H^{s}(\mathbb{R})$
is the usual, $L^{2}$-based Sobolev space of functions on $\mathbb{R}$ and $C(0,T;X)$ denotes
the space of functions $\phi=\phi(t)$ having, for each $t\in [0,T]$, values in a Banach space $X$
and are such that the map $[0,T] \mapsto \|\phi\|_{X}$ is continuous.) \par
It is well known that the initial-value problem $(\ref{eq11})$, $(\ref{eq13})$ has classical
{\em{solitary-wave}} solutions. In \cite{ad} we review the properties of such solutions, we construct
them numerically and investigate, by means of numerical experiments, the resolution of general
initial profiles into sequences of solitary waves, the details of the interaction of solitary waves
during head-on and overtaking collisions, the interaction of solitary waves with boundaries, and
several other issues. These numerical experiments were performed in \cite{ad} by fully discrete
Galerkin methods that approximate solutions of $(\ref{eq11})$ posed on finite intervals and subject
to boundary conditions. It is of interest therefore to study initial-boundary-value problems
for $(\ref{eq11})$ and establish error estimates for their numerical approximations. \par
In this paper we shall analyze the numerical solution of the following initial-boundary-value problem
(ibvp) for $(\ref{eq11})$: For some $0<T<\infty$, we seek $\eta=\eta(x,t)$, $u=u(x,t)$, defined
for $0\leq x\leq 1$, $0\leq t\leq T$, and satisfying
\begin{equation}
\begin{aligned}
\eta_{t} & + u_{x} + (\eta u)_{x} = 0\,, \qquad \qquad (x,t) \in [0,1]\times [0,T]\,, \\
u_{t} & + \eta_{x} + uu_{x} - \tfrac{1}{3}u_{xxt} = 0\,, \quad (x,t) \in [0,1]\times [0,T]\,, \\
\eta(& x,0) = \eta_{0}(x)\,, \quad u(x,0)=u_{0}(x)\,, \quad x \in [0,1]\,, \\
u(&0,t) = 0\,, \quad u(1,t)=0\,, \quad t \in [0,T]\,.
\end{aligned}
\tag{CB}
\label{cb}
\end{equation}
This ibvp (for $0 < t < \infty$) has been studied by Adamy, \cite{ada}, who showed
that it has weak (distributional)
solutions $(\eta, u) \in L^{\infty}(\mathbb{R}^{+};L^{1}\times H_{0}^{1})$ provided e.g. that
$\eta_{0} \in L^{1}$, $u_{0}\in H_{0}^{1}$ with $\inf_{x\in [0,1]}\eta_{0}(x) > -1$. (Here
$L^{1}=L^{1}(0,1)$, and $H_{0}^{1}=H_{0}^{1}(0,1)$ is the subspace of the Sobolev space
$H^{1}(0,1)$ whose elements vanish at $x=0$ and $x=1$.) The proof uses a parabolic
regularization of the first pde of $(\ref{cb})$, a technique used in the context of the Cauchy
problem by Schonbeck, \cite{sch}. It should be noted that the homogeneous Dirichlet boundary
conditions on $u$ in $(\ref{cb})$ are one kind of boundary conditions that lead to well posed ibvp's
in the case of the linearized system, \cite{fp}. It is noteworthy that the CB system needs
only two boundary conditions for well-posedness as opposed to the four boundary conditions
(for example, Dirichlet conditions on $\eta$ and $u$ at each endpoint of the interval) required
in the case of other Boussinesq systems, such as the BBM-BBM ($a=c=0$, $b=d$ in $(\ref{eq12})$,\cite{bc}),
or the Bona-Smith systems ($a=0$, $b=d>0$, $c<0$ in $(\ref{eq12})$, \cite{adm1}.)
The case of the homogeneous Dirichlet boundary conditions in
$(\ref{cb})$ may be viewed as  a base for studying the nonhomogeneous analog wherein $u(0,t)$
and $u(1,t)$ are given functions of $t\geq 0$ corresponding to measurements of the velocity
variable at two points along the channel. \par
In \cite{bcl}, Bona, Colin, and Lannes introduced another type of Boussinesq systems that they called
`completely symmetric'. These are obtained by a nonlinear change of variables from the usual systems
$(\ref{eq12})$, and have certain mathematical and modelling advantages over the latter. In this paper
we shall also therefore consider the analogous problem for the symmetric system, specifically the ibvp
\begin{equation}
\begin{aligned}
\qquad \eta_{t} & + u_{x} + \tfrac{1}{2}(\eta u)_{x} = 0\,, \qquad \qquad \qquad
\quad (x,t) \in [0,1]\times [0,T]\,, \\
u_{t} & + \eta_{x} + \tfrac{3}{2}uu_{x} + \tfrac{1}{2}\eta\eta_{x}
- \tfrac{1}{3}u_{xxt} = 0\,, \quad (x,t) \in [0,1]\times [0,T]\,, \\
\eta(& x,0) = \eta_{0}(x)\,, \quad u(x,0)=u_{0}(x)\,, \quad x \in [0,1]\,, \\
u(&0,t) = 0\,, \quad u(1,t)=0\,, \quad t \in [0,T]\,.
\end{aligned}
\tag{SCB}
\label{scb}
\end{equation}
It is not hard to see that the solution of $(\ref{scb})$ satisfies the conservation property
\begin{equation}
\|\eta(t)\|^{2} + \|u(t)\|^{2} + \tfrac{1}{3}\|u_{x}(t)\|^{2} =
\|\eta_{0}\|^{2} + \|u_{0}\|^{2}+ \tfrac{1}{3}\|u_{0}'\|^{2}\,,
\label{eq14}
\end{equation}
for $t\geq 0$, which simplifies the study of its well-posedness and the analysis of its numerical
approximations as will be seen in the sequel. \par
In this paper we shall analyze semidiscrete and fully discrete Galerkin-finite element methods
for the ibvp's $(\ref{cb})$ and $(\ref{scb})$ assuming that their solutions are sufficiently
smooth. Previously, rigorous error estimates for Galerkin
methods for these `classical' Boussinesq systems were proved in the case of the {\em{periodic}}
initial-value problem in \cite{ant} and \cite{adm3}. (For numerical work for this type and for other
Boussinesq systems of the form $(\ref{eq12})$ we refer the reader e.g. to \cite{ant}, \cite{ad1998},
\cite{ad2000}, \cite{p1996}, \cite{pd}, \cite{bc}, \cite{adm2}, \cite{bdm1}, \cite{bdm2},
\cite{ddlm}, \cite{ad}.) \par
The analysis of Galerkin methods for $(\ref{cb})$ or $(\ref{scb})$ is of considerable interest due to
the loss of optimal order of accuracy that emerges from the error estimates and is also supported
by the numerical experiments. We shall investigate this phenomenon in detail in the paper but one
may in general say that in the uniform spatial mesh case the (limited) loss of accuracy seems
to be due to a combination of effects related to the boundary conditions and the specific form of the
`classical' Boussinesq systems. (For example, no loss of accuracy occurs in the case of periodic
boundary conditions on $\eta$ and $u$ for these systems, cf. \cite{ant}, \cite{adm3}.) In the case
of general (quasiuniform) mesh the (more severe) loss of accuracy seems to stem from the lack of
cancellation effects that are present in the uniform mesh case, and from the hyperbolic character of
the first p.d.e. of these systems. (The loss of optimal order of accuracy in standard Galerkin
semidiscretizations of first-order hyperbolic equations manifested in various contexts is well
known and was early observed e.g. by Dupont in \cite{d}.) \par
In section $2$ we consider the standard Galerkin semidiscretizations of $(\ref{cb})$ and $(\ref{scb})$
in the space of piecewise linear, continuous functions. In the case of a quasiuniform mesh, we prove
that if the solution $(\eta,u)$ of $(\ref{cb})$ or $(\ref{scb})$ is sufficiently smooth and
$(\eta_{h}, u_{h})$ is the semidiscrete approximation, then $\|\eta-\eta_{h}\|=O(h)$ and
$\|u-u_{h}\|_{1}=O(h)$. (Here $\|\cdot\|$, $\|\cdot\|_{1}$ denote, respectively, the $L^{2}$ and
$H^{1}$ norms on $[0,1]$.) In the case of uniform mesh, a suitable superaccuracy result for the error
of the interpolant into the finite element subspace (a consequence of cancellations due to the uniform
mesh) affords proving the improved estimates $\|\eta-\eta_{h}\|=O(h^{3/2})$ and $\|u-u_{h}\|=O(h^{2})$.
These rates of convergence are confirmed by the numerical experiments at the end of the section. \par
In section $3$ we turn to the semidiscretization of $(\ref{cb})$ or $(\ref{scb})$ using as finite
element subspace the $C^{2}$ cubic splines. In the case of quasiuniform mesh we now get, as expected,
$\|\eta-\eta_{h}\|=O(h^{3})$, $\|u-u_{h}\|_{1}=O(h^{3})$, following the technique of the analogous proof
in the previous section. Use of a uniform mesh and of relevant superconvergence results of
Wahlbin, \cite{w}, allows proving a series of suitable superaccuracy estimates for the error of the
interpolant and the error of the elliptic projection into the cubic spline subspaces. These results
lead to error estimates such as $\|\eta-\eta_{h}\|=O(h^{3.5}\sqrt{\ln 1/h})$,
$\|u-u_{h}\|=O(h^{4}\sqrt{\ln 1/h})$, that are consistent with the rates of convergence of the
errors obtained from numerical experiments.\par
In section $4$ we turn to the analysis of fully discrete schemes. We consider only {\em{explicit}}
time-stepping schemes in order to avoid the more costly implicit methods that require solving
nonlinear systems of equations at every time step. Of course, with explicit methods there arises
the issue of stability of the fully discrete schemes. We confine ourselves to a uniform mesh
spatial discretization and consider three representative explicit Runge-Kutta schemes, namely,
the Euler, the improved Euler, and the classical, four-stage, Runge-Kutta methods, whose orders
of accuracy are $1$, $2$, and $4$, respectively. We couple the Euler and improved Euler schemes
with a piecewise linear spatial discretization and the fourth-order RK scheme with cubic splines.
We show that the stability restrictions on the time step $k$ needed by these schemes are of form
$k=O(h^{2})$, $k=O(h^{4/3})$, and $k\leq \lambda_{0}h$ for a constant $\lambda_{0}$ sufficiently small,
for the schemes with Euler, improved Euler, and fourth-order RK temporal discretizations, respectively.
Under these restrictions, we prove optimal-order in time error estimates for the fully discrete schemes;
the spatial rates of convergence are these of the semidiscrete approximations. The evidence of numerical
experiments is consistent with the theoretical results. \par
In section $5$ we analyze a nonstandard Galerkin semidiscrete approximation of $(\ref{cb})$ and
$(\ref{scb})$ that approximates $\eta$ by piecewise linear continuous functions and $u$ by piecewise
quadratic $C^{1}$ functions defined on the same (uniform) mesh. (The method may be generalized using
the analogous pairs of higher-order splines but here we confine ourselves to the low-order case.)
A series of superaccuracy results for the error of the $L^{2}$ projection onto the space of
piecewise linear continuous functions is developed, one of which makes use of a result of Demko,
\cite{de}, on the exponential decay of the off-diagonal elements of the inverse of a certain
tridiagonal matrix. With the aid of these results we show optimal-order error estimates e.g.
of the type $\|\eta-\eta_{h}\|=O(h^{2})$, $\|u-u_{h}\|=O(h^{3})$. Finally, in section $6$ we make
some remarks on the application of standard Galerkin methods on some simple first-order hyperbolic
problems. Specifically, superaccuracy tools that were developed in the case of uniform mesh in
previous sections are combined with further similar estimates proved in this section to provide
the basis for extending the optimal-order $L^{2}$ error estimate of Dupont, \cite{d}, for the periodic
initial value problem, to the case of two types of initial-boundary value problems for first-order
hyperbolics. The situation is contrasted with what happens in the case of nonuniform meshes by means
of numerical examples. \par
In the paper we use the following notation: We let $C^{k}=C^{k}[0,1]$, $k=0,1,2,\dots$, denote the
space of $k$ times continuously differentiable functions on $[0,1]$ and define
$C_{0}^{k}=\{\phi \in C^{k}; \phi(0)=\phi(1)=0\}$. We let for integer $k\geq 0$ $H^{k}$,
$\|\cdot\|_{k}$ denote the usual, $L^{2}$-based Sobolev space of classes of functions on $[0,1]$
and its associated norm. (In the case $k=1$ we use the equivalent norm defined by
$\|v\|_{1} = (\|v\|^{2} + \tfrac{1}{3}\|v'\|^{2})^{1/2}$.) The inner product and norm on
$L^{2}=L^{2}(0,1)$ we denote simply by $\|\cdot\|$, $(\cdot,\cdot)$, respectively. The norms on
$L^{\infty}=L^{\infty}(0,1)$ and on $W_{\infty}^{k}=W_{\infty}^{k}(0,1)$ we denote by
$\|\cdot\|_{\infty}$, $\|\cdot\|_{k,\infty}$, respectively. We let $\mathbb{P}_{r}$ be the
polynomials of degree $\leq r$,and by $\langle\cdot,\cdot\rangle$, $|\cdot|$, we denote the
Euclidean inner product and norm on $\mathbb{R}^{N}$.
\section{Standard Galerkin semidiscretization with piecewise linear, continuous functions}
\subsection{Semidiscretization on a quasiuniform mesh}
Let $0=x_{1}<x_{2}<\dots<x_{N+1}=1$ denote a quasiuniform partition of $[0,1]$ with
$h:=\max_{i}(x_{i+1}-x_{i})$, and let
$S_{h}^{2}:=\{\phi \in C^{0} : \phi\big|_{[x_{j},x_{j+1}]} \in \mathbb{P}_{1}\,, 1\leq j\leq N\}$,
$S_{h,0}^{2}=\{\phi \in S_{h}^{2}\,, \phi(0)=\phi(1)=0\}$.
Let $I_{h}$, $I_{h,0}$ denote the interpolation operators, with respect to the partition
$\{x_{j}\}$, into the spaces $S_{h}^{2}$, $S_{h,0}^{2}$ respectively. Then, it is well known that
there exists a constant $C$ independent of $h$ such that
\begin{equation}
\|w-I_{h}w\| + h\|(w - I_{h}w)'\| \leq Ch^{k}\|w^{(k)}\|\,,
\label{eq21}
\end{equation}
for $w \in H^{k}$, $k=1,2$, and that a similar estimate holds for $I_{h,0}w$ if
$w \in H^{k}\cap H_{0}^{1}$. (In the sequel, $C$ will denote a generic constant,
independent of $h$.) Let $a(\cdot,\cdot)$ denote the bilinear form
\[
a(\psi,\chi) := (\psi,\chi) + \tfrac{1}{3}(\psi',\chi') \quad \forall \psi,\chi \in S_{h,0}^{2}\,,
\]
and $R_{h} : H^{1} \to S_{h,0}^{2}$ be the elliptic projection operator relative to
$a(\cdot,\cdot)$, defined by
\[
a(R_{h}w,\chi) = a(w,\chi) \quad \forall \chi \in S_{h,0}^{2}\,.
\]
It follows by standard estimates that for $k=0,1$
\begin{equation}
\|R_{h}w - w\|_{k} \leq C h^{2-k}\|w''\| \quad \text{if}\quad w \in H^{2}\cap H_{0}^{1}\,.
\label{eq22}
\end{equation}
There also holds, \cite{ddw}, that
\begin{equation}
\|R_{h}w - w\|_{\infty} + h\|R_{h}w - w\|_{1,\infty} \leq Ch^{2} \|w\|_{2,\infty}\,,
\label{eq23}
\end{equation}
provided that $w \in W^{2,\infty}\cap H_{0}^{1}$. \par
As a consequence of the quasiuniformity of the mesh, the inverse inequalities
\begin{align}
\|\chi\|_{1} & \leq Ch^{-1} \|\chi\|\,,
\label{eq24} \\
\|\chi\|_{\infty} & \leq C h^{-1/2} \|\chi\|\,,
\label{eq25}
\end{align}
hold for any $\chi \in S_{h}^{2}$ (or $\chi \in S_{h,0}^{2}$), and so does the
estimate, \cite{ddw},
\begin{equation}
\|Pv - v\|_{\infty} \leq C h^{2} \|v\|_{2,\infty}\,, \  \quad \text{for}
\quad v \in W^{2,\infty}\,,
\label{eq26}
\end{equation}
where $P : L^{2} \to S_{h}^{2}$ is the $L^{2}$-projection operator onto $S_{h}^{2}$. \par
The {\em standard Galerkin semidiscretization} on $S_{h}^{2}$ of $(\ref{cb})$ is defined
as follows: We seek $\eta_{h} : [0,T] \to S_{h}^{2}$, $u_{h} : [0,T] \to S_{h,0}^{2}$,
such that for $t \in [0,T]$
\begin{equation}
\begin{aligned}
(\eta_{ht},\phi) + (u_{hx},\phi) + ((\eta_{h}u_{h})_{x},\phi) & = 0 \quad \forall \phi \in S_{h}^{2}\,,\\
a(u_{ht},\chi) + (\eta_{hx},\chi) + (u_{h}u_{hx},\chi) & = 0 \quad \forall \chi \in S_{h,0}^{2}\,,
\end{aligned}
\label{eq27}
\end{equation}
with initial conditions
\begin{equation}
\eta_{h}(0) = P\eta_{0}\,, \qquad u_{h}(0)=R_{h}u_{0}\,.
\label{eq28}
\end{equation}
Similarly, we define the analogous semidiscretization of $(\ref{scb})$, which is given
for $0\leq t\leq T$ by
\begin{equation}
\begin{aligned}
(\eta_{ht},\phi) + (u_{hx},\phi) + \tfrac{1}{2}((\eta_{h}u_{h})_{x},\phi) & = 0
\quad \forall \phi \in S_{h}^{2}\,,\\
a(u_{ht},\chi) + (\eta_{hx},\chi) + \tfrac{3}{2}(u_{h}u_{hx},\chi) + \tfrac{1}{2}(\eta_{h}\eta_{hx},\chi)& = 0
\quad \forall \chi \in S_{h,0}^{2}\,,
\end{aligned}
\label{eq29}
\end{equation}
with
\begin{equation}
\eta_{h}(0) = P\eta_{0}\,, \qquad u_{h}(0)=R_{h}u_{0}\,.
\label{eq210}
\end{equation}
Upon choice of a basis for $S_{h}^{2}$, it is seen that the semidiscrete problems
$(\ref{eq27})$-$(\ref{eq28})$ and $(\ref{eq29})$-$(\ref{eq210})$ represent initial-value problems
for systems of o.d.e's. Clearly, these systems have unique solutions at least locally in time.
One conclusion of the next proposition is that they possess unique solutions up to $t=T$,
where $[0,T]$ will denote henceforth the interval of existence of solutions of
$(\ref{cb})$ or $(\ref{scb})$.
\begin{proposition}
Let $h$ be sufficiently small. Suppose that the solutions of $(\ref{cb})$, and $(\ref{scb})$,
are such that $\eta \in C(0,T;W_{\infty}^{2})$, $u\in C(0,T;W_{\infty}^{2}\cap H_{0}^{1})$. Then,
the semidiscrete problems $(\ref{eq27})$, $(\ref{eq28})$ and $(\ref{eq29})$, $(\ref{eq210})$
have unique solutions $(\eta_{h},u_{h})$ for $0\leq t\leq T$ that satisfy
\begin{align}
\max_{0\leq t\leq T}\|\eta(t) - \eta_{h}(t)\| & \leq C h\,,
\label{eq211} \\
\max_{0\leq t\leq T}\|u(t) - u_{h}(t)\|_{1} & \leq C h\,.
\label{eq212}
\end{align}
\end{proposition}
\begin{proof}
We first consider the case of the symmetric system $(\ref{scb})$. Putting
$\phi = \eta_{h}$ and $\chi = u_{h}$ in $(\ref{eq29})$ and adding the resulting equations,
we obtain the discrete analog of $(\ref{eq14})$, i.e. that
\begin{equation}
\|\eta_{h}(t)\|^{2} + \|u_{h}(t)\|_{1}^{2} = \|\eta_{h}(0)\|^{2} + \|u_{h}(0)\|_{1}^{2}
\label{eq213}
\end{equation}
is valid in the temporal interval of existence of solutions of $(\ref{eq29})$-$(\ref{eq210})$.
By standard o.d.e. theory we conclude that the system $(\ref{eq29})$-$(\ref{eq210})$ possesses
a unique solution in any finite temporal interval $[0,T]$. \par
We now let $\rho := \eta - P\eta$, $\theta := P\eta - \eta_{h}$, $\sigma := u-R_{h}u$,
$\xi := R_{h}u - u_{h}$. Using $(\ref{scb})$ and $(\ref{eq29})$-$(\ref{eq210})$ we obtain
for $0\leq t\leq T$,
\small
\begin{align}
(\theta_{t},\phi) + (\sigma_{x} + \xi_{x},\phi) +
\tfrac{1}{2}((\eta u - \eta_{h}u_{h}),\phi) & = 0 \quad \forall \phi \in S_{h}^{2}\,,
\label{eq214} \\
a(\xi_{t},\chi) + (\rho_{x}+\theta_{x},\chi) + \tfrac{3}{2}(uu_{x}-u_{h}u_{hx},\chi)
+\tfrac{1}{2}(\eta\eta_{x}-\eta_{h}\eta_{hx},\chi) & = 0 \quad \forall \chi \in S_{h,0}^{2}\,.
\label{eq215}
\end{align}
\normalsize
Note that
\begin{align*}
\eta u - \eta_{h}u_{h} & = u(\rho + \theta) - (\rho + \theta)(\sigma + \xi) + \eta (\sigma + \xi)\,, \\
uu_{x} - u_{h}u_{hx} & = (u\sigma)_{x} + (u\xi)_{x} - (\sigma \xi)_{x} - \sigma\sigma_{x} - \xi\xi_{x}\,,\\
\eta\eta_{x} - \eta_{h}\eta_{hx} & = (\eta\rho)_{x} + (\eta\theta)_{x} - (\rho\theta)_{x} - \rho\rho_{x}
-\theta\theta_{x}\,.
\end{align*}
By continuity, in view of $(\ref{eq210})$, we conclude that there exists a maximal temporal instance
$t_{h} > 0$ such that $\|\theta(t)\|_{\infty} \leq 1$ for $t\leq t_{h}$\,. Suppose that
$t_{h} < T$. Then, taking $\phi=\theta$ in $(\ref{eq214})$ and using $(\ref{eq21})$, $(\ref{eq22})$,
$(\ref{eq23})$, $(\ref{eq24})$, $(\ref{eq26})$, and integrating by parts, we have for $0\leq t\leq t_{h}$
\begin{equation}
\begin{aligned}
\tfrac{1}{2}\tfrac{d}{dt}\|\theta\|^{2} & = - (\sigma_{x},\theta) - (\xi_{x},\theta)
-\tfrac{1}{2}\bigl[ ((\rho u)_{x},\theta) + \tfrac{1}{2}(u_{x}\theta,\theta)
- ((\rho\sigma)_{x},\theta) \\
&\hspace{10.8pt} - ((\rho\xi)_{x},\theta) -\tfrac{1}{2}(\sigma_{x}\theta,\theta)
- \tfrac{1}{2}(\xi_{x}\theta,\theta) + ((\eta\sigma)_{x},\theta) + ((\eta\xi)_{x},\theta\bigr]\\
& \leq \|\sigma_{x}\|\|\theta\| + \|\xi_{x}\|\|\theta\| + \tfrac{1}{2}\|u\|_{\infty}\|\rho_{x}\|\|\theta\|
+ \tfrac{1}{2}\|u_{x}\|_{\infty}\|\rho\|\|\theta\| \\
&\hspace{10.8pt} + \tfrac{1}{4}\|u_{x}\|_{\infty}\|\theta\|^{2} + \tfrac{1}{2}\|\rho_{x}\|\|\sigma\|_{\infty}\|\theta\|
+ \tfrac{1}{2}\|\rho\|\|\sigma_{x}\|_{\infty}\|\theta\|
+ \tfrac{1}{2}\|\rho\|_{\infty}\|\xi_{x}\|\|\theta\| \\
&\hspace{10.8pt} + C\|\rho_{x}\|\|\xi\|_{1}\|\theta\| + \tfrac{1}{4}\|\sigma_{x}\|_{\infty}\|\theta\|^{2}
+ \tfrac{1}{4}\|\xi_{x}\|\|\theta\| + \tfrac{1}{2}\|\eta\|_{\infty}\|\sigma_{x}\|\|\theta\| \\
&\hspace{10.8pt} + \tfrac{1}{2}\|\eta_{x}\|_{\infty}\|\sigma\|\|\theta\|
+ \tfrac{1}{2}\|\eta\|_{\infty}\|\xi_{x}\|\|\theta\| + \tfrac{1}{2}\|\eta_{x}\|_{\infty}\|\xi\|\|\theta\|\\
& \leq C (h + \|\xi\|_{1} + \|\theta\|)\|\theta\|\,,
\end{aligned}
\label{eq216}
\end{equation}
where $C$ is independent of $t_{h}$. In addition, putting $\chi=\xi$ in $(\ref{eq215})$ we similarly
obtain, for $0\leq t\leq t_{h}$
\begin{equation}
\begin{aligned}
\tfrac{1}{2}\tfrac{d}{dt}\|\xi\|^{2}_{1} & = (\rho + \theta,\xi_{x})
+ \tfrac{3}{2}\bigl[(u\sigma,\xi_{x}) + (u\xi,\xi_{x}) - (\sigma\xi,\xi_{x}) + (\sigma\sigma_{x},\xi)\bigr] \\
&\hspace{10.8pt} + \tfrac{1}{2}\bigl[ (\eta\rho,\xi_{x}) + (\eta\theta,\xi_{x})
-(\rho\theta,\xi_{x}) - \tfrac{1}{2}(\rho\xi_{x},\rho) - (\theta\xi_{x},\theta)\bigr]\\
&\leq \|\rho\|\|\xi_{x}\| + \|\theta\|\|\xi_{x}\| + \tfrac{3}{2}\|u\|_{\infty}\|\sigma\|\|\xi_{x}\|
+\tfrac{3}{2}\|u\|_{\infty}\|\xi\|\|\xi_{x}\| \\
&\hspace{10.8pt} + \tfrac{3}{2}\|\sigma\|_{\infty}\|\xi\|\|\xi_{x}\|
+ \tfrac{3}{2}\|\sigma\|_{\infty}\|\sigma_{x}\|\|\xi\| + \tfrac{1}{2}\|\eta\|_{\infty}\|\rho\|\|\xi_{x}\|\\
&\hspace{10.8pt} + \tfrac{1}{2}\|\eta\|_{\infty}\|\theta\|\|\xi_{x}\|
+ \tfrac{1}{2}\|\rho\|\|\xi_{x}\| + C\|\rho\|_{1}\|\rho\|\|\xi\|_{1} + \tfrac{1}{2}\|\xi_{x}\|\|\theta\| \\
&\leq C (h^{2} + \|\xi\|_{1} + \|\theta\|)\|\xi\|_{1}\,.
\end{aligned}
\label{eq217}
\end{equation}
From $(\ref{eq216})$ and $(\ref{eq217})$ it is seen that for $0\leq t\leq t_{h}$ there holds
\[
\tfrac{d}{dt}(\|\theta\| + \|\xi\|_{1}) \leq C (h + \|\theta\| + \|\xi\|_{1})\,,
\]
from which, by Gronwall's Lemma and $(\ref{eq210})$, we conclude that
\begin{equation}
\|\theta(t)\| + \|\xi(t)\|_{1} \leq C h\,, \quad 0\leq t\leq t_{h}\,,
\label{eq218}
\end{equation}
where $C$ is independent of $t_{h}$. Since by $(\ref{eq25})$
$\|\theta\|_{\infty}\leq Ch^{-1/2}\|\theta\|$, if $h$ is sufficiently small the maximality
property of $t_{h}$ is contradicted. Therefore  we may take $t_{h}=T$, and $(\ref{eq211})$ and
$(\ref{eq212})$ follow from $(\ref{eq218})$, $(\ref{eq21})$ and $(\ref{eq22})$. \par
Consider now $(\ref{cb})$ and its semidiscrete analog $(\ref{eq27})$-$(\ref{eq28})$. The invariance property
$(\ref{eq213})$ no longer holds, and the i.v.p. $(\ref{eq27})$-$(\ref{eq28})$ has a local
unique solution. Using the same notation as in the case of $(\ref{scb})$,
we consider the i.v.p. of finding $\theta(t) \in S_{h}^{2}$, $\xi(t) \in S_{h,0}^{2}$ for $t\geq 0$,
such that
\begin{equation}
\begin{aligned}
(\theta_{t},\phi) & + (\sigma_{x}+\xi_{x},\phi) +
\bigl( [u(\rho + \theta) - (\rho+\theta)(\sigma+\xi)+\eta(\sigma+\xi)],\phi\bigr) = 0
\quad \forall \phi \in S_{h}^{2}\,,\\
a(\xi_{t},\chi) & + (\rho_{x}+\theta_{x},\chi)
- \bigl( u(\sigma+\xi)-\sigma\xi,\chi'\bigr) - (\sigma\sigma_{x}+\xi\xi_{x},\chi) = 0
\qquad \forall \chi \in S_{h,0}^{2}\,,\\
\theta(0) & = 0\,, \quad \xi(0) = 0\,.
\end{aligned}
\label{eq219}
\end{equation}
Obviously, $(\ref{eq219})$ has a local unique solution. Let $t_{h} \in (0,T)$ be the maximal time instance
for which this solution exists and satisfies $\|\theta(t)\|_{\infty}\leq 1$ for $0\leq t\leq t_{h}$. Then,
as in the case of $(\ref{scb})$, we obtain again that
\[
\|\theta(t)\| + \|\xi(t)\|_{1} \leq C h\,, \quad 0\leq t\leq t_{h}\,,
\]
where $C$ is independent of $t_{h}$. We conclude that we may take $t_{h}=T$. If $\eta_{h}=P\eta-\theta$,
$u_{h} = R_{h}u - \xi$, it follows that $(\eta_{h},u_{h})$ is a unique solution of $(\ref{eq27})$-$(\ref{eq28})$
in $[0,T]$ and that it satisfies the estimates $(\ref{eq211})$ and $(\ref{eq212})$.
\end{proof}
\subsection{Uniform mesh}
We now turn to the case of {\em uniform mesh}. For integer $N \geq 2$ we let $h=1/N$ and $x_{i}=(i-1)h$,
$i=1,2,\dots,N+1$. (We shall use the previously established notation for the finite-dimensional spaces
$S_{h}^{2}$, $S_{h,0}^{2}$, and their associated interpolation and projection operators.)
We let $\{\phi_{j}\}_{j=1}^{N+1}$ denote the basis
of $S_{h}^{2}$ satisfying $\phi_{j}(x_{i}) = \delta_{ij}$, $1\leq i,j\leq N+1$. In the following lemma
we collect results that will prove useful in the error estimates that follow.
\begin{lemma}
(i) Let $G_{ij} = (\phi_{j},\phi_{i})$, $1\leq i,j\leq N+1$. Then there exist positive constants
$c_{1}$ and $c_{2}$ such that
\[
c_{1}h|\gamma|^{2} \leq \langle G\gamma,\gamma\rangle
 \leq c_{2}h|\gamma|^{2}\quad \forall \gamma \in \mathbb{R}^{N+1}\,.
\]
(ii) Let $b \in \mathbb{R}^{N+1}$, $\gamma = G^{-1}b$, and $\psi=\sum_{j=1}^{N+1}\gamma_{j}\phi_{j}$.
Then
\[
\|\psi\| \leq (c_{1}h)^{-1/2}|b|\,.
\]
(iii) Let $w \in C^{3}$. Then, there exists a constant $C_{1}=C_{1}(\|w^{(3)}\|_{\infty})$ such that
for any $\hat{x} \in [x_{i},x_{i+1}]$
\[
(w-I_{h}w)(x) = -\tfrac{1}{2}w''(\hat{x})(x-x_{i})(x_{i+1}-x) + \tilde{g}(x)\,, \quad
x_{i}\leq x\leq x_{i+1}\,,
\]
where $\|\tilde{g}\|_{\infty} + h\|\tilde{g}'\|_{\infty} \leq C_{1}h^{3}$.
\end{lemma}
\begin{proof} The proofs of $(i)$, $(ii)$, $(iii)$ are given in \cite{d} in case the finite element subspace
consists of continuous, piecewise linear, periodic functions on $[0,1]$. It is straightforward to adapt
them in the case of $S_{h}^{2}$ at hand.
\end{proof}
The following superapproximation result, a consequence of cancellations due to the uniform mesh, is central
for the sequel.
\begin{lemma} Let $v \in C^{3}$ and $w \in C^{1}$. If $\ve:=v - I_{h}v$ and $\psi \in S_{h}^{2}$
is such that
\[
(\psi,\phi) = ((w\ve)',\phi) \quad \forall \phi \in S_{h}^{2}\,,
\]
then $\|\psi\|\leq C h^{3/2}$. If in addition $w(0)=w(1)=0$, then $\|\psi\|\leq C h^{2}$.
\end{lemma}
\begin{proof}
Let $b_{i}:=((w\ve)',\phi_{i}) = -(w\ve,\phi_{i}')$, $1\leq i\leq N+1$. Clearly,
$|b_{1}|=O(h^{2})$ and $|b_{N+1}|=O(h^{2})$. Let $2\leq i\leq N$. Then, from Lemma $2.1(iii)$ we have
\begin{align*}
b_{i} & = \tfrac{v''(x_{i})}{2h}\int_{x_{i-1}}^{x_{i}}w(x)(x-x_{i-1})(x_{i}-x)dx
-\tfrac{v''(x_{i})}{2h}\int_{x_{i}}^{x_{i+1}}w(x)(x-x_{i})(x_{i+1}-x)dx + O(h^{3}) \\
& = \tfrac{-v''(x_{i})}{2h}\int_{x_{i}}^{x_{i+1}}[w(x)-w(x-h)](x-x_{i})(x_{i+1}-x)dx + O(h^{3})\\
& = \tfrac{-v''(x_{i})}{2}\int_{x_{i}}^{x_{i+1}}w'(\xi_{x})(x-x_{i})(x_{i+1}-x)dx + O(h^{3})\,.
\end{align*}
Therefore, $b_{i}=O(h^{3})$, $2\leq i\leq N$; consequently $|b|=O(h^{2})$ and $\|\psi\|=O(h^{3/2})$
by Lemma $2.1(ii)$. If $w(0)=0$, Lemma $2.1(iii)$ gives
\begin{align*}
b_{1} = -(w\ve,\phi_{1}') & = \tfrac{-v''(0)}{2h}\int_{0}^{h}w(x)x(h-x)dx + O(h^{3}) \\
& = \tfrac{-v''(0)}{2}\int_{0}^{h}w'(\xi_{x})x^{2}(h-x)dx + O(h^{3}) = O(h^{3})\,.
\end{align*}
Similarly, if $w(1)=0$, then $b_{N+1}=O(h^{3})$. Hence, $|b|=O(h^{5/2})$, giving $\|\psi\|=O(h^{2})$
by Lemma $2.1(ii)$.
\end{proof}
In the uniform mesh case at hand we consider again the semidiscretizations $(\ref{eq27})$ and
$(\ref{eq29})$ of $(\ref{cb})$ and $(\ref{scb})$ respectively, with the initial conditions for
both systems now given by the interpolants of $\eta_{0}$, $u_{0}$:
\begin{equation}
\eta_{h}(0) = I_{h}\eta_{0}\,, \quad u_{h}(0) = I_{h,0}u_{0}\,.
\label{eq220}
\end{equation}
The main result of this paragraph is
\begin{theorem} Let $h=1/N$ be sufficiently small. Suppose that the solutions of
$(\ref{cb})$ and $(\ref{scb})$ are such that $\eta \in C(0,T;C^{3})$, $\eta_{t} \in C(0,T;C^{2})$,
$u$, $u_{t} \in C(0,T;C_{0}^{3})$. Then, the semidiscrete problems $(\ref{eq27})$, $(\ref{eq220})$
and $(\ref{eq29})$, $(\ref{eq220})$ have unique solutions $(\eta_{h},u_{h})$ for $0\leq t\leq T$
that satisfy
\begin{align*}
(i) \quad \max_{0\leq t\leq T}\|\eta(t)-\eta_{h}(t)\| & \leq Ch^{3/2}\,, \quad
\max_{0\leq t\leq T}\|u(t)-u_{h}(t)\|_{1}\leq Ch\,, \\
(ii) \quad \max_{0\leq t\leq T}\|u(t)-u_{h}(t)\| & \leq Ch^{2}\,, \quad \,\,\,
\max_{0\leq t\leq T}\|u_{t}(t)-u_{ht}(t)\|\leq Ch^{2}\,.
\end{align*}
\end{theorem}
\begin{proof}
We give the proof in detail in the case of $(\ref{scb})$, where existence of solutions of
the i.v.p. $(\ref{eq29})$, $(\ref{eq220})$ for $0\leq t\leq T$ follows from $(\ref{eq213})$.
The proof in the case of $(\ref{cb})$ folows from an argument analogous to that given in the
proof of Proposition $2.1$ and will be omitted. Let
\[
\rho := \eta - I_{h}\eta\,, \,\,\, \theta:=I_{h}\eta - \eta_{h}\,, \,\,\,
\sigma:=u-I_{h,0}u\,, \,\,\, \xi:=I_{h,0}u - u_{h}\,.
\]
Note that
\[
\eta u - \eta_{h}u_{h} = \eta\sigma + u\theta - \theta\xi + F\,,
\]
where
\[
F:=\eta\xi + u\rho -\rho\sigma -\rho\xi - \theta\sigma\,.
\]
In addition
\[
\eta\eta_{x} - \eta_{h}\eta_{hx} = -\theta\theta_{x} + (\eta\theta)_{x} + G_{x}\,,
\]
where
\[
G:=\eta\rho - \rho\theta - \tfrac{1}{2}\rho^{2}\,,
\]
and
\[
uu_{x} - u_{h}u_{hx} = H_{x}\,,
\]
where
\[
H:=u\sigma + u\xi - \sigma\xi - \tfrac{1}{2}\sigma^{2} - \tfrac{1}{2}\xi^{2}\,.
\]
Then, from $(\ref{scb})$ and $(\ref{eq29})$, $(\ref{eq220})$ it follows for $0\leq t\leq T$ that
\begin{equation}
\begin{aligned}
(\theta_{t},\phi) + (\xi_{x},\phi) + \bigl( ((1 + \tfrac{1}{2}\eta & )\sigma)_{x},\phi\bigr)
+ \tfrac{1}{2}((u\theta)_{x},\phi) - \tfrac{1}{2}((\theta\xi)_{x},\phi) \\
& + \tfrac{1}{2}(F_{x},\phi) = - (\rho_{t},\phi) \quad \forall \phi \in S_{h}^{2}\,,
\end{aligned}
\label{eq221}
\end{equation}
\begin{equation}
\begin{aligned}
a(\xi_{t},\chi) + (\theta_{x},\chi) & + (\rho_{x},\chi) - \tfrac{1}{2}(\theta\theta_{x},\chi)
+ \tfrac{1}{2}((\eta\theta)_{x},\chi) + \tfrac{1}{2}(G_{x},\chi) \\
& + \tfrac{3}{2}(H_{x},\chi) = -(\sigma_{t},\chi) \quad \forall \chi \in S_{h,0}^{2}\,,
\end{aligned}
\label{eq222}
\end{equation}
with
\begin{equation}
\theta(0) = 0\,, \quad \xi(0) = 0\,.
\label{eq223}
\end{equation}
(In the right-hand side of $(\ref{eq222})$ we used the fact that for $\chi\in S_{h,0}^{2}$
$a(\sigma_{t},\chi) = (\sigma_{t},\chi)$, since $(v'-(I_{h,0}v)',\chi')=0$ for
$v \in H_{0}^{1}$.) In order to show the estimates in $(i)$, we put $\phi = \theta$ and $\chi = \xi$
in $(\ref{eq221})$ and $(\ref{eq222})$, integrate by parts, and add the resulting equations to get
for $0\leq t\leq T$
\begin{equation}
\begin{aligned}
\tfrac{1}{2}\tfrac{d}{dt}(\|\theta\|^{2} & + \|\xi\|_{1}^{2})
+ \bigl( ((1 + \tfrac{1}{2}\eta)\sigma)_{x},\theta\bigr) + \tfrac{1}{2}((u\theta)_{x},\theta)
+ \tfrac{1}{2}(F_{x},\theta) \\
& + (\rho_{x},\xi) + \tfrac{1}{2}((\eta\theta)_{x},\xi) + \tfrac{1}{2}(G_{x},\xi)
+ \tfrac{3}{2}(H_{x},\xi) = -(\rho_{t},\theta) - (\sigma_{t},\xi)\,.
\end{aligned}
\label{eq224}
\end{equation}
Using now the approximation properties of $S_{h}^{2}$ and $S_{h,0}^{2}$, integration by parts,
and Lemmas $2.1$ and $2.2$ we see that
\[
\Bigl| \bigl( ((1+\tfrac{1}{2}\eta)\sigma)_{x},\theta\bigr)\Bigr| =
\Bigl| \bigl( P[((1+\tfrac{1}{2}\eta)\sigma)_{x}],\theta\bigr) \Bigr| \leq Ch^{3/2}\|\theta\|\,,
\]
\[
|((u\theta)_{x},\theta)| = \tfrac{1}{2}|(u_{x}\theta,\theta)| \leq C\|\theta\|^{2}\,,
\]
\begin{align*}
|(F_{x},\theta)| & \leq |((\eta\xi)_{x},\theta)| + |((u\rho)_{x},\theta)| + |((\rho\sigma)_{x},\theta)|
+ |((\rho\xi)_{x},\theta)| + \tfrac{1}{2}|(\sigma_{x}\theta,\theta)| \\
& \leq C\bigl( \|\xi\|_{1}\|\theta\| + h^{2}\|\theta\| + h^{3}\|\theta\| + h\|\xi\|_{1}\|\theta\|
+ h\|\theta\|^{2}\bigr) \\
& \leq  C\bigl( \|\xi\|_{1}\|\theta\| + h^{2}\|\theta\| + h \|\theta\|^{2}\bigr)\,,
\end{align*}
\[
|(\rho_{x},\xi)| \leq C h^{2}\|\xi\|_{1}\,,
\]
\[
|((\eta\theta)_{x},\xi)| \leq C \|\xi\|_{1}\|\theta\|\,,
\]
\begin{align*}
|(G_{x},\xi)| & \leq |(\eta\rho,\xi_{x})| + |(\rho\theta,\xi_{x})| + \tfrac{1}{2}|(\rho^{2},\xi_{x})|\\
& \leq C\bigl( h^{2}\|\xi\|_{1} + h^{2}\|\theta\|\|\xi\|_{1} + h^{4}\|\xi\|_{1}\bigr)\\
& \leq C \bigl( h^{2}\|\xi\|_{1} + h^{2}\|\theta\|\|\xi\|_{1}\bigr)\,,
\end{align*}
\begin{align*}
|(H_{x},\xi)| & \leq |(u\sigma,\xi_{x})| + \tfrac{1}{2}|(u_{x}\xi,\xi)| + |(\sigma\xi,\xi_{x})|
+ \tfrac{1}{2}|(\sigma^{2},\xi_{x})| \\
& \leq C\bigl( h^{2}\|\xi\|_{1} + \|\xi\|^{2} + h^{2}\|\xi\|_{1}^{2} + h^{4}\|\xi\|_{1}\bigr)\\
& \leq C \bigl( h^{2}\|\xi\|_{1} + \|\xi\|_{1}^{2}\bigr)\,,
\end{align*}
\[
|(\rho_{t},\theta)| + |(\sigma_{t},\xi)| \leq C \bigl( h^{2}\|\theta\| + h^{2}\|\xi\|\bigr)\,.
\]
Hence, from $(\ref{eq224})$ we conclude that for $0\leq t\leq T$
\begin{align*}
\tfrac{d}{dt}(\|\theta\|^{2} + \|\xi\|_{1}^{2}) & \leq C\bigl( h^{3/2}\|\theta\| + h^{2}\|\xi\|_{1}\bigr)\\
& \,\,\,\,\,\, + C\bigl( \|\theta\|^{2} + \|\xi\|_{1}^{2}\bigr)
\leq C\bigl[ h^{3} + (\|\theta\|^{2} + \|\xi\|_{1}^{2})\bigr]\,.
\end{align*}
Hence, by Gronwall's Lemma we obtain
\[
\|\theta(t)\|^{2} + \|\xi(t)\|_{1}^{2} \leq C[h^{3} + \|\theta(0)\|^{2} + \|\xi(0)\|_{1}^{2}]\,,
\]
from which, in view of $(\ref{eq220})$, we get
\begin{equation}
\|\theta\| + \|\xi\|_{1} \leq C h^{3/2}\,, \quad 0\leq t\leq T\,,
\label{eq225}
\end{equation}
and the estimates in $(i)$ follow. In addition, from $(\ref{eq225})$ and the approximation
properties of $S_{h}^{2}$ and $S_{h,0}^{2}$ one may easily derive the following $L^{2}$
estimates of $F$, $G$, and $H$, that we note for further reference
\begin{equation}
\begin{aligned}
\|F\| & \leq C(\|\xi\| + h^{2})\,, \\
\|G\| & \leq C h^{2}\,, \\
\|H\| & \leq C (\|\xi\| + h^{2})\,.
\end{aligned}
\label{eq226}
\end{equation}
\par
We proceed now to prove the optimal-order error estimates in $(ii)$. Equation $(\ref{eq222})$
may be written in the form
\begin{equation}
\xi_{t} = R_{h}v\,,
\label{eq227}
\end{equation}
where $v$ is the solution of the problem
\begin{equation}
\begin{aligned}
v - \tfrac{1}{3}v'' & = -(\theta+\rho)_{x} - \tfrac{1}{2}(\eta\theta)_{x}
-\tfrac{1}{2}(G - \tfrac{1}{2}\theta^{2}+3H)_{x} - \sigma_{t}\,, \quad x \in [0,1]\,, \\
v(0) & = v(1) = 0\,.
\end{aligned}
\label{eq228}
\end{equation}
Considering the weak form of $(\ref{eq228})$ in $H_{0}^{1}$, and using integration by parts
in the right-hand side, $(\ref{eq225})$ and $(\ref{eq226})$ we see that
\begin{equation}
\|v\|_{1} \leq C h^{3/2}\,.
\label{eq229}
\end{equation}
In order to derive a bound for $\|v\|$, let $\zeta \in L^{2}$ and $V$ be the solution of the
problem
\begin{equation}
\begin{aligned}
V - \tfrac{1}{3}V'' & = \zeta\,, \quad x \in [0,1]\,, \\
V(0) & = V(1) = 0\,.
\end{aligned}
\label{eq230}
\end{equation}
Then, by $(\ref{eq228})$
\begin{equation}
(v,\zeta) = a(v,V)  = (\theta+\rho,V') + \tfrac{1}{2}(\eta\theta,V')
+ \tfrac{1}{2}(G-\tfrac{1}{2}\theta^{2}+3H,V')-(\sigma_{t},V)\,.
\label{eq231}
\end{equation}
From $(\ref{eq221})$ with $\phi=1$ we see that $(\theta_{t}+\rho_{t},1)=0$, $0\leq t\leq T$\,.
Hence,
\[
\int_{0}^{1}(\theta + \rho)dx = \int_{0}^{1}\rho(x,0)dx =:J=\text{const}\,.
\]
Therefore, if for $(x,t) \in [0,1]\times [0,T]$
\begin{equation}
\gamma(x,t):= \int_{0}^{x}\bigl( \theta(s,t) + \rho(s,t)\bigr) ds - xJ\,,
\label{eq232}
\end{equation}
it follows that $\gamma \in H_{0}^{1}$ and $\gamma_{x} = \theta + \rho - J$. Hence,
$(\ref{eq231})$ yields
\begin{align*}
(v,\zeta) & = (\gamma_{x},V') + \tfrac{1}{2}(\eta\gamma_{x},V')
+ \tfrac{1}{2}(G-\tfrac{1}{2}\theta^{2} + 3H - \eta\rho+\eta J,V') - (\sigma_{t},V)\\
& = -(\gamma,V'') - \tfrac{1}{2}(\gamma,\eta V''+\eta_{x}V')
+ \tfrac{1}{2}(G-\tfrac{1}{2}\theta^{2} + 3H - \eta\rho + \eta J,V') - (\sigma_{t},V)\,.
\end{align*}
Now, using $(\ref{eq225})$, $(\ref{eq226})$, the approximation and inverse properties
of $S_{h}^{2}$, $S_{h,0}^{2}$, and elliptic regularity in $(\ref{eq230})$ we obtain
\begin{align*}
|(v,\zeta)| & \leq \|\gamma\|\|V''\| + C\|\gamma\|(\|V''\| + \|V'\|)
+ C(h^{2} + \|\xi\|)\|V'\| + Ch^{2}\|V\| \\
& \leq C(h^{2} + \|\gamma\| + \|\xi\|)\|\zeta\|\,,
\end{align*}
and conclude that
\begin{equation}
\|v\| \leq C(h^{2} + \|\gamma\| + \|\xi\|)\,.
\label{eq233}
\end{equation}
Let now $W$ the solution of the problem
\begin{equation}
\begin{aligned}
W - \tfrac{1}{3}W'' & = \xi\,, \quad x \in [0,1]\,, \\
W(0) = W(1) & = 0\,.
\end{aligned}
\label{eq234}
\end{equation}
Using $(\ref{eq227})$, $(\ref{eq215})$, elliptic regularity in $(\ref{eq234})$, and the estimates
$(\ref{eq229})$ and $(\ref{eq233})$ gives
\begin{align*}
(\xi,\xi_{t}) & = a(W,\xi_{t}) = a(W,R_{h}v) = a(R_{h}W,R_{h}v) \\
& = a(R_{h}v,R_{h}W) = a(v,R_{h}W) = a(v,R_{h}W-W) + a(W,v)\\
& = a(v,R_{h}W-W) + (\xi,v) \leq Ch\|\xi\|\|v\|_{1} + \|\xi\|\|v\|\\
& \leq C(h^{2} + \|\gamma\| + \|\xi\|)\|\xi\|\,,
\end{align*}
from which there follows that
\begin{equation}
\tfrac{1}{2}\tfrac{d}{dt}\|\xi\|^{2} \leq C(h^{4} + \|\gamma\|^{2} + \|\xi\|^{2})\,,
\quad 0\leq t\leq T\,.
\label{eq235}
\end{equation}
In order to obtain the required optimal-order estimate for $\|\xi\|$ from $(\ref{eq235})$ we need
a similar estimate for a suitable approximation of $\gamma$. For this purpose, observe that
$(\ref{eq221})$ yields, for $0\leq t\leq T$ and $\phi \in S_{h}^{2}$,
\begin{equation}
(\gamma_{xt},\phi) + \tfrac{1}{2}((u\theta)_{x},\phi) = - (w_{x},\phi)\,,
\label{eq236}
\end{equation}
where $w=\xi + (1 + \tfrac{1}{2}\eta)\sigma - \tfrac{1}{2}\theta\xi + \tfrac{1}{2}F$; note  that
$w \in H_{0}^{1}$ and that $(\ref{eq225})$ and $(\ref{eq226})$ give $\|w\| \leq C(\|\xi\|+h^{2})$.
Using integration by parts and the definition of $\gamma$ in $(\ref{eq236})$ yields for
$0\leq t\leq T$
\begin{equation}
(\gamma_{t},\phi') + \tfrac{1}{2}(u\gamma_{x},\phi') =
-(w - \tfrac{1}{2}u\rho+\tfrac{1}{2}u J,\phi')\quad \forall \phi \in S_{h}^{2}\,.
\label{eq237}
\end{equation}
Consider now the space $S_{h}^{-1}$ of discontinuous, piecewise constant functions relative to
the partition $\{x_{j}\}$. Given any $\psi \in S_{h}^{-1}$ consider in $(\ref{eq237})$
$\phi \in S_{h}^{2}$ such that $\phi' = \psi$. Hence we have for $0\leq t\leq T$
\[
(\gamma_{t},\psi) + \tfrac{1}{2}(u\gamma_{x},\psi) = -(K,\psi)\quad \forall \psi \in S_{h}^{-1}\,,
\]
where $K:= w - \tfrac{1}{2}u\rho + \tfrac{1}{2}u J$ satisfies
$\|K\| \leq C(\|\xi\| + h^{2})$. Taking now $\psi = P_{0}\gamma$ in the above, where
$P_{0}$ is the $L^{2}$-projection operator onto $S_{h}^{-1}$, yields for $0\leq t\leq T$
\begin{equation}
\tfrac{1}{2}\tfrac{d}{dt}\|P_{0}\gamma\|^{2} + \tfrac{1}{2}(u\gamma_{x},P_{0}\gamma)
= - (K,P_{0}\gamma)\,.
\label{eq238}
\end{equation}
Since $\|P_{0}\gamma - \gamma\| \leq Ch\|\gamma\|_{1}$, cf. e.g. $(16.24)$ of \cite{t}, and
$\|\gamma\|_{1}\leq Ch^{3/2}$ by $(\ref{eq225})$, we have
\begin{align*}
|(u\gamma_{x},P_{0}\gamma)| & = |(u\gamma_{x},P_{0}\gamma-\gamma)-\tfrac{1}{2}(u_{x}\gamma,\gamma)|
\leq C(h\|\gamma\|_{1}^{2} + \|\gamma\|^{2}) \\
&\leq C(h\|\gamma\|_{1}^{2} + \|P_{0}\gamma-\gamma\|^{2} + \|P_{0}\gamma\|^{2})\\
& \leq C(h^{4} + \|P_{0}\gamma\|^{2})\,.
\end{align*}
Hence, $(\ref{eq238})$ yields, for $0\leq t\leq T$, that
\begin{equation}
\tfrac{1}{2}\tfrac{d}{dt}\|P_{0}\gamma\|^{2} \leq C(h^{4} + \|P_{0}\gamma\|^{2} + \|\xi\|^{2})\,.
\label{eq239}
\end{equation}
Now, from $(\ref{eq235})$, since $\|\gamma\|\leq \|P_{0}\gamma-\gamma\| + \|P_{0}\gamma\|
\leq Ch^{5/2} + \|P_{0}\gamma\|$, we get for $0\leq t\leq T$
\begin{equation}
\tfrac{1}{2}\tfrac{d}{dt}\|\xi\|^{2} \leq C(h^{4} + \|P_{0}\gamma\|^{2} + \|\xi\|^{2})\,.
\label{eq240}
\end{equation}
Adding $(\ref{eq239})$ and $(\ref{eq240})$ we finally obtain by Gronwall's Lemma and $(\ref{eq223})$
that
\begin{equation}
\|P_{0}\gamma\|^{2} + \|\xi\|^{2} \leq Ch^{4}\,, \quad 0\leq t\leq T\,.
\label{eq241}
\end{equation}
Therefore, the first inequality of the conclusion $(ii)$ of the theorem holds;
in addition, by similar estimates with the ones used above, we also obtain
\begin{equation}
\|\gamma\| \leq Ch^{2}\,, \quad 0\leq t\leq T\,.
\label{eq242}
\end{equation}
Finally, we prove the second estimate of $(ii)$. Let $Z$ be the solution of the problem
\begin{align*}
Z - \tfrac{1}{3}Z'' & = \xi_{t}\,, \quad x \in [0,1]\,, \\
Z(0) = Z(1) & = 0\,.
\end{align*}
Then, by $(\ref{eq227})$
\[
\|\xi_{t}\|^{2} = a(Z,\xi_{t}) = a(Z,R_{h}v) = a(R_{h}Z,v) = a(R_{h}Z-Z,v) + (\xi_{t},v)\,.
\]
Hence, elliptic regularity and $(\ref{eq229})$, $(\ref{eq233})$, $(\ref{eq241})$, $(\ref{eq242})$
give
\[
\|\xi_{t}\|^{2}\leq Ch\|Z\|_{2}\|v\|_{1} + \|\xi_{t}\|\|v\| \leq C h^{2}\|\xi_{t}\|\,,
\]
i.e. $\|\xi_{t}\| \leq Ch^{2}$, and the second estimate of $(ii)$ follows.
\end{proof}
\begin{remark}
It is not hard to see that the conclusions of Theorem $2.1$ hold if we take $\eta_{h}(0)$ as any
approximation of $\eta_{0}$ in $S_{h}^{2}$ with optimal-order $L^{2}$ rate of convergence. However,
$u_{h}(0)$ has to be taken as $I_{h,0}u_{0}$ or $R_{h}u_{0}$.
\end{remark}
\begin{remark}
The superaccuracy estimate $\|\xi\|_{1}=O(h^{3/2})$ of $(\ref{eq225})$ combined with
$(\ref{eq23})$ and Sobolev's inequality yields the $L^{\infty}$ estimate
$\|u-u_{h}\|_{\infty} = O(h^{3/2})$ for $u$.
\end{remark}
\begin{remark}
In the case of a {\em quasiuniform} mesh with $h = \max_{i}(x_{i+1}-x_{i})$,
one may easily check
that the analog of Lemma $2.1$ still holds; however in the proof of Lemma $2.2$ there is no
cancellation anymore from adjacent intervals and the conclusion is just that
$\|\psi\|\leq Ch$. As a consequence, the techniques of the proof of Theorem $2.1$ now yield
that $\|\eta - \eta_{h}\| = O(h)$ and $\|u - u_{h}\|_{1} = O(h)$, and, instead of the optimal-order
estimates in $(ii)$, that $\|u-u_{h}\|= O(h^{3/2})$ and $\|u_{t} - u_{ht}\|=O(h^{3/2})$. However,
for the {\em linearized} problem
\begin{equation}
\begin{aligned}
\begin{aligned}
\eta_{t} & + u_{x} = 0\,, \\
u_{t} & + \eta_{x} - \tfrac{1}{3}u_{xxt} = 0\,,
\end{aligned}
&\,\,\,\,\,\,\,\,\,\,  \quad (x,t) \in [0,1]\times [0,T]\,, \\
u(0,t) = 0\,, \,\,\quad u(1,t) & = 0\,, \quad t \in [0,T]\,, \\
\eta(x,0) = \eta_{0}(x)\,, \quad u(&x,0) = u_{0}(x)\,, \quad x\in [0,1]\,,
\end{aligned}
\label{eq243}
\end{equation}
the last two estimates may be improved to yield optimal order, i.e. to give $\|u-u_{h}\|=O(h^{2})$,
$\|u_{t}-u_{ht}\|=O(h^{2})$.
Our numerical experiments suggest that $\|u-u_{h}\|= O(h^{2})$ even in the nonlinear case,
but we have not been able to prove this thusfar.
\end{remark}
\subsection{Numerical experiments}
We first consider the case of approximations on a uniform mesh with $h=1/N$ on $[0,1]$.
Table \ref{tbl21} shows the errors and the associated rates of convergence,
\def\baselinestretch{1}
\captionsetup[subtable]{labelformat=empty,position=top,margin=1pt,singlelinecheck=false}
\scriptsize
\begin{table}[h]
\subfloat[$L^{2}$-errors]{
\begin{tabular}[t]{ | c | c | c | c | c | }\hline
$N$   &    $\eta$      &  $order$  &      $u$       &  $order$   \\ \hline
$40$  &  $0.1894(-1)$  &           &  $0.1749(-3)$  &            \\ \hline
$80$  &  $0.6849(-2)$  &  $1.467$  &  $0.4259(-4)$  &  $2.038$   \\ \hline
$120$ &  $0.3761(-2)$  &  $1.478$  &  $0.1877(-4)$  &  $2.021$   \\ \hline
$160$ &  $0.2454(-2)$  &  $1.484$  &  $0.1051(-4)$  &  $2.015$   \\ \hline
$200$ &  $0.1761(-2)$  &  $1.487$  &  $0.6710(-5)$  &  $2.011$   \\ \hline
$240$ &  $0.1342(-2)$  &  $1.490$  &  $0.4652(-5)$	&  $2.009$   \\ \hline
$280$ &  $0.1066(-2)$  &  $1.491$  &  $0.3413(-5)$	&  $2.008$   \\ \hline
$320$ &  $0.8738(-3)$  &  $1.492$  &  $0.2611(-5)$  &  $2.007$   \\ \hline
$360$ &  $0.7328(-3)$  &  $1.493$  &  $0.2062(-5)$	&  $2.006$   \\ \hline
$400$ &  $0.6261(-3)$  &  $1.494$  &  $0.1669(-5)$  &  $2.005$   \\ \hline
$440$ &  $0.5430(-3)$  &  $1.495$  &  $0.1379(-5)$  &  $2.005$   \\ \hline
$480$ &  $0.4767(-3)$  &  $1.495$  &  $0.1158(-5)$  &  $2.004$   \\ \hline
$520$ &  $0.4230(-3)$  &  $1.496$  &  $0.9864(-6)$  &  $2.004$   \\ \hline
\end{tabular}
} \qquad
\subfloat[$L^{\infty}$-errors]{
\begin{tabular}[t]{ | c | c | c | c | c | }\hline
$N$   &    $\eta$      &  $order$  &      $u$       & $order$    \\ \hline
$40$  &  $0.1126$      &           &  $0.5916(-3)$  &            \\ \hline
$80$  &  $0.5617(-1)$  &  $1.004$  &  $0.1509(-3)$  & $1.971$    \\ \hline
$120$ &  $0.3741(-1)$  &  $1.002$  &  $0.6755(-4)$  & $1.983$    \\ \hline
$160$ &  $0.2804(-1)$  &  $1.002$  &  $0.3813(-4)$  & $1.988$    \\ \hline
$200$ &  $0.2243(-1)$  &  $1.001$  &  $0.2445(-4)$  & $1.991$    \\ \hline
$240$ &  $0.1868(-1)$  &  $1.001$  &  $0.1701(-4)$  & $1.992$    \\ \hline
$280$ &  $0.1601(-1)$  &  $1.001$  &  $0.1251(-4)$  & $1.994$    \\ \hline
$320$ &  $0.1401(-1)$  &  $1.001$  &  $0.9582(-5)$  & $1.994$    \\ \hline
$360$ &  $0.1245(-1)$  &  $1.001$  &  $0.7575(-5)$  & $1.995$    \\ \hline
$400$ &  $0.1121(-1)$  &  $1.001$  &  $0.6139(-5)$  & $1.996$    \\ \hline
$440$ &  $0.1019(-1)$  &  $1.001$  &  $0.5075(-5)$  & $1.996$    \\ \hline
$480$ &  $0.9337(-2)$  &  $1.001$  &  $0.4266(-5)$  & $1.996$    \\ \hline
$520$ &  $0.8618(-2)$  &  $1.001$  &  $0.3636(-5)$  & $1.997$    \\ \hline
\end{tabular}
}
\\
\subfloat[$H^{1}$-errors]{
\begin{tabular}[t]{ | c | c | c | c | c | }\hline
$N$   &    $\eta$      &  $order$  &      $u$       & $order$    \\ \hline
$40$  &  $0.2440(+1)$  &           &  $0.2449(-1)$  &            \\ \hline
$80$  &  $0.1776(+1)$  &  $0.459$  &  $0.1192(-1)$  & $1.039$    \\ \hline
$120$ &  $0.1466(+1)$  &  $0.473$  &  $0.7875(-2)$  & $1.022$    \\ \hline
$160$ &  $0.1277(+1)$  &  $0.480$  &  $0.5880(-2)$  & $1.016$    \\ \hline
$200$ &  $0.1146(+1)$  &  $0.484$  &  $0.4691(-2)$	& $1.012$    \\ \hline
$240$ &  $0.1049(+1)$  &  $0.487$  &  $0.3902(-2)$	& $1.010$    \\ \hline
$280$ &  $0.9725$	   &  $0.489$  &  $0.3341(-2)$  & $1.008$    \\ \hline
$320$ &  $0.9109$	   &  $0.490$  &  $0.2920(-2)$	& $1.007$    \\ \hline
$360$ &  $0.8596$	   &  $0.492$  &  $0.2594(-2)$	& $1.006$    \\ \hline
$400$ &  $0.8161$	   &  $0.493$  &  $0.2333(-2)$  & $1.006$    \\ \hline
$440$ &	 $0.7787$	   &  $0.493$  &  $0.2120(-2)$	& $1.005$    \\ \hline
$480$ &  $0.7459$      &  $0.494$  &  $0.1942(-2)$  & $1.005$    \\ \hline
$520$ &  $0.7170$      &  $0.494$  &  $0.1792(-2)$  & $1.004$    \\ \hline
\end{tabular}
}
\normalsize
\caption{
Errors and orders of convergence. $(\ref{cb})$ system, standard Galerkin semidiscretization
with piecewise linear, continuous functions on a uniform mesh.}
\label{tbl21}
\end{table}
\def\baselinestretch{1}
\captionsetup[subtable]{labelformat=empty,position=top,margin=1pt,singlelinecheck=false}
\scriptsize
\begin{table}[t]
\subfloat[$L^{2}$-errors]{
\begin{tabular}[b]{ | c | c | c | c | c | }\hline
$N$   &    $\eta$      &  $order$  &      $u$       &  $order$   \\ \hline
$40$  &  $0.7423(-2)$  &           &  $0.3613(-3)$  &            \\ \hline
$80$  &  $0.2678(-2)$  &  $1.471$  &  $0.8849(-4)$  &  $2.030$   \\ \hline
$120$ &	 $0.1469(-2)$  &  $1.481$  &  $0.3907(-4)$  &  $2.016$   \\ \hline
$160$ &  $0.9579(-3)$  &  $1.486$  &  $0.2190(-4)$  &  $2.011$   \\ \hline
$200$ &  $0.6871(-3)$  &  $1.489$  &  $0.1399(-4)$  &  $2.009$   \\ \hline
$240$ &  $0.5236(-3)$  &  $1.491$  &  $0.9703(-5)$  &  $2.007$   \\ \hline
$280$ &	 $0.4160(-3)$  &  $1.492$  &  $0.7122(-5)$	&  $2.006$   \\ \hline
$320$ &  $0.3408(-3)$  &  $1.493$  &  $0.5449(-5)$	&  $2.005$   \\ \hline
$360$ &	 $0.2858(-3)$  &  $1.494$  &  $0.4303(-5)$	&  $2.005$   \\ \hline
$400$ &	 $0.2441(-3)$  &  $1.495$  &  $0.3484(-5)$	&  $2.004$   \\ \hline
$440$ &	 $0.2117(-3)$  &  $1.495$  &  $0.2878(-5)$	&  $2.004$   \\ \hline
$480$ &	 $0.1859(-3)$  &  $1.496$  &  $0.2418(-5)$	&  $2.003$   \\ \hline
$520$ &	 $0.1649(-3)$  &  $1.496$  &  $0.2060(-5)$	&  $2.003$   \\ \hline
\end{tabular}
}\qquad
\subfloat[$L^{\infty}$-errors]{
\begin{tabular}[b]{ | c | c | c | c | c | }\hline
$N$   &    $\eta$      &  $order$  &      $u$       &  $order$   \\ \hline
$40$  &  $0.3491(-1)$  &           &  $0.9655(-3)$  &            \\ \hline
$80$  &	 $0.1704(-1)$  &  $1.034$  &  $0.2537(-3)$  &  $1.928$   \\ \hline
$120$ &	 $0.1125(-1)$  &  $1.024$  &  $0.1140(-3)$	&  $1.972$   \\ \hline
$160$ &	 $0.8394(-2)$  &  $1.018$  &  $0.6442(-4)$	&  $1.985$   \\ \hline
$200$ &	 $0.6694(-2)$  &  $1.014$  &  $0.4132(-4)$	&  $1.991$   \\ \hline
$240$ &	 $0.5567(-2)$  &  $1.012$  &  $0.2873(-4)$	&  $1.994$   \\ \hline
$280$ &	 $0.4764(-2)$  &  $1.010$  &  $0.2112(-4)$	&  $1.995$   \\ \hline
$320$ &	 $0.4164(-2)$  &  $1.008$  &  $0.1618(-4)$	&  $1.996$   \\ \hline
$360$ &	 $0.3698(-2)$  &  $1.008$  &  $0.1279(-4)$	&  $1.997$   \\ \hline
$400$ &	 $0.3326(-2)$  &  $1.007$  &  $0.1036(-4)$	&  $1.998$   \\ \hline
$440$ &	 $0.3021(-2)$  &  $1.006$  &  $0.8563(-5)$	&  $1.998$   \\ \hline
$480$ &	 $0.2768(-2)$  &  $1.006$  &  $0.7197(-5)$	&  $1.998$   \\ \hline
$520$ &	 $0.2554(-2)$  &  $1.005$  &  $0.6133(-5)$	&  $1.999$   \\ \hline
\end{tabular}
}\\
\subfloat[$H^{1}$-errors]{
\begin{tabular}[b]{ | c | c | c | c | c | }\hline
$N$   &    $\eta$      &  $order$  &      $u$       & $order$    \\ \hline
$40$  &  $0.1049(+1)$  &           &  $0.2670(-1)$  &            \\ \hline
$80$  &	 $0.7383$	   &  $0.506$  &  $0.1249(-1)$	& $1.096$    \\ \hline
$120$ &	 $0.6026$	   &  $0.501$  &  $0.8131(-2)$	& $1.059$    \\ \hline
$160$ &	 $0.5219$	   &  $0.500$  &  $0.6024(-2)$	& $1.042$    \\ \hline
$200$ &	 $0.4668$	   &  $0.500$  &  $0.4784(-2)$	& $1.033$    \\ \hline
$240$ &	 $0.4262$	   &  $0.499$  &  $0.3967(-2)$	& $1.027$    \\ \hline
$280$ &	 $0.3946$	   &  $0.499$  &  $0.3388(-2)$	& $1.023$    \\ \hline
$320$ &	 $0.3692$	   &  $0.499$  &  $0.2956(-2)$	& $1.020$    \\ \hline
$360$ &	 $0.3481$	   &  $0.500$  &  $0.2622(-2)$	& $1.018$    \\ \hline
$400$ &	 $0.3302$	   &  $0.500$  &  $0.2356(-2)$	& $1.016$    \\ \hline
$440$ &	 $0.3149$	   &  $0.500$  &  $0.2139(-2)$	& $1.015$    \\ \hline
$480$ &	 $0.3015$	   &  $0.500$  &  $0.1959(-2)$	& $1.013$    \\ \hline
\end{tabular}
}
\normalsize
\caption{
Errors and orders of convergence. $(\ref{scb})$ system, standard Galerkin semidiscretization
with piecewise linear, continuous functions on a uniform mesh.}
\label{tbl22}
\end{table}
\normalsize
in the $L^{2}$-, $L^{\infty}$-
and $H^{1}$-norms at $T =1$, of the standard Galerkin approximation (using piecewise linear,
continuous functions) to the $(\ref{cb})$ system with suitable right-hand so that its exact
solution is given by $\eta = \exp(2t)(\cos(\pi x) + x + 2)$, $u=\exp(-xt)x\sin(\pi x)$. The system
was integrated up to $T=1$ with the classical, four-stage, fourth-order explicit Runge-Kutta
method (see Section $4$) using a time step $k=h/10$. We checked that the temporal error of the
discretization was very small compared to the spatial error, so that the errors and rates of
convergence shown are essentially those of the semidiscrete problem $(\ref{eq27})$,
$(\ref{eq220})$. The table suggests that
the $L^{2}$ rates of convergence of $\eta$ and $u$ approach $3/2$ and $2$, respectively,
thus confirming the relevant estimates of Theorem $2.1$. It also suggests that
$\|u-u_{h}\|_{1} = O(h)$ (confirming the result of Theorem $2.1$), and
$\|\eta - \eta_{h}\|_{1}=O(h^{1/2})$, $\|\eta - \eta_{h}\|_{\infty}=O(h)$,
$\|u-u_{h}\|_{\infty}=O(h^{2})$.
Table \ref{tbl22} shows the analogous errors and rates in the case of a
nonhomogeneous $(\ref{scb})$ system with the same exact solution
$\eta = \exp(2t)(\cos(\pi x) + x + 2)$, $u=\exp(-xt)x\sin(\pi x)$, integrated on a uniform spatial mesh
up to $T=1$ with the same time stepping procedure.
The convergence rates are essentially the same.
Table \ref{tbl23}(a) shows some results from a computation of the exact solution
$\eta = \exp(2t)(\cos(\pi x) + x + 2)$, $u=\exp(xt)(\sin(\pi x) + x^{3} - x^{2})$ of a suitably
nonhomogeneous version
of the $(\ref{cb})$ system with the standard Galerkin method with piecewise
linear, continuous functions on the quasiuniform mesh on $[0,1]$ given by $h_{2i-1} = 1.2\Delta x$,
$h_{2i} = 0.8\Delta x$, $1\leq i\leq N/2$, where $h_{i}=x_{i+1}-x_{i}$ and
\normalsize
$\Delta x=1/N$. Again
the fully discrete problem was solved by the fourth-order accurate four-stage classical explicit
Runge-kutta scheme with timestep $k=\Delta x/10$. We integrated the system up to $T=0.4$ starting with
the $L^{2}$-projections of $\eta_{0}$ and $u_{0}$ on the finite element subspaces. The temporal error
was much smaller than the spatial
\def\baselinestretch{1}
\captionsetup[subtable]{labelformat=empty, position=bottom, singlelinecheck=true}
\scriptsize
\begin{table}[t]
\subfloat[(a)]{
\begin{tabular}[b]{ | c | c | c | c | c | }\hline
$N$   &    $\eta$      &  $order$  &      $u$       &  $order$   \\ \hline
$80$  &  $0.1277(-1)$  &           &  $0.7432(-4)$  &            \\ \hline
$160$ &  $0.6383(-2)$  &  $1.000$  &  $0.1858(-4)$  &  $2.000$   \\ \hline
$240$ &  $0.4258(-2)$  &  $0.999$  &  $0.8259(-5)$  &  $2.000$   \\ \hline
$320$ &  $0.3194(-2)$  &  $0.999$  &  $0.4646(-5)$	&  $2.000$   \\ \hline
$400$ &	 $0.2556(-2)$  &  $0.999$  &  $0.2973(-5)$	&  $2.000$   \\ \hline
$480$ &	 $0.2131(-2)$  &  $0.999$  &  $0.2065(-5)$	&  $2.000$   \\ \hline
$560$ &	 $0.1826(-2)$  &  $0.999$  &  $0.1517(-5)$	&  $2.000$   \\ \hline
$640$ &  $0.1598(-2)$  &  $0.999$  &  $0.1161(-5)$  &  $2.000$   \\ \hline
$720$ &  $0.1421(-2)$  &  $0.999$  &  $0.9177(-5)$  &  $2.000$   \\ \hline
\end{tabular}
}\qquad
\scriptsize
\subfloat[(b)]{
\begin{tabular}[b]{ | c | c | c | c | c | }\hline
$N$   &    $\eta$      &  $order$  &      $u$       &  $order$   \\ \hline
$40$  &  $0.5852(-1)$  &           &  $0.1693(-2)$  &            \\ \hline
$80$  &  $0.2933(-1)$  &  $0.997$  &  $0.4271(-3)$  &  $1.987$   \\ \hline
$120$ &  $0.1942(-1)$  &  $1.017$  &  $0.1899(-3)$  &  $2.000$   \\ \hline
$160$ &  $0.1449(-1)$  &  $1.019$  &  $0.1068(-3)$	&  $2.000$   \\ \hline
$200$ &	 $0.1155(-1)$  &  $1.016$  &  $0.6834(-4)$	&  $2.001$   \\ \hline
$240$ &	 $0.9600(-2)$  &  $1.014$  &  $0.4745(-4)$	&  $2.001$   \\ \hline
$280$ &	 $0.8213(-2)$  &  $1.012$  &  $0.3486(-4)$	&  $2.001$   \\ \hline
$320$ &  $0.7176(-2)$  &  $1.011$  &  $0.2669(-4)$  &  $2.001$   \\ \hline
$360$ &  $0.6371(-2)$  &  $1.010$  &  $0.2109(-4)$  &  $2.000$   \\ \hline
$400$ &  $0.5729(-2)$  &  $1.009$  &  $0.1708(-4)$  &  $2.001$   \\ \hline
\end{tabular}
}
\normalsize
\caption{$L^{2}$-errors and orders of convergence. $(\ref{cb})$ system, standard Galerkin
semidiscretization with piecewise linear, continuous functions on a quasiuniform mesh
with $\tfrac{\max h_{i}}{\min h_{i}}=1.5$ (a), $\tfrac{\max h_{i}}{\min h_{i}}=150$ (b)}
\label{tbl23}
\end{table}
\normalsize
error. Table\ref{tbl23}(a) shows the $L^{2}$-errors for $\eta$ and $u$
and the associated rates of convergence at $T=0.4$. The data strongly suggest that
$\|\eta - \eta_{h}\|=O(h)$ and $\|u-u_{h}\|=O(h^{2})$, thus confirming the relevant theoretical result
for $\eta$ (cf. Proposition $2.1$ and Remark $2.2$), and supporting the conjecture that the
$L^{2}$ rate of convergence for $u$ is actually equal to $2$ even in the case of the nonlinear problem;
recall that the technique of proof of Theorem $2.1$ in the case of a quasiuniform mesh gives a
pessimistic bound of $O(h^{3/2})$ for $\|u-u_{h}\|$, cf. Remark $2.2$. These results are confirmed by the
rates shown in Table \ref{tbl23}(b), which was obtained by integrating the same problem with the same method
on the quasiuniform mesh on $[0,1]$ with $h_{10i-9}=0.02\Delta x$, $h_{10i-8}=0.05\Delta x$,
$h_{10i-7}=0.08\Delta x$, $h_{10i-6}=0.35\Delta x$, $h_{10i-5}=0.5\Delta x$,
$h_{10i-4}=h_{10i-3}=\Delta x$, $h_{10i-2}=h_{10i-1}=2\Delta x$, $h_{10i}=3\Delta x$,
$1\leq i\leq N/10$, and $k=\Delta x/10$.
\normalsize
\section{Standard Galerkin semidiscretization with cubic splines}
\subsection{Semidiscretization on a quasiuniform mesh}
Let $0=x_{1} < x_{2} < \ldots < x_{N+1}=1$ denote a quasiuniform partition of $[0,1]$ with
$h:=max_{i}(x_{i+1} - x_{i})$ and let
\[
S_{h}^{4}:=\{\phi \in C^{2} : \phi\big|_{[x_{j},x_{j+1}]} \in \mathbb{P}_{3}\,, 1\leq j\leq N\}\,, \quad
S_{h,0}^{4}=\{\phi \in S_{h}^{4} : \phi(0)=\phi(1)=0\}\,,
\]
be the space of (the $C^{2}$) cubic splines on $[0,1]$ relative to the partition $\{x_{j}\}$,
and the space of cubic splines that vanish at $x=0$ and at $x=1$. In this section we shall
denote by $I_{h} : C^{1} \to S_{h}^{4}$ the interpolation operator, with the properties that for any
$v \in C^{1}$, $(I_{h}v)(x_{i}) = v(x_{i})$, $1\leq i\leq N+1$, $(I_{h}v)'(x_{k}) = v'(x_{k})$,
$k=1,N+1$, and let $I_{h,0} : C^{1}_{0} \to S_{h,0}^{4}$ be the analogous interpolant onto $S_{h,0}^{4}$.
It is well known that
\begin{equation}
\sum_{j=0}^{2}h^{j}\|w-I_{h}w\|_{j} \leq C h^{k}\|w^{(k)}\|
\label{eq31}
\end{equation}
holds for any $w \in H^{k}$ for $k=2$, $3$, $4$ and that a similar estimate holds for $I_{h,0}w$ if
$w\in H^{k}\cap H_{0}^{1}$. More generally, \cite{schr}, if $1\leq k\leq 4$ and $0\leq j < k$ we have
that
\[
\min_{\chi \in S_{h}^{4}}\|(w - \chi)^{(j)}\| \leq C h^{k-j} \|w^{(k)}\| \quad
\text{if} \quad w \in H^{k}
\]
and
\[
\min_{\chi \in S_{h}^{4}}\|(w - \chi)^{(j)}\|_{\infty} \leq C h^{k-j} \|w^{(k)}\|_{\infty} \quad
\text{if} \quad w \in C^{k}\,,
\]
and that a similar estimate holds in $S_{h,0}^{4}$ for $w$ in $H^{k}$ or in $C^{k}$ that also vanishes
at $x=0$ and $x=1$.
Let $a(\cdot,\cdot)$ denote again the bilinear form
\[
a(\psi,\chi) = (\psi,\chi) + \tfrac{1}{3}(\psi',\chi')\,, \quad \forall \psi,\chi \in S_{h,0}^{4}\,,
\]
and $R_{h} : H^{1} \to S_{h,0}^{4}$ be the associated elliptic projection operator
defined by
\[
a(R_{h}v,\chi) = a(v,\chi)\,, \quad \forall \chi \in S_{h,0}^{4}\,.
\]
It follows by standard estimates that if $1\leq k\leq 4$
\begin{equation}
\|R_{h}w - w\|_{j} \leq Ch^{k-j}\|w\|_{k}\quad j=0,1\,, \,\,\, \text{if}\,\,\,
w\in H^{k} \cap H_{0}^{1}\,,
\label{eq32}
\end{equation}
and that, \cite{ddw},
\begin{equation}
\|R_{h}w - w\|_{\infty} + h\|R_{h}w - w\|_{1,\infty}\leq Ch^{4}\|w\|_{4,\infty}\,\,\,\,
\text{if}\,\,\,\, w\in W^{4,\infty}\cap H_{0}^{1}\,.
\label{eq33}
\end{equation}
Similar estimates hold for the analogous elliptic projection into $S_{h}^{4}$. In addition,
as a consequence of the quasiuniformity of the mesh, the inverse inequalities
\begin{equation}
\begin{aligned}
\|\chi\|_{\beta} & \leq C h^{-(\beta-\alpha)}\|\chi\|_{\alpha}\,, \quad
0\leq \alpha\leq\beta\leq 2\,, \\
\|\chi\|_{s,\infty} & \leq Ch^{-(s +1/2)}\|\chi\|\,, \quad 0\leq s\leq 2\,,
\end{aligned}
\label{eq34}
\end{equation}
hold for any $\chi \in S_{h}^{4}$ (or any $\chi \in S_{h,0}^{4}$), and so does the estimate,
\cite{ddw},
\begin{equation}
\|Pv - v\|_{\infty} \leq C h^{4}\|v\|_{4,\infty}\,, \quad v \in W^{4,\infty}\,,
\label{eq35}
\end{equation}
where $P : L^{2} \to S_{h}^{4}$ is the $L^{2}$-projection operator onto $S_{h}^{4}$.
As a consequence of the approximation and inverse properties of the cubic spline spaces we also have that
$P$ is stable in $L^{\infty}$ and in $H^{1}$, and that $R_{h}$ is stable in $H_{0}^{1}$ and in
$H^{2}\cap H_{0}^{1}$. \par
We let the {\em standard Galerkin semidiscretization} on $S_{h}^{4}$ of $(\ref{cb})$ be defined
as follows: We seek $\eta_{h} : [0,T] \to S_{h}^{4}$, $u_{h} : [0,T] \to S_{h,0}^{4}$, such that
for $t \in [0,T]$
\begin{equation}
\begin{aligned}
(\eta_{ht},\phi) + (u_{hx},\phi) + ((\eta_{h}u_{h})_{x},\phi) & = 0 \quad \forall \phi \in S_{h}^{4}\,,\\
a(u_{ht},\chi) + (\eta_{hx},\chi) + (u_{h}u_{hx},\chi) & = 0\quad \forall \chi \in S_{h,0}^{4}\,,
\end{aligned}
\label{eq36}
\end{equation}
with initial conditions
\begin{equation}
\eta_{h}(0) = P\eta_{0}\,, \quad u_{h}(0) = R_{h}u_{0}\,.
\label{eq37}
\end{equation}
We also similarly define the analogous semidiscretization of $(\ref{scb})$ which is given
for $t \in [0,T]$ by
\begin{equation}
\begin{aligned}
(\eta_{ht},\phi) & + (u_{hx},\phi) + \tfrac{1}{2}((\eta_{h}u_{h})_{x},\phi) = 0
\quad \forall \phi \in S_{h}^{4}\,,\\
a(u_{ht},\chi) & + (\eta_{hx},\chi) + \tfrac{3}{2}(u_{h}u_{hx},\chi)
+ \tfrac{1}{2}(\eta_{h}\eta_{hx},\phi) = 0 \quad \forall \chi \in S_{h,0}^{4}\,,
\end{aligned}
\label{eq38}
\end{equation}
with
\begin{equation}
\eta_{h}(0) = P\eta_{0}\,, \quad u_{h}(0) = R_{h}u_{0}\,.
\label{eq39}
\end{equation}
The following result may be proved in a totally analogous manner to the analogous result
of Proposition $2.1$ and, hence, its proof will be omitted.
\begin{proposition} Let $h$ be sufficiently small. Suppose that the solutions of $(\ref{cb})$
and $(\ref{scb})$ are such that $\eta \in C(0,T;W_{\infty}^{4})$,
$u \in C(0,T;W_{\infty}^{4}\cap H_{0}^{1})$. Then, the semidiscrete problems
$(\ref{eq36})$, $(\ref{eq37})$ and $(\ref{eq38})$, $(\ref{eq39})$ have unique solutions
$(\eta_{h},u_{h})$ for $0\leq t\leq T$ that satisfy
\begin{align}
\max_{0\leq t\leq T}\|\eta(t) - \eta_{h}(t)\| & \leq C h^{3}\,,
\label{eq310} \\
\max_{0\leq t\leq T}\|u(t) - u_{h}(t)\|_{1} & \leq C h^{3}\,.
\label{eq311}
\end{align}
\end{proposition}
\subsection{Uniform mesh}
For integer $N \geq 2$ we let $h=1/N$ and $x_{i}=(i-1)h$, $i=1,2,\ldots,N+1$. In this section we will
denote by $\{\phi_{j}\}_{j=1}^{N+3}$ the usual $B$-spline basis of $S_{h}^{4}$ defined by
the restrictions on $[0,1]$ of the functions
$\phi_{j}(x)=\Phi(\tfrac{x}{h}-(j-2))$, where $\Phi$ is the cubic spline on $\mathbb{R}$ with
respect to the partition $\{-2,-1,0,1,2\}$ with support $[-2,2]$ and nodal values
$\Phi(0)=1$, $\Phi(\pm 1)=1/4$, $\Phi(\pm 2)=0$. Thus, e.g. supp$(\phi_{j})=[x_{j-3},x_{j+1}]$
and $\phi_{j}(x_{j-1})=1$ for $4\leq j\leq N$, etc. Our aim in this paragraph is to prove that
the cubic spline standard Galerkin semidiscretizations of the two systems on this mesh satisfy the
estimates $\|\eta-\eta_{h}\|=O( h^{3.5}\sqrt{\ln 1/h})$, $\|u-u_{h}\|_{1}=O(h^{3})$,
$\|u-u_{h}\|=O(h^{4}\sqrt{\ln 1/h})$, $\|u_{t}-u_{ht}\|=O(h^{4}\sqrt{\ln 1/h})$.
For this purpose we shall first state and prove a series of auxiliary results. \par
Our first lemma
is a well known result, the cubic spline analog of Lemma $2.1$.
\begin{lemma} (i) Let $G_{ij}=(\phi_{j},\phi_{i})$, $1\leq i,j\leq N+3$. Then, there exist
positive constants $c_{1}$ and $c_{2}$ such that
\[
c_{1}h|\gamma|^{2} \leq \langle G\gamma,\gamma\rangle \leq c_{2}h|\gamma|^{2}\quad
\forall \gamma \in \mathbb{R}^{N+3}\,.
\]
(ii) Let $b \in \mathbb{R}^{N+3}$, $\gamma=G^{-1}b$, and
$\psi = \sum_{j=1}^{N+3}\gamma_{j}\psi_{j}$. Then
\[
\|\psi\| \leq (c_{1}h)^{-1/2}|b|\,.
\]
(iii) Let $w \in C^{5}$. Then, there exists a constant $C_{1}=C_{1}(\|w^{(5)}\|_{\infty})$
such that for any $\hat{x} \in [x_{i},x_{i+1}]$,
\[
(w - I_{h}w)(x)=\tfrac{1}{4!}w^{(4)}(\hat{x})(x-x_{i})^{2}(x_{i+1}-x)^{2} + \tilde{g}(x)\,,
\quad x_{i}\leq x\leq x_{i+1}\,,
\]
where $\|\tilde{g}\|_{\infty} + h\|\tilde{g}'\|_{\infty}\leq C_{1}h^{5}$.
\end{lemma}
\begin{proof}
The proofs of $(i)$-$(iii)$ are given in \cite{d} in the case of periodic cubic splines.
It is straightforward to adapt them to the case of $S_{h}^{4}$ and $I_{h}$ at hand.
\end{proof}
We next prove a superaccuracy estimate for $P[(w\ve)']$, where $\ve$ is the error of the
cubic spline interpolant of a sufficiently smooth function and $w$ is a $C^{1}$ weight.
This estimate is a consequence of cancellation effects due to the uniform mesh and may be viewed as the cubic
spline analog of Lemma $2.2$.
\begin{lemma} Let $v\in C^{5}$ and $w\in C^{1}$. If $\ve=v-I_{h}v$ and $\psi \in S_{h}^{4}$ is
such that
\[
(\psi,\phi) = ((w\ve)',\phi)\quad \forall \phi \in S_{h}^{4}\,,
\]
then $\|\psi\|\leq Ch^{3.5}$. If in addition $w(0)=w(1)=0$, then $\|\psi\|\leq Ch^{4}$.
\end{lemma}
\begin{proof}
Let $b_{j} = ((w\ve)',\phi_{j})=-(w\ve,\phi_{j}')$, $1\leq j\leq N+3$. In view of Lemma $3.1(ii)$
it suffices to show that $|b|=O(h^{4})$. It is clear by $(\ref{eq31})$ and the properties of
the basis functions $\phi_{j}$ that $b_{j}=O(h^{4})$ for all $j$. We will prove that actually
$b_{j}=O(h^{5})$ for $j=4,5,\ldots,N$, thus establishing that $|b|=O(h^{4})$. Let $4\leq j\leq N$.
Using Lemma $3.1(iii)$ and putting $q(x)=x^{2}(h-x)^{2}/4!$ we have
\begin{align*}
(w\ve,\phi_{j}') & = \sum_{k=0}^{3}\int_{x_{j-3+k}}^{x_{j-2+k}}w(x)\ve(x)\phi_{j}'(x)dx \\
& = v^{(4)}(x_{j-2})\int_{x_{j-3}}^{x_{j-2}}w(x)q(x-x_{j-3})\phi_{j}'(x)dx
+ v^{(4)}(x_{j-1})\int_{x_{j-2}}^{x_{j-1}}w(x)q(x-x_{j-2})\phi_{j}'(x)dx \\
& \,\,\,\,\,
+ v^{(4)}(x_{j-1})\int_{x_{j-1}}^{x_{j}}w(x)q(x-x_{j-1})\phi_{j}'(x)dx
+ v^{(4)}(x_{j})\int_{x_{j}}^{x_{j+1}}w(x)q(x-x_{j})\phi_{j}'(x)dx + O(h^{5}) \\
&=: v^{(4)}(x_{j-2})J_{1} + v^{(4)}(x_{j-1})(J_{2}+J_{3}) + v^{(4)}(x_{j})J_{4} + O(h^{5})\,.
\end{align*}
Since $v^{(4)}(x_{j-2}) = v^{(4)}(x_{j-1}) - h v^{(5)}(t_{1j})$ and
$v^{(4)}(x_{j})=v^{(4)}(x_{j-1}) + h v^{(5)}(t_{2j})$ for some $t_{1j}\in (x_{j-2},x_{j-1})$
and $t_{2j}\in (x_{j-1},x_{j})$, we obtain
\begin{equation}
(w\ve,\phi_{j}') = v^{(4)}(x_{j-1})(J_{1} + J_{2} + J_{3} + J_{4}) + O(h^{5})\,.
\label{eq312}
\end{equation}
Suitable changes of variable in each of the four integrals yield
\begin{align*}
J_{1} & = \int_{0}^{h}w(x+x_{j-3})q(x)\phi_{4}'(x)dx\,, \quad
J_{2} = \int_{0}^{h}w(x+x_{j-2})q(x)\phi_{3}'(x)dx\,, \\
J_{3} & = \int_{0}^{h}w(x+x_{j-1})q(x)\phi_{2}'(x)dx\,, \quad
J_{4} = \int_{0}^{h}w(x+x_{j})q(x)\phi_{1}'(x)dx\,.
\end{align*}
In addition, if $x_{j\mu}=(x_{j-2} + x_{j-1})/2$, we have for $x \in [0,h]$
\begin{align*}
w(x+x_{j-3}) & = w(x+x_{j\mu}) - \tfrac{3h}{2}w'(\tau_{1jx})\,, \quad
w(x+x_{j-2}) = w(x+x_{j\mu}) -\tfrac{h}{2}w'(\tau_{2jx})\,,\\
w(x+x_{j-1}) & = w(x+x_{j\mu}) + \tfrac{h}{2}w'(\tau_{3jx})\,, \quad
w(x+x_{j}) = w(x+x_{j\mu}) + \tfrac{3h}{2}w'(\tau_{4jx})\,,
\end{align*}
for some $\tau_{ijx}$ between $x+x_{j-4+i}$ and $x+x_{j\mu}$, $i=1,2,3,4$. Thus,
\begin{align*}
J_{1}+J_{2}+J_{3}+J_{4} & = \int_{0}^{h}w(x+x_{j\mu})q(x)[\phi_{1}'(x)+\phi_{2}'(x)
+\phi_{3}'(x) + \phi_{4}'(x)]dx + O(h^{5})\\
& = w(x_{j\mu})\int_{0}^{h}q(x)[\phi_{1}'(x)+\phi_{2}'(x)
+\phi_{3}'(x) + \phi_{4}'(x)]dx + O(h^{5})\,.
\end{align*}
The last integral is equal to zero, since $\phi_{4}(x)=\phi_{1}(h-x)$,
$\phi_{3}(x) =\phi_{4}(h-x)$, and $q(x)=q(h-x)$ for $x\in [0,h]$. Hence
\[
J_{1} + J_{2} + J_{3} + J_{4} = O(h^{5})\,,
\]
and $(\ref{eq312})$ implies that $b_{j}=O(h^{5})$. Thus, the first assertion of the lemma
is verified. To prove the second assertion, suppose that $w(0)=0$. Then $b_{j}=O(h^{5})$
for $j=1,2,3$. Indeed, using Lemma $3.1(iii)$, we have
\begin{align*}
-b_{1} = (w\ve,\phi_{1}') & = v^{(4)}(h)\int_{0}^{h}w(x)q(x)\phi_{1}'(x)dx + O(h^{5}) \\
& = v^{(4)}(h)\int_{0}^{h}xw'(t_{x})q(x)\phi_{1}'(x)dx + O(h^{5}) = O(h^{5})\,.
\end{align*}
In addition,
\begin{align*}
-b_{2} & = (w\ve,\phi_{2}') = \int_{0}^{h}w(x)\ve(x)\phi_{2}'(x)dx +
\int_{h}^{2h}w(x)\ve(x)\phi_{2}'(x)dx \\
& = v^{(4)}(h)\int_{0}^{h}w(x)q(x)\phi_{2}'(x)dx + v^{(4)}(h)\int_{h}^{2h}w(x)q(h-x)\phi_{2}'(x)dx
+ O(h^{5})\\
& = v^{(4)}(h)\int_{0}^{h}w(x)q(x)\phi_{2}'(x)dx + v^{(4)}(h)\int_{0}^{h}w(x+h)q(x)\phi_{1}'(x)dx
+ O(h^{5})\\
& = v^{(4)}(h)\int_{0}^{h}xw'(t_{x})q(x)\phi_{2}'(x)dx
+ v^{(4)}(h)\int_{0}^{h}(w(x+h)-w(x))q(x)\phi_{1}'(x)dx\\
& = - b_{1} + O(h^{5}) = O(h^{5})\,.
\end{align*}
Finally,
\begin{align*}
-b_{3} & = (w\ve,\phi_{3}') = \int_{0}^{h}w(x)\ve(x)\phi_{3}'(x)dx +
\int_{h}^{2h}w(x)\ve(x)\phi_{3}'(x)dx + \int_{2h}^{3h}w(x)\ve(x)\phi_{3}'(x)dx \\
& = v^{(4)}(h)\Bigl[\int_{0}^{h}w(x)q(x)\phi_{3}'(x)dx + \int_{0}^{h}w(x+h)q(x)\phi_{2}'(x)dx
\\
&\qquad \qquad \quad + \int_{0}^{h}w(x+2h)q(x)\phi_{1}'(x)dx\Bigr] + O(h^{5})\\
& = v^{(4)}(h)\Bigl[\int_{0}^{h}xw'(t_{x})q(x)\phi_{3}'(x)dx - b_{1} - b_{2}\Bigr] + O(h^{5})
= O(h^{5})\,.
\end{align*}
Similarly, if $w(1)=0$ we have $b_{j}=O(h^{5})$ for $j=N+1,N+2,N+3$. Hence, if $w$
vanishes at $x=0$ and $x=1$, $|b|=O(h^{5})$ and the second assertion of the lemma follows by
Lemma $3.1(ii)$.
\end{proof}
We shall also derive a superaccuracy estimate for $P[(we)']$, where
$e$ is the error of the elliptic projection of a function $v \in C_{0}^{5}$ and
$w$ is a $C^{1}$ weight. For this purpose, we first state two superconvergence results that
follow from the analysis of Wahlbin in \cite{w}, and which are valid for interior nodes whose
distance from the endpoints of the interval is at least of $O(h \ln 1/h)$.
\begin{proposition}
Suppose that $v\in C_{0}^{5}$ and let $v_{h}=R_{h}v$ be its elliptic projection onto
$S_{h,0}^{4}$. Then the following hold:
\begin{itemize}
\item[$(i)$] There exists a constant $C$ independent of $h$ such that
\begin{equation}
|(v-v_{h})'(x_{i})| \leq Ch^{4}\|v\|_{W_{\infty}^{5}} \quad
\text{provided} \quad \text{dist}(x_{i},\partial I)\geq C_{1}h\ln\tfrac{1}{h}\,,
\label{eq313}
\end{equation}
where $C_{1}$ is a sufficiently large constant independent of $h$.
\item[$(ii)$] If $x_{i}$, $x_{i+1}$ are two adjacent nodes for which the second inequality in
$(\ref{eq313})$ holds and $e(x)=v(x) - v_{h}(x)$, we have
\begin{equation}
e(x_{i+1}) - e(x_{i}) = O(h^{5})\|v\|_{W_{\infty}^{5}}\,.
\label{eq314}
\end{equation}
\end{itemize}
\end{proposition}
\begin{proof}
$(i)$ The estimate $(\ref{eq313})$ follows from Corollary $1.6.2$ of \cite{w} (which is
strictly valid when the elliptic projection is defined by $(\tilde{v}_{h}',\phi')=(v',\phi')$
$\forall \phi \in S_{h,0}^{4}$), combined with Theorem $1.3.1$ of \cite{w} that allows us
to state the result for $v_{h}=R_{h}v$ defined as in section $2$ of the paper at hand. \par
$(ii)$ The cancellation property expressed by $(\ref{eq314})$ is a consequence of the fact that
$\tilde{e}(x) = v(x) - \tilde{v}_{h}(x)$ may be represented in the form
\[
\tilde{e}(x) = Q_{i}(x) + O(h^{5})\|w\|_{W_{\infty}^{5}(x_{i},x_{i+1})}\,, \quad
x\in [x_{i},x_{i+1}]\,,
\]
where $Q_{i}(x)$ is a polynomial of degree four such that $Q_{i}(x_{i})=Q_{i}(x_{i+1})$.
This representation follows from the remarks in Example $1.8.2$ of \cite{w} and by adapting
the arguments of the proof in Section $1.8$ of \cite{w} (which are valid for $C^{1}$ Hermite cubics)
to the case of the $C^{2}$ cubic splines at hand. Hence
$\tilde{e}(x_{i+1}) - \tilde{e}(x_{i}) = O(h^{5})\|w\|_{W_{\infty}^{5}(x_{i},x_{i+1})}$, and
$(\ref{eq314})$ follows by the function values superconvergence estimate for elliptic
projections given in Theorem $1.3.2$ of \cite{w}.
\end{proof}
In the next lemma we present some further formulas for the error $v-R_{h}v$ that will be
used in the sequel.
\begin{lemma}
(i) Let $v\in C_{0}^{5}$, $v_{h}=R_{h}v$, $e=v-v_{h}$, and $1\leq i\leq N$. Then,
there exists a constant $C=C(\|w\|_{W_{\infty}^{5}})$ such that for any
$\hat{x} \in [x_{i},x_{i+1}]$,
\begin{equation}
e(x) = \gamma_{i}(x) + \tfrac{1}{4!}v^{(4)}(\hat{x})(x-x_{i})^{2}(x_{i+1}-x)^{2}
+ \delta_{i}(x)\,, \quad x\in[x_{i},x_{i+1}]\,,
\label{eq315}
\end{equation}
where $\|\delta_{i}\|_{\infty}\leq Ch^{5}$ and $\gamma_{i}$ is the cubic Hermite polynomial
interpolating the values of $e$ and its derivative at the nodes $x_{i}$ and $x_{i+1}$.\\
(ii) In addition to the hypotheses in (i), suppose that $x_{i}$ and $x_{i+1}$ satisfy the
second inequality in $(\ref{eq313})$. Then, there is a constant $C=C(\|v\|_{W_{\infty}^{5}})$
such that for any $\hat{x}\in [x_{i},x_{i+1}]$
\begin{equation}
e(x)=e(x_{i}) + \tfrac{1}{4!}v^{(4)}(\hat{x})(x-x_{i})^{2}(x_{i+1}-x)^{2}
+ \tilde{\delta}_{i}(x)\,, \quad x\in [x_{i},x_{i+1}]\,,
\label{eq316}
\end{equation}
where $\|\tilde{\delta}_{i}\|_{\infty} \leq Ch^{5}$.
\end{lemma}
\begin{proof} $(i)$ By the standard representation of the error of Hermite interpolation
we have, since $v_{h}\in\mathbb{P}_{3}$ in $[x_{i},x_{i+1}]$, that for $x\in [x_{i},x_{i+1}]$
there holds
\[
e(x) - \gamma_{i}=\tfrac{1}{4!}(x-x_{i})^{2}(x_{i+1}-x)^{2}v^{(4)}(t_{x})
\]
for some $t_{x} \in (x_{i},x_{i+1})$. Hence $(\ref{eq315})$ follows from the mean-value
theorem since $v \in C^{5}$.\\
$(ii)$ We let $I=[0,1]$, and $K=\{x\in I :
dist(x,\partial I)\geq C_{1}h\ln 1/h\}$. In view of $(\ref{eq315})$ it suffices to show
that
\[
\gamma_{i}(x) = e(x_{i}) + O(h^{5})\|v\|_{W_{\infty}^{5}} \quad \text{for}\quad
x_{i},x_{i+1} \in K\,.
\]
Now
\[
\gamma_{i}(x) = e(x_{i})A_{i,1}(x) + e'(x_{i})B_{i,1}(x) + e(x_{i+1})A_{i,2}(x)
+e'(x_{i+1})B_{i,2}(x)\,,
\]
where
\begin{align*}
A_{i,1}(x) & = \bigl[ 1+\tfrac{2(x-x_{i})}{h}\bigr] \tfrac{(x_{i+1}-x)^{2}}{h^{2}}\,, \quad
B_{i,1}(x) = \tfrac{(x-x_{i})(x_{i+1}-x)^{2}}{h^{2}}\,, \\
A_{i,2}(x) & = \bigl[ 1+\tfrac{2(x_{i+1}-x)}{h}\bigr] \tfrac{(x-x_{i})^{2}}{h^{2}}\,, \quad
B_{i,2}(x) = \tfrac{-(x_{i+1}-x)(x-x_{i})^{2}}{h^{2}}\,.
\end{align*}
Hence, using $(\ref{eq313})$ and $(\ref{eq314})$ we see that
\begin{align*}
\gamma_{i}(x) & = e(x_{i})A_{i,1}(x) + e(x_{i+1})A_{i,2}(x) + O(h^{5})\|v\|_{W_{\infty}^{5}}\\
& = e(x_{i})(A_{i,1}(x) + A_{i,2}(x)) + O(h^{5})\|v\|_{W_{\infty}^{5}}\\
& = e(x_{i}) + O(h^{5})\|v\|_{W_{\infty}^{5}}\,,
\end{align*}
which concludes the proof of the lemma.
\end{proof}
We are now ready to prove the required estimate for $P[(we)']$, where $w$ is a $C^{1}$ weight.
\begin{lemma} Let $v\in C_{0}^{5}$ and $w\in C^{1}$. If $e=v-R_{h}v$ and
$\psi \in S_{h}^{4}$ is such that
\[
(\psi,\phi) = ((we)',\phi) \quad \forall \phi \in S_{h}^{4}\,,
\]
then $\|\psi\| \leq Ch^{3.5}\sqrt{\ln 1/h}$.
\end{lemma}
\begin{proof}
Let $b_{j}=((we)',\phi_{j})$, $1\leq j\leq N+3$. In view of Lemma $3.1(ii)$ it suffices to show
that $|b| = O(h^{4}\sqrt{\ln 1/h})$. It is clear by $(\ref{eq32})$ and the properties of the basis
functions $\phi_{j}$ that $b_{j}=O(h^{4})$. We will show that for `most' of the the indices $j$
it is true that $b_{j}=O(h^{5})$.
Let $I=[0,1]$ and $K=\{x\in I : dist(x,\partial I)\geq C_{1}h\ln 1/h\}$. Since
$x_{i}=(i-1)h$, $1\leq i\leq N+1$, it is clear that if $x_{i} \in K$ then
$1+C_{1}\ln 1/h \leq i\leq N+1-C_{1}\ln 1/h$. Therefore $x_{i}\in K$ if and only if
$i\in M=\{\nu \in \mathbb{N} : \nu \in [1+C_{1}\ln 1/h, N+1-C_{1}\ln 1/h] \}$. Let $\mu=\min M$ and $m=\max M$.
We shall show that
\begin{equation}
|b_{j}| \leq Ch^{5}\,, \quad j=\mu + 4,\ldots, m-1\,.
\label{eq317}
\end{equation}
Indeed, for such $j$ we have
\begin{align*}
-b_{j} & = (we,\phi_{j}') = \sum_{k=0}^{3}\int_{x_{j-3+k}}^{x_{j-2+k}}w(x)e(x)\phi_{j}'(x)dx \\
& = \sum_{k=0}^{3}\Bigl{\{}e(x_{j-3+k})\int_{x_{j-3+k}}^{x_{j-2+k}}w(x)\phi_{j}'(x)dx
+ \int_{x_{j-3+k}}^{x_{j-2+k}}(e(x)-e(x_{j-3+k}))w(x)\phi_{j}'(x)dx\Bigr{\}}\,.
\end{align*}
Using now Lemma $3.3(ii)$ in each one of the second group of integrals of this expression and similar
considerations as in the proof of Lemma $3.2$, we see that
\[
-b_{j} = e(x_{j-3})J_{1} + e(x_{j-2})J_{2} + e(x_{j-1})J_{3} + e(x_{j})J_{4} + O(h^{5})\,,
\]
where $J_{i} = \int_{x_{j-4+i}}^{x_{j-3+i}}w(x)\phi_{j}'(x)dx$, $1\leq i\leq 4$. But since
$e(x_{i+1}) = e(x_{i}) + O(h^{5})$ when $x_{i},x_{i+1} \in K$ (by Proposition $3.2(ii)$), and
taking into account that the $J_{i}$ are bounded independently of $h$ and that
$J_{1}+ J_{2} + J_{3} + J_{4} = O(h)$ (by similar considerations as in the proof of
Lemma $3.2$), we finally obtain $(\ref{eq317})$. By the definitions of $\mu$ and $m$ we now
have
\begin{align*}
|b|^{2} & = \sum_{j=1}^{\mu+3}b_{j}^{2} + \sum_{j=\mu+4}^{m-1}b_{j}^{2}
+ \sum_{j=m}^{N+3}b_{j}^{2} \leq C[(\mu+3)h^{8} + h^{9} + (N+4-m)h^{8}]\\
& \leq C[(C_{2} + C_{1}\ln\tfrac{1}{h})h^{8} + h^{9}]\leq C_{3}h^{8}\ln\tfrac{1}{h}\,,
\end{align*}
where $C_{3}=C_{3}(\|v\|_{W_{\infty}^{5}}, \|w\|_{W_{\infty}^{1}}, C_{1})$, thus
concluding the proof of the lemma.
\end{proof}
The main result of this section follows.
\begin{theorem} Let $h=1/N$ be sufficiently small. Suppose that the solutions of
$(\ref{cb})$ and $(\ref{scb})$ are such that $\eta \in C(0,T;C^{5})$,
$\eta_{t}\in C(0,T;C^{4})$, $u, u_{t} \in C(0,T;C_{0}^{5})$. Then the semidiscrete
problems $(\ref{eq36})$ and $(\ref{eq38})$ with initial conditions
\begin{equation}
\eta_{h}(0)=I_{h}\eta_{0}\,, \quad u_{h}(0)=R_{h}u_{0}\,,
\label{eq318}
\end{equation}
have unique solutions $(\eta_{h},u_{h})$ for $0\leq t\leq T$ that satisfy
\begin{align*}
(i)\,\,\,\,\, \max_{0\leq t\leq T}\|\eta(t)-\eta_{h}(t)\| & \leq Ch^{3.5}\sqrt{\ln\tfrac{1}{h}}\,,
\quad \max_{0\leq t\leq T}\|u(t)-u_{h}(t)\|_{1}\leq Ch^{3}\,, \\
(ii)\,\,\,\, \max_{0\leq t\leq T}\|u(t)-u_{h}(t)\|&\leq Ch^{4}\sqrt{\ln\tfrac{1}{h}}\,,\quad
\,\,\,\max_{0\leq t\leq T}\|u_{t}(t)-u_{ht}(t)\|\leq Ch^{4}\sqrt{\ln\tfrac{1}{h}}\,.
\end{align*}
\end{theorem}
\begin{proof}
Again, we give the proof in the case of $(\ref{scb})$, noting that the analogous proof
for $(\ref{cb})$ follows as in Proposition $2.1$. The technique of proof is basically the
same as the one used in Theorem $2.1$ and we shall give here the details that are different.
We define now
\[
\rho:=\eta-I_{h}\eta\,, \quad \theta:=I_{h}\eta - \eta_{h}\,, \quad
\sigma:=u-R_{h}u\,, \quad \xi:=R_{h}u - u_{h}\,,
\]
and, arguing as in the proof of Theorem $2.1$, we have for $0\leq t\leq T$
\begin{equation}
\begin{aligned}
\tfrac{1}{2} & \tfrac{d}{dt}(\|\theta\|^{2} + \|\xi\|_{1}^{2})
+ \bigl( ((1+\tfrac{1}{2}\eta)\sigma)_{x},\theta\bigr) + \tfrac{1}{2}((u\theta)_{x},\theta)
+ \tfrac{1}{2}(F_{x},\theta) \\
& +(\rho_{x},\xi) + \tfrac{1}{2}((\eta\theta)_{x},\xi)
+ \tfrac{1}{2}(G_{x},\xi) + \tfrac{3}{2}(H_{x},\xi) \\
& = -(\rho_{t},\theta)\,,
\end{aligned}
\label{eq319}
\end{equation}
where
\begin{align*}
F & := \eta\xi + u\rho - \rho\sigma - \rho\xi - \theta\sigma\,, \\
G & := \eta\rho - \rho\theta - \tfrac{1}{2}\rho^{2}\,, \\
H & := u\sigma + u\xi - \sigma\xi - \tfrac{1}{2}\sigma^{2} - \tfrac{1}{2}\xi^{2}\,.
\end{align*}
Using the approximation properties of $S_{h}^{4}$ and $S_{h,0}^{4}$, integration by parts,
and Lemmas $3.2$ and $3.4$ we have
\begin{align*}
\bigl| \bigl( ((1+\tfrac{1}{2}\eta)\sigma)_{x},\theta\bigr)\bigr| & \leq
Ch^{3.5}\sqrt{\ln\tfrac{1}{h}}\|\theta\|\,, \hspace{60pt}
|((u\theta)_{x},\theta)|\leq C\|\theta\|^{2}\,,\\
|(F_{x},\theta)| & \leq C(\|\xi\|_{1}\|\theta\| + h^{4}\|\theta\| + \|\theta\|^{2})\,, \hspace{20pt}
|(\rho_{x},\xi)| \leq Ch^{4}\|\xi\|_{1}\,, \\
|((\eta\theta)_{x},\xi)| & \leq C\|\xi\|_{1}\|\theta\|\,, \hspace{94pt}
|(G_{x},\xi)| \leq C(\|\xi\|_{1}\|\theta\| + h^{4}\|\xi\|_{1})\,,\\
|(H_{x},\xi)| & \leq C(\|\xi\|^{2} + h^{4}\|\xi\|_{1})\,, \hspace{64pt}
|(\rho_{t},\theta)| \leq C h^{4}\|\theta\|\,.
\end{align*}
(Here and in the rest of the proof of this theorem we have omitted details that are similar to those
in the analogous steps of the proof of Theorem $2.1$.) Therefore, $(\ref{eq319})$ gives for
$0\leq t\leq T$
\[
\tfrac{d}{dt}(\|\theta\|^{2} + \|\xi\|_{1}^{2}) \leq C(h^{7}\ln \tfrac{1}{h}
+ \|\theta\|^{2} + \|\xi\|_{1}^{2})\,,
\]
from which, by Gronwall's Lemma and $(\ref{eq318})$ we get
\begin{equation}
\|\theta\| + \|\xi\|_{1} \leq Ch^{3.5}\sqrt{\ln \tfrac{1}{h}}\,,
\label{eq320}
\end{equation}
and the estimates $(i)$ follow. \par
In order to prove the estimates $(ii)$, we first note that that
\[
\|F\|\leq C(\|\xi\| + h^{4})\,, \quad \|G\|\leq Ch^{4}\,, \quad \|H\|\leq C(\|\xi\| + h^{4})\,.
\]
Noting that $\xi_{t} = R_{h}v$, where $v$ is the solution of the problem
\begin{align*}
& v - \tfrac{1}{3}v'' = -(\theta + \rho)_{x} - \tfrac{1}{2}(\eta\theta)_{x}
-\tfrac{1}{2}(G-\tfrac{1}{2}\theta^{2}+3H)_{x}\,, \quad x\in [0,1]\,,\\
& v(0) = v(1) = 0\,,
\end{align*}
we see by $(\ref{eq320})$ that $\|v\|_{1} \leq Ch^{3.5}\sqrt{\ln 1/h}$. Arguing as in the proof of
Theorem $2.1$ we have now $\|v\| \leq C(h^{4} + \|\gamma\| + \|\xi\|)$,
where $\gamma = \int_{0}^{x}(\theta(s,t)+ \rho(s,t))ds - x\int_{0}^{1}\rho(s,0)ds$. Then there follows
that
\[
\tfrac{1}{2}\tfrac{d}{dt} \|\xi\|^{2} \leq C(h^{8} + \|P_{0}\gamma\|^{2} + \|\xi\|^{2})\,,
\]
where $P_{0}$ is the $L^{2}$-projection operator onto $S_{h}^{3}$, the space of $C^{1}$
piecewise quadratic functions relative to the partition $\{x_{j}\}$. We also have, in view of
$(\ref{eq320})$ and the estimate $\|P_{0}\gamma - \gamma\| \leq Ch\|\gamma\|_{1}$, that
\[
\tfrac{1}{2}\tfrac{d}{dt} \|P_{0}\gamma\|^{2} \leq C(h^{8}\ln\tfrac{1}{h} + \|P_{0}\gamma\|^{2}
+ \|\xi\|^{2})\,.
\]
From the last two inequalities and Gronwall's Lemma we obtain
\[
\|P_{0}\gamma\| + \|\xi\| \leq Ch^{4}\sqrt{\ln\tfrac{1}{h}}\,,
\]
from which the first estimate of $(ii)$ follows. The second estimate also follows along the lines
of the proof of Theorem $2.1$.
\end{proof}
\begin{remark}
The results of the Theorem also hold if we take as $\eta_{h}(0)$ any other approximation
of $\eta_{0}$ in $S_{h}^{4}$ of optimal order of accuracy in $L^{2}$.
\end{remark}
\begin{remark}
The superaccuracy estimate $\|\xi\|_{1} = O(h^{3.5}\sqrt{\ln 1/h})$ of $(\ref{eq320})$, combined
with $(\ref{eq33})$ and Sobolev's inequality yield the $L^{\infty}$ estimate
$\|u-u_{h}\|_{\infty} = O(h^{3.5}\sqrt{\ln 1/h})$.
\end{remark}
\subsection{Numerical experiments}
We considered the nonhomogeneous $(\ref{scb})$ system and we discretized it on a uniform mesh
with diminishing $h = 1/N$ on $[0,1]$
\def\baselinestretch{1}
\captionsetup[subtable]{labelformat=empty,position=top,margin=1pt,singlelinecheck=false}
\scriptsize
\begin{table}[h]
\subfloat[$L^{2}$-errors]{
\begin{tabular}[h]{ | c | c | c | c | c | }\hline
$N$   &    $\eta$      &  $order$  &      $u$       &   $order$   \\ \hline
$40$  &  $0.8063(-6)$  &           &  $0.8032(-7)$  &            \\ \hline
$80$  &  $0.7178(-7)$  &  $3.490$  &  $0.5062(-8)$  &   $3.988$   \\ \hline
$120$ &  $0.1744(-7)$  &  $3.489$  &  $0.1003(-8)$  &   $3.993$   \\ \hline
$160$ &  $0.6393(-8)$  &  $3.489$  &  $0.3178(-9)$	&   $3.994$   \\ \hline
$200$ &	 $0.2934(-8)$  &  $3.490$  &  $0.1303(-9)$	&   $3.996$   \\ \hline
$240$ &	 $0.1553(-8)$  &  $3.490$  &  $0.6288(-10)$	&   $3.996$   \\ \hline
$280$ &	 $0.9068(-9)$  &  $3.490$  &  $0.3395(-10)$	&   $3.998$   \\ \hline
$320$ &  $0.5691(-9)$  &  $3.490$  &  $0.1986(-10)$  &  $4.015$   \\ \hline
$360$ &  $0.3773(-9)$  &  $3.490$  &  $0.1238(-10)$  &  $4.012$   \\ \hline
$400$ &  $0.2612(-9)$  &  $3.489$  &  $0.8106(-11)$  &  $4.021$   \\ \hline
$440$ &  $0.1873(-9)$  &  $3.489$  &  $0.5533(-11)$  &  $4.008$   \\ \hline
$480$ &  $0.1382(-9)$  &  $3.492$  &  $0.4011(-11)$  &  $3.696$   \\ \hline
$520$ &  $0.1046(-9)$  &  $3.488$  &  $0.2840(-11)$  &  $4.315$   \\ \hline
\end{tabular}
}\qquad
\subfloat[$L^{\infty}$-errors]{
\begin{tabular}[h]{ | c | c | c | c | c | }\hline
$N$   &    $\eta$      &  $order$  &      $u$       &   $order$   \\ \hline
$40$  &  $0.2666(-5)$  &           &  $0.2456(-6)$  &            \\ \hline
$80$  &  $0.3611(-6)$  &  $2.884$  &  $0.1590(-7)$  &   $3.949$   \\ \hline
$120$ &  $0.1100(-6)$  &  $2.931$  &  $0.3175(-8)$  &   $3.973$   \\ \hline
$160$ &  $0.4708(-7)$  &  $2.951$  &  $0.1010(-8)$	&   $3.982$   \\ \hline
$200$ &	 $0.2431(-7)$  &  $2.962$  &  $0.4150(-9)$	&   $3.986$   \\ \hline
$240$ &	 $0.1415(-7)$  &  $2.969$  &  $0.2005(-9)$	&   $3.989$   \\ \hline
$280$ &	 $0.8947(-8)$  &  $2.973$  &  $0.1084(-9)$	&   $3.990$   \\ \hline
$320$ &  $0.6010(-8)$  &  $2.980$  &  $0.6367(-10)$  &  $3.986$   \\ \hline
$360$ &  $0.4230(-8)$  &  $2.982$  &  $0.3980(-10)$  &  $3.988$   \\ \hline
$400$ &  $0.3088(-8)$  &  $2.986$  &  $0.2616(-10)$  &  $3.982$   \\ \hline
$440$ &  $0.2323(-8)$  &  $2.988$  &  $0.1790(-10)$  &  $3.985$   \\ \hline
$480$ &  $0.1797(-8)$  &  $2.950$  &  $0.1256(-10)$  &  $4.070$   \\ \hline
$520$ &  $0.1412(-8)$  &  $3.014$  &  $0.9168(-11)$  &  $3.932$   \\ \hline
\end{tabular}
}
\\
\subfloat[$H^{1}$-errors]{
\begin{tabular}[h]{ | c | c | c | c | c | }\hline
$N$   &    $\eta$      &  $order$  &      $u$       &   $order$   \\ \hline
$40$  &  $0.1298(-3)$  &           &  $0.2014(-4)$  &            \\ \hline
$80$  &  $0.2215(-4)$  &  $2.550$  &  $0.2540(-5)$  &   $2.987$   \\ \hline
$120$ &  $0.7923(-5)$  &  $2.536$  &  $0.7550(-6)$  &   $2.992$   \\ \hline
$160$ &  $0.3829(-5)$  &  $2.528$  &  $0.3190(-6)$	&   $2.995$   \\ \hline
$200$ &	 $0.2182(-5)$  &  $2.521$  &  $0.1635(-6)$	&   $2.996$   \\ \hline
$240$ &	 $0.1379(-5)$  &  $2.516$  &  $0.9467(-7)$	&   $2.997$   \\ \hline
$280$ &	 $0.9363(-6)$  &  $2.512$  &  $0.5964(-7)$	&   $2.997$   \\ \hline
$320$ &  $0.6699(-6)$  &  $2.508$  &  $0.3997(-7)$  &   $2.998$   \\ \hline
$360$ &  $0.4987(-6)$  &  $2.506$  &  $0.2808(-7)$  &   $2.998$   \\ \hline
$400$ &  $0.3831(-6)$  &  $2.503$  &  $0.2047(-7)$  &   $2.998$   \\ \hline
$440$ &  $0.3018(-6)$  &  $2.502$  &  $0.1538(-7)$  &   $2.998$   \\ \hline
$480$ &  $0.2428(-6)$  &  $2.500$  &  $0.1185(-7)$  &   $2.998$   \\ \hline
$520$ &  $0.1988(-6)$  &  $2.499$  &  $0.9323(-8)$  &   $2.999$   \\ \hline
\end{tabular}
}
\normalsize
\caption{Errors and orders of convergence. $(\ref{scb})$ system, standard Galerkin
semidiscretization with cubic splines on a uniform mesh.}
\label{tbl31}
\end{table}
\normalsize
using cubic splines for the spatial discretization and
fourth-order accurate explicit Runge Kutta time stepping, (cf. Section $4$), with time step
$k=h/10$ for which the temporal discretization error was negligible in comparison with the
spatial error.
We took a suitable right-hand side so that the exact solution of the system
was $\eta = \exp(2t)(\cos(\pi x) + x + 2)$, $u=\exp(xt)(\sin(\pi x) + x^{3}-x^{2})$. The errors
and orders of convergence produced by this numerical experiments are shown in Table \ref{tbl31}.
The rates are close to the theoretical predictions of Theorem $3.1$. The table suggests that
the $L^{2}$ rate of convergence for $\eta$ is slightly less than $3.5$, while that for $u$
is essentially four. It further suggests that $\|\eta - \eta_{h}\|_{1} = O(h^{2.5})$,
$\|u-u_{h}\|_{1} = O(h^{3})$ (agreeing with the second estimate of $(i)$ of Theorem $3.1$),
$\|\eta - \eta_{h}\|_{\infty} = O(h^{3})$ and $\|u-u_{h}\|_{\infty}=O(h^{4})$. We also mention
that the $W^{1,\infty}$ orders of convergence (not shown here) were approximately equal to
$2.2$ for $\eta$ and $3$ for $u$, and that the convergence rates from a similar experiment
with $(\ref{cb})$ were practically the same.
In Figure $3.1$ we plot the quantity $\kappa : = \|\eta - \eta_{h}\|/(h^{3.5}\sqrt{\ln 1/h})$ as a
function of $N=1/h$; here $\|\eta - \eta_{h}\|$ are the $L^{2}$-errors from Table \ref{tbl31}. We
observe that $\kappa$ apparently approaches a constant close to $0.13$ as $N$ grows, which
seems to be consistent with the presence of a slow-varying modulation of
$h^{3.5}$ as $h\to0$, such as $(\ln 1/h)^{1/2}$. We close this paragraph with a remark on the `effect of
the boundary' on the error estimates
\begin{figure}[h]
\includegraphics[totalheight=1.8in,width=3.86in]{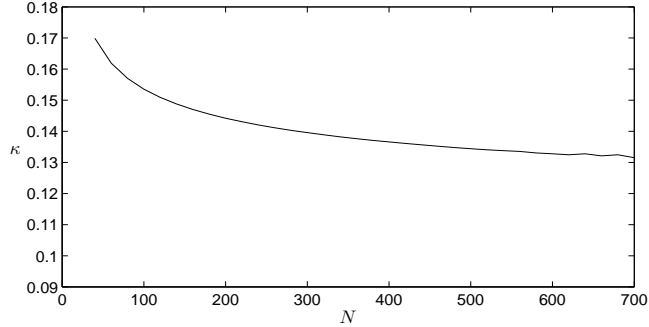}
\caption{$\kappa:=\|\eta - \eta_{h}\|/h^{3.5}\sqrt{\ln 1/h}$ as a function of $N=1/h$;
the $\|\eta - \eta_{h}\|$ are the $L^{2}$ errors from Table \ref{tbl31}.}
\label{fig31}
\end{figure}
of Theorem $3.1$. The proofs of Lemma $3.4$ and Theorem $3.1$ suggest that the accuracy of $\psi$
in Lemma $3.4$ and e.g. of $\|\eta - \eta_{h}\|$ in Theorem $3.1$ degenerates near the boundary
of the interval. This is consistent with the results of the following numerical experiment.
We integrated in time the $(\ref{scb})$ system on $[0,1]$ with suitable right-hand side and
initial conditions so that the wave elevation is given by the travelling Gaussian profile
$\eta(x,t) = 0.5\exp [-144(x - 0.5 - 0.2t)^{2}]$ and the velocity by
$u(x,t) = 6(\sqrt{\eta + 1} - 1)x(x-1)$. (We use cubic splines in space and the explicit
fourth-order RK scheme in time.) The support of the initial $\eta$-profile is effectively
contained in the interval $[0.3,0.5]$ and the wave moves to the right and starts crossing the
boundary at $x=1$ at about $t=1.5$ (see Figure $3.2$).
\captionsetup[subfloat]{labelformat=empty,position=bottom,singlelinecheck=true}
\begin{figure}[h]
  \begin{center}
    \subfloat[(a) $\eta_{h}$ at $t=0.0$]{\includegraphics[scale=0.27]{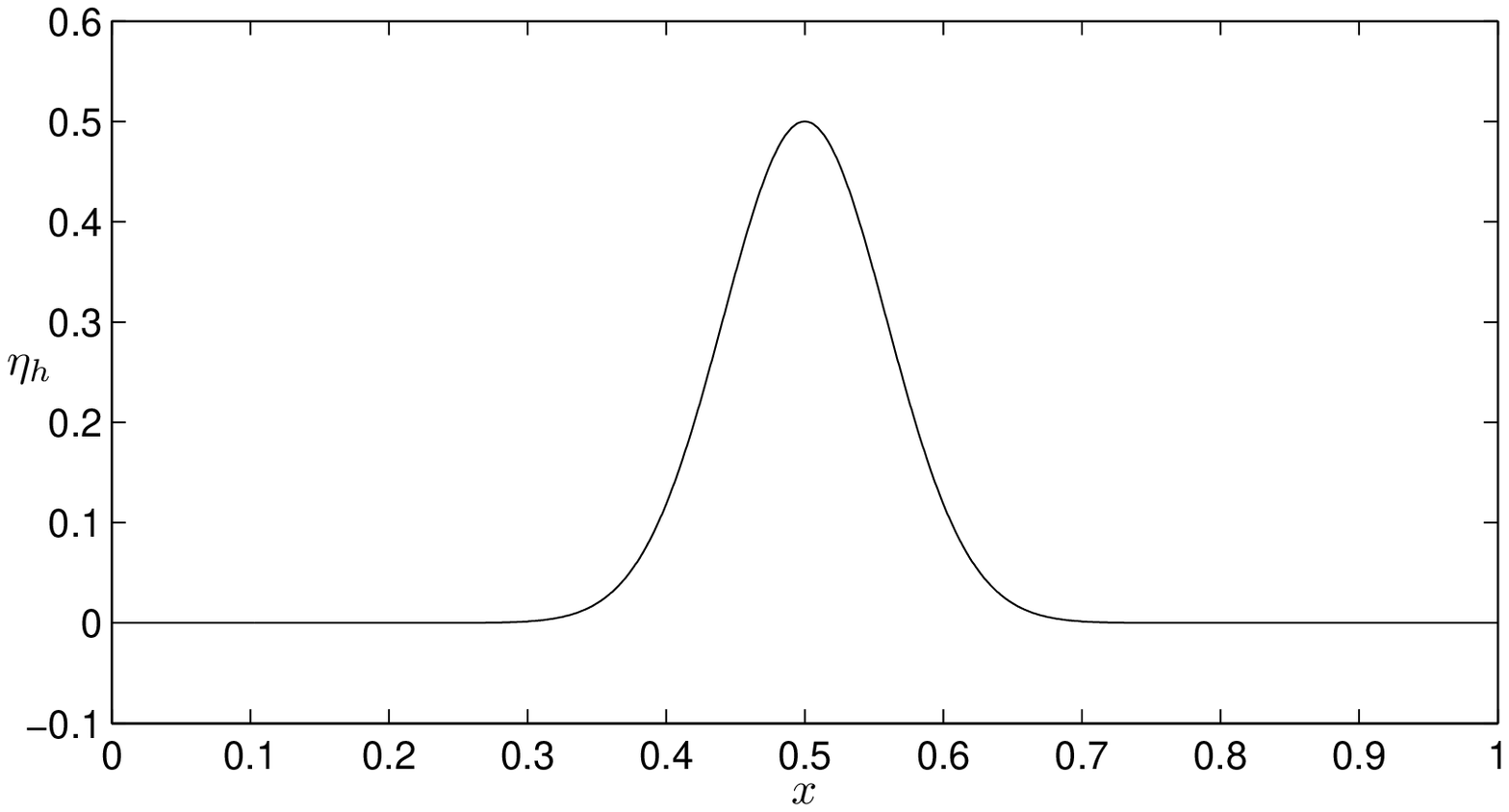}}
    \subfloat[(b) $\eta_{h}$ at $t=0.5$]{\includegraphics[scale=0.27]{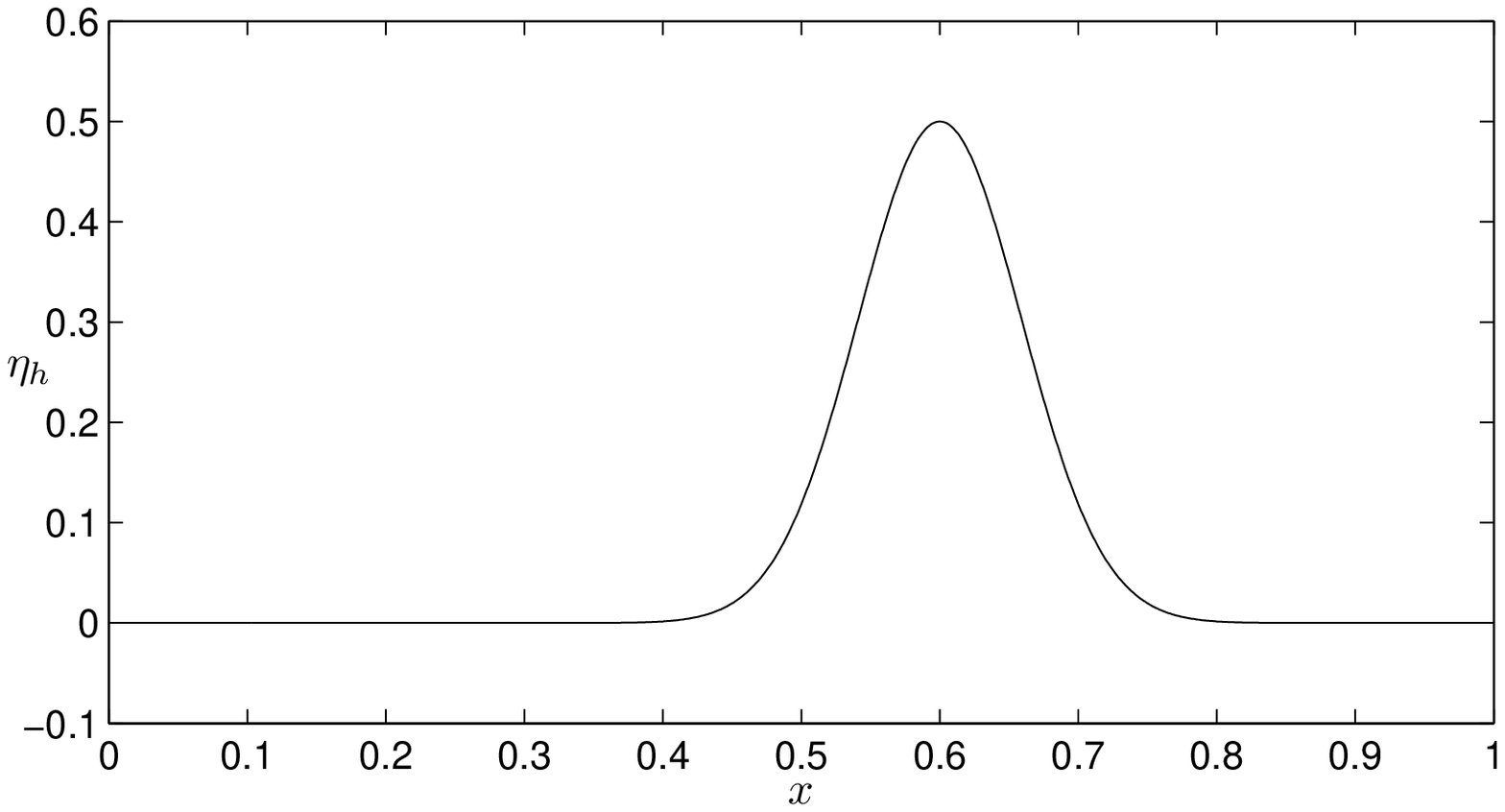}}
    \subfloat[(c) $\eta_{h}$ at $t=1.5$]{\includegraphics[scale=0.27]{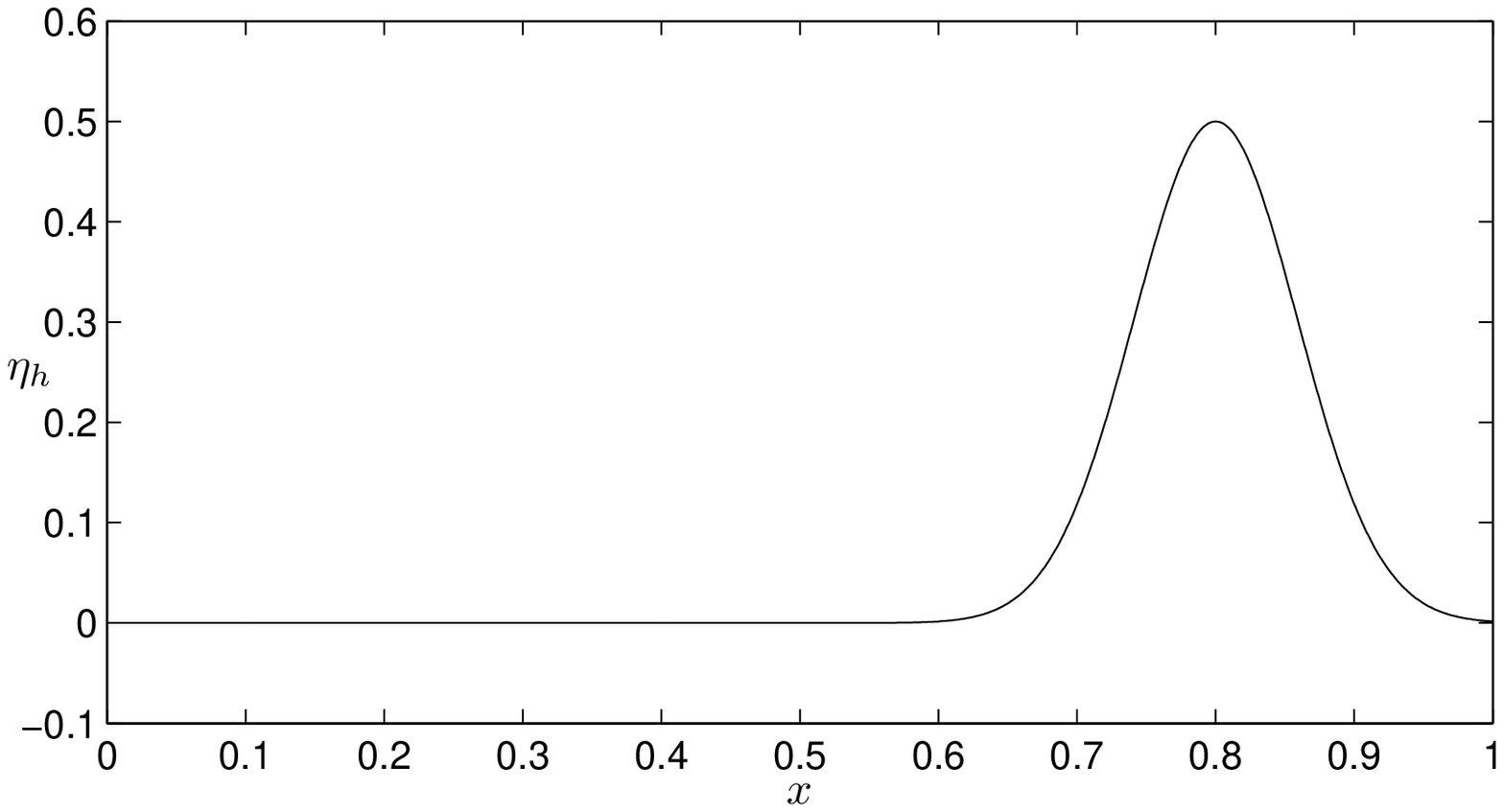}}
  \end{center}
  \caption{Travelling Gaussian $\eta$-profile. Nonhomogeneous $(\ref{scb})$ system.}
  \label{fig32}
\end{figure}
In Table \ref{tbl32} we show the $L^{2}$ errors of $\eta$, as $N=1/h$ increases, at the temporal
instances $t=0.5$, $1.0$, $1.5$, $2.0$ and $2.5$. The rates of convergence are practically
equal to four up to $t=1.5$ but as $\eta$ becomes
nonzero at the boundary they fall to a value consistent with the first inequality of $(i)$
of Theorem $3.1$.
\def\baselinestretch{1}
\scriptsize
\begin{table}[h]
\begin{center}
\begin{tabular}[h]{ | c | c | c | c | c | c | c | c | c | c | c | }\hline
$time$ & \multicolumn{2}{c |}{$0.5$} &
\multicolumn{2}{c |}{$1.0$} & \multicolumn{2}{c |}{$1.5$} & \multicolumn{2}{c |}{$2.0$}
& \multicolumn{2}{c |}{$2.5$} \\ \hline
$N$ &  $L^{2}$-$error$  &  $order$ &  $L^{2}$-$error$ & $order$ &
$L^{2}$-$error$ & $order$ & $L^{2}$-$error$  &  $order$ & $L^{2}$-$error$  &  $order$ \\ \hline
$250$   &  $1.3314(-08)$  &             &  $1.0661(-08)$ &        & $1.3596(-08)$
&       & $1.5924(-08)$  &             & $1.9906(-08)$ &  \\ \hline
$500$   &  $8.2780(-10)$  &  $4.008$ &  $6.6223(-10)$ & $4.009$ & $8.4585(-10)$ & $4.007$ &
$1.0596(-09)$  & $3.910 $ & $1.7594(-09)$  & $3.500$ \\ \hline
$750$   &  $1.6334(-10)$  & $4.003$ &  $1.3067(-10)$ & $4.003$ & $1.6706(-10)$   & $4.000$ &
$2.2223(-10)$ & $3.852$  & $4.2637(-10)$ & $3.496$ \\ \hline
$1000$   &  $5.1617(-11)$  & $4.004$ &  $4.1350(-11)$ & $4.000$ & $5.2838(-11)$ & $4.001$ &
$7.4176(-11)$ & $3.814$  & $1.5595(-10)$ & $3.496$ \\ \hline
$1250$  & $2.1179(-11)$   &  $3.992$ &  $1.6922(-11)$ & $4.004$ & $2.1710(-11)$ & $3.986$ &
$3.1966(-11)$ & $3.772$  & $7.1471(-11)$ & $3.497$ \\ \hline
$1500$  & $1.0287(-11)$  &  $3.961$  &  $8.1703(-12)$ & $3.994$ & $1.0554(-11)$ & $3.956$ &
$1.6213(-11)$ & $3.724$  & $3.7803(-11)$ & $3.493$ \\ \hline
\end{tabular}
\end{center}
\caption{$L^{2}$-errors of $\eta$ and orders of convergence. Example of Figure \ref{fig32}.}
\label{tbl32}
\end{table}
\normalsize
\newpage
\section{Fully discrete schemes}
In this section we turn to the study of some temporal discretizations of the o.d.e. systems
represented by the standard Galerkin spatial discretizations of $(\ref{cb})$ or $(\ref{scb})$,
such as $(\ref{eq22})$ or $(\ref{eq24})$, for example. We shall confine ourselves to
{\em explicit} time stepping schemes in order to avoid the more costly implicit methods
that require solving nonlinear systems of equations at each time step. Of course, with
explicit methods there arises the issue of {\em stability} of the fully discrete schemes.
We will not be exhaustive in our analysis but we will study as examples three simple,
well known explicit Runge-Kutta temporal discretizations, that require, respectively,
stability conditions of the type $k=O(h^{2})$, $k=O(h^{4/3})$, and $k\leq \lambda_{0} h$ for
$\lambda_{0}$ sufficiently small, where $k$ is the time step.
\subsection{The explicit Euler scheme}
Let $M$ be a positive integer, $k=T/M$ denote the (uniform) time step, and put $t^{n}=nk$,
$n=0,1,\dots,M$. We consider the standard Galerkin semidiscretizations with piecewise linear,
continuous functions on a uniform spatial mesh on $[0,1]$ with $h=1/N$, given by the
initial-value problems $(\ref{eq27})$, $(\ref{eq220})$ and $(\ref{eq29})$, $(\ref{eq220})$ in the case
of the $(\ref{cb})$ and the $(\ref{scb})$ systems, respectively. We discretize the systems in time
with the explicit Euler scheme. Hence, we seek for $0\leq n\leq M$ $H_{h}^{n} \in S_{h}^{2}$,
$U_{h}^{n} \in S_{h,0}^{2}$, approximations of the solution $\eta(x,t^{n})$, $u(x,t^{n})$ of the
$(\ref{cb})$ system, such that for $0\leq n \leq M-1$
\begin{equation}
\begin{aligned}
(H_{h}^{n+1}-H_{h}^{n},\phi) & + k(U_{hx}^{n},\phi) + k((H_{h}^{n}U_{h}^{n})_{x}),\phi) = 0
\quad \forall \phi \in S_{h}^{2}\,,\\
a(U_{h}^{n+1}-U_{h}^{n},\chi) & + k(H_{hx}^{n},\chi) + k(U_{h}^{n}U_{hx}^{n},\chi)
= 0 \quad \forall \chi \in S_{h,0}^{2}\,,
\end{aligned}
\label{eq41}
\end{equation}
with
\begin{equation}
H_{h}^{0} = I_{h}\eta^{0}\,, \quad U_{h}^{0} = I_{h,0}u^{0}\,.
\label{eq42}
\end{equation}
The analogous fully discrete approximation of the $(\ref{scb})$ system is defined, for
$0 \leq n\leq M-1$, by
\begin{equation}
\begin{aligned}
(H_{h}^{n+1}-H_{h}^{n},\phi) & + k(U_{hx}^{n},\phi) + \tfrac{k}{2}((H_{h}^{n}U_{h}^{n})_{x}),\phi) = 0
\quad \forall \phi \in S_{h}^{2}\,,\\
a(U_{h}^{n+1}-U_{h}^{n},\chi) & + k(H_{hx}^{n},\chi) + \tfrac{3k}{2}(U_{h}^{n}U_{hx}^{n},\chi)
+ \tfrac{k}{2}(H_{h}^{n}H_{hx}^{n},\chi)= 0 \quad \forall \chi \in S_{h,0}^{2}\,,
\end{aligned}
\label{eq43}
\end{equation}
with
\begin{equation}
H_{h}^{0} = I_{h}\eta^{0}\,, \quad U_{h}^{0} = I_{h,0}u^{0}\,.
\label{eq44}
\end{equation}
Let $A : L^{2}\to S_{h,0}^{2}$, be defined for $f \in L^{2}$ by
\begin{equation}
a(Af,\chi) = (f,\chi) \quad \forall \chi \in S_{h,0}^{2}\,,
\label{eq45}
\end{equation}
i.e. as the discrete solution operator such that $w_{h}=Af$, where $w_{h}$ is the standard Galerkin
approximation in $S_{h,0}^{2}$ of the solution of the two-point bvp $-\tfrac{1}{3}w''+w=f$\,,
$0\leq x\leq 1$, $w(0)=w(1)=0$. From $(\ref{eq45})$ we have immediately that
\begin{equation}
\|Af\|_{1} \leq C \|f\|_{-1}\,,
\label{eq46}
\end{equation}
where the $\|\cdot\|_{-1}$ norm is defined for $f\in L^{2}$ by
\[
\|f\|_{-1} = \sup_{\substack{g\in H_{0}^{1}\\ g\ne 0}}\frac{(f,g)}{\|g\|_{1}}\,.
\]
With this notation in place and letting as before $P$ denote the $L^{2}$ projection operator
onto $S_{h}^{2}$, we may rewrite the fully discrete scheme $(\ref{eq41})$, $(\ref{eq42})$ for
$(\ref{cb})$ as
\begin{equation}
\begin{aligned}
H_{h}^{n+1}-H_{h}^{n} & + kP U_{hx}^{n} + kP(H_{h}^{n}U_{h}^{n})_{x} = 0\,, \\
U_{h}^{n+1}-U_{h}^{n} & + kA H_{hx}^{n} + kA(U_{h}^{n}U_{hx}^{n})= 0\,,
\end{aligned}
\label{eq47}
\end{equation}
for $0\leq n\leq M-1$, with $H_{h}^{0} = I_{h}\eta^{0}$,\,\, $U_{h}^{0} = I_{h,0}u^{0}$.
Similarly, for $(\ref{scb})$ we have from $(\ref{eq43})$, $(\ref{eq44})$
\begin{equation}
\begin{aligned}
H_{h}^{n+1}-H_{h}^{n} & + kPU_{hx}^{n} + \tfrac{k}{2}P(H_{h}^{n}U_{h}^{n})_{x} = 0\,,\\
U_{h}^{n+1}-U_{h}^{n} & + kA H_{hx}^{n} + \tfrac{3k}{2}A(U_{h}^{n}U_{hx}^{n})
+ \tfrac{k}{2}A(H_{h}^{n}H_{hx}^{n}) = 0\,,
\end{aligned}
\label{eq48}
\end{equation}
for $0\leq n\leq M-1$, with $H_{h}^{0}=I_{h}\eta^{0}$, $U_{h}^{0}=I_{h,0}u^{0}$.
We will prove error estimates for the schemes $(\ref{eq47})$ and $(\ref{eq48})$
by comparing $H_{h}^{n}$ with $I_{h}\eta(t^{n})$ and $U_{h}^{n}$ with $I_{h,0}u(t^{n})$.
For this purpose, it is useful to establish the following estimates of the
{\em truncation errors} of the interpolants. (In the sequel we shall analyze mainly the
approximation of the $(\ref{scb})$ system; the analogous results for $(\ref{cb})$ follow as
in Sections $2$ and $3$. Frequently, we shall suppress the $x$ variable, denoting e.g.
$\eta(\cdot,t)$ by $\eta(t)$ etc.)
\begin{lemma} Suppose that the solution $(\eta,u)$ of $(\ref{scb})$ is sufficiently
smooth in $[0,T]$. Let $H(t)=I_{h}\eta(t)$, $U(t)=I_{h,0}u(t)$, and define
$\psi(t) \in S_{h}^{2}$, $\zeta(t) \in S_{h,0}^{2}$ for $0\leq t\leq T$ by
\begin{align}
H_{t} + PU_{x} + \tfrac{1}{2}P(HU)_{x} & = \psi\,, \label{eq49} \\
U_{t} + AH_{x} + \tfrac{3}{2}A(UU_{x}) + \tfrac{1}{2}A(HH_{x})& = \zeta\,. \label{eq410}
\end{align}
Then
\begin{equation}
\begin{aligned}
\|\psi\| & \leq Ch^{3/2}\,, \quad \|\psi_{t}\| \leq Ch^{3/2}\,, \\
\|\zeta\|_{1} &\leq Ch^{2}\,, \quad \,\,\,\,\|\zeta_{t}\|_{1}\leq Ch^{2}\,,
\end{aligned}
\label{eq411}
\end{equation}
hold for $0\leq t\leq T$. An analogous result holds for $(\ref{cb})$.
\end{lemma}
\begin{proof}
Subtracting the equations $P(\eta_{t} + u_{x}+\tfrac{1}{2}(u\eta)_{x})=0$ and
$(\ref{eq49})$, and putting $\rho:=\eta-I_{h}\eta$, $\sigma:=u-I_{h,0}u$, we
obtain
\[
P\bigl( \rho_{t} + [(1+\tfrac{1}{2}\eta)\sigma]_{x} + \tfrac{1}{2}(u\rho)_{x}
-\tfrac{1}{2}(\rho\sigma)_{x}\bigr) = -\psi\,.
\]
Therefore, using the approximation properties of $S_{h}^{2}$ and $S_{h,0}^{2}$, and
Lemma $2.2$, we have
\[
\|\psi\| \leq \|\rho_{t}\| + \|P[(1+\tfrac{1}{2}\eta)\sigma]_{x}\|
+ \tfrac{1}{2}\|P(u\rho)_{x}\| + \|(\rho\sigma)_{x}\|
\leq C(h^{2} + h^{3/2} + h^{2} + h^{3}) \leq Ch^{3/2}\,.
\]
Similarly, since e.g. by Lemma $2.2$
\[
\|P[(1+\tfrac{1}{2}\eta)\sigma]_{xt}\| \leq \|P(\tfrac{1}{2}\eta_{t}\sigma)_{x}\|
+ \|P[(1+\tfrac{1}{2}\eta)\sigma_{t}]_{x}\| \leq Ch^{3/2}\,,
\]
we have
\[
\|\psi_{t}\| \leq \|\rho_{t}\| + \|P[(1+\tfrac{1}{2}\eta)\sigma]_{xt}\|
+ \tfrac{1}{2}\|P(u\rho)_{xt}\| + \|(\rho\sigma)_{xt}\|
\leq C(h^{2} + h^{3/2} + h^{2} + h^{3}) \leq Ch^{3/2}\,.
\]
Note now that for any $\chi \in S_{h,0}^{2}$, $(\ref{eq45})$ and the fact that
$(\sigma_{t}',\chi')=0$ yield
\[
a(A(u_{t}-\tfrac{1}{3}u_{txx})-U_{t},\chi)=(u_{t},\chi)+\tfrac{1}{3}(u_{tx},\chi_{x})
-a(U_{t},\chi) = a(\sigma_{t},\chi)=(\sigma_{t},\chi)=a(A\sigma_{t},\chi)\,.
\]
Hence
\[
A(u_{t} - \tfrac{1}{3}u_{txx}) - U_{t} = A\sigma_{t}\,,
\]
which implies, in view of the second p.d.e. of $(\ref{scb})$ that
\[
A\sigma_{t} + U_{t} + A(\eta_{x} + \tfrac{3}{2}uu_{x} + \tfrac{1}{2}\eta\eta_{x})=0\,.
\]
Subtracting now this equation from $(\ref{eq410})$ we see, after some algebra, that
\begin{equation}
A\bigl( \sigma_{t} + \rho_{x} + \tfrac{3}{2}[(u\sigma)_{x} - \sigma\sigma_{x}]
+ \tfrac{1}{2}[(\eta\rho)_{x} - \rho\rho_{x}]\bigr) = -\zeta\,.
\label{eq412}
\end{equation}
Therefore, using $(\ref{eq46})$ and the approximation properties of $S_{h}^{2}$ and
$S_{h,0}^{2}$, we obtain
\begin{align*}
\|\zeta\|_{1} & \leq C(\|\sigma_{t}\|_{-1} + \|\rho_{x}\|_{-1}
+ \|(u\sigma-\tfrac{1}{2}\sigma^{2})_{x}\|_{-1} + \|(\eta\rho - \tfrac{1}{2}\rho^{2})_{x}\|_{-1}\\
& \leq C(\|\sigma_{t}\| + \|\rho\| + \|u\sigma\| + \|\sigma^{2}\| + \|\eta\rho\|
+ \|\rho^{2}\|) \leq Ch^{2}\,.
\end{align*}
Similarly, after differentiating $(\ref{eq412})$ with respect to $t$, we see that
\[
\|\zeta_{t}\|_{1} \leq Ch^{2}\,,
\]
thus ending the proof. The same results hold for $(\ref{cb})$ of course.
\end{proof}
We now proceed to prove error estimates for the explicit Euler-Galerkin schemes $(\ref{eq47})$
and $(\ref{eq48})$. We begin by a consistency result.
\begin{lemma} Suppose that the solution $(\eta,u)$ of $(\ref{scb})$ is sufficiently smooth.
Let $H^{n}:=H(t^{n})=I_{h}\eta(t^{n})$, $U^{n}=U(t^{n})=I_{h,0}u(t^{n})$, and define,
for $0\leq n\leq M-1$, $\delta_{1}^{n}$ and $\delta_{2}^{n}$ by the equations
\begin{align*}
\delta_{1}^{n} & := H^{n+1} - H^{n} + kPU_{x}^{n} + \tfrac{k}{2}P(H^{n}U^{n})_{x}\,, \\
\delta_{2}^{n} & := U^{n+1} - U^{n} + kAH_{x}^{n} + \tfrac{3k}{2}A(U^{n}U_{x}^{n})
+ \tfrac{k}{2}A(H^{n}H_{x}^{n})\,.
\end{align*}
Then
\[
\max_{0\leq n\leq M-1} (\|\delta_{1}^{n}\| + \|\delta_{2}^{n}\|_{1})
\leq C k (k + h^{3/2})\,.
\]
The analogous result holds for $(\ref{cb})$ as well.
\end{lemma}
\begin{proof}
Using $(\ref{eq49})$ and $(\ref{eq410})$ we have, with $\psi^{n}=\psi(t^{n})$,
$\zeta^{n} = \zeta(t^{n})$, that
$\delta_{1}^{n} = H^{n+1} - H^{n} -kH_{t}^{n} + k\psi^{n}$,
$\delta_{2}^{n} = U^{n+1} - U^{n} -kU_{t}^{n}+k\zeta^{n}$. Hence, for $0\leq n\leq M-1$,
\begin{align*}
\|\delta_{1}^{n}\| + \|\delta_{2}^{n}\|_{1} & \leq \|H^{n+1}-H^{n}-kH_{t}^{n}\|
+ k\|\psi^{n}\| + \|U^{n+1}-U^{n}-kU_{t}^{n}\|_{1} + k\|\zeta^{n}\|_{1}\\
& \leq C(k^{2} + kh^{3/2} + k^{2} + kh^{2}) \leq Ck(k + h^{3/2})\,,
\end{align*}
by Taylor's theorem and $(\ref{eq411})$.
\end{proof}
\begin{proposition}
Suppose that the solutions $(\eta,u)$ of $(\ref{scb})$ and $(\ref{cb})$ are sufficiently smooth
on $[0,T]$. Then, if $\mu = k/h^{2}$, there is a constant $C=C(\mu)$, which is an increasing
continuous function of $\mu$, such that
\begin{equation}
\max_{0\leq  n\leq M} \|H_{h}^{n} - \eta(t^{n})\| \leq C(k + h^{3/2})\,,
\quad \max_{0\leq n\leq M}\|U_{h}^{n} - u(t^{n})\|_{1} \leq C(k+h)\,,
\label{eq413}
\end{equation}
where $(H_{h}^{n},U_{h}^{n})$ satisfy $(\ref{eq47})$ or $(\ref{eq48})$ as the case may be.
\end{proposition}
\begin{proof}
Consider the case of $(\ref{scb})$. We use the notation of Lemmas $4.1$ and $4.2$ and put
$\ve^{n} = H^{n} - H_{h}^{n}$, $e^{n} = U^{n} - U_{h}^{n}$. Using the definition of
$\delta_{1}^{n}$, $\delta_{2}^{n}$, and $(\ref{eq48})$ we obtain, after some straightforward
computations, that for $0\leq n\leq M-1$
\begin{align}
\ve^{n+1} & = \ve^{n} -kP[e_{x}^{n} + \tfrac{1}{2}(H^{n}e^{n})_{x}
- \tfrac{1}{2}(\ve^{n}e^{n})_{x} + \tfrac{1}{2}(U^{n}\ve^{n})_{x}] + \delta_{1}^{n}\,,
\label{eq414} \\
e^{n+1} & = e^{n} - kA[\ve_{x}^{n} + \tfrac{3}{2}(U^{n}e^{n})_{x} - \tfrac{3}{2}(e^{n}e_{x}^{n})
+ \tfrac{1}{2}(H^{n}\ve^{n})_{x} - \tfrac{1}{2}\ve^{n}\ve_{x}^{n}] + \delta_{2}^{n}\,.
\label{eq415}
\end{align}
It follows from $(\ref{eq46})$ that
\begin{align}
\|A(U^{n}e^{n})_{x}\|_{1} & \leq C \|U^{n}e^{n}\| \leq C \|e^{n}\|\,,
\label{eq416} \\
\|A(H^{n}\ve^{n})_{x}\|_{1} & \leq C \|H^{n}\ve^{n}\| \leq C \|\ve^{n}\|\,.
\label{eq417}
\end{align}
Let now $0\leq n^{*}\leq M-1$ be the maximal integer such that
\begin{equation}
\|\ve^{n}\|_{1} + \|e^{n}\|_{1} \leq 1\,, \quad 0\leq n\leq n^{*}\,.
\label{eq418}
\end{equation}
Then, for $0\leq n\leq n^{*}$, using $(\ref{eq46})$, we have
\begin{equation}
\begin{aligned}
\|P\bigl( e_{x}^{n} + \tfrac{1}{2}(H^{n}e^{n})_{x} - \tfrac{1}{2}(\ve^{n}e^{n})_{x}\bigr)\|
&\leq \|e_{x}^{n}\| + \tfrac{1}{2}\|(H^{n}e^{n})_{x}\| +\tfrac{1}{2}\|(\ve^{n}e^{n})_{x}\| \\
& \leq C(\|e^{n}\|_{1} + \|e^{n}\|_{1}\|\ve^{n}\|_{1}) \leq C \|e^{n}\|_{1}\,, \label{eq419}
\end{aligned}
\end{equation}
\begin{equation}
\|A(e^{n}e_{x}^{n})\|_{1} + \|A(\ve^{n}\ve_{x}^{n})\|_{1} \leq
C(\|(e^{n})^{2}\| + \|(\ve^{n})^{2}\|) \leq C(\|e^{n}\| + \|\ve^{n}\|)\,.
\label{eq420}
\end{equation}
Using $(\ref{eq416})$-$(\ref{eq420})$ in $(\ref{eq414})$ and $(\ref{eq415})$ we have,
for $0\leq n\leq n^{*}$
\begin{align}
\|\ve^{n+1}\| & \leq \|\ve^{n} - \tfrac{k}{2}P(U^{n}\ve^{n})_{x}\|
+ Ck\|e^{n}\|_{1} + \|\delta_{1}^{n}\|\,, \label{eq421}\\
\|e^{n+1}\|_{1} & \leq \|e^{n}\|_{1} + Ck(\|\ve^{n}\| + \|e^{n}\|) + \|\delta_{2}^{n}\|_{1}\,.
\label{eq422}
\end{align}
Now
\begin{equation}
\|\ve^{n} - \tfrac{k}{2}P(U^{n}\ve^{n})_{x}\|^{2} =
\|\ve^{n}\|^{2} + \tfrac{k^{2}}{4}\|P(U^{n}\ve^{n})_{x}\|^{2}
-k(\ve^{n},P(U^{n}\ve^{n})_{x})\,. \label{eq423}
\end{equation}
But, by inverse assumptions
\[
\|P(U^{n}\ve^{n})_{x}\| \leq \|(U^{n}\ve^{n})_{x}\| \leq \tfrac{C}{h}\|\ve^{n}\|\,.
\]
In addition,
\[
|(\ve^{n},P(U^{n}\ve^{n})_{x})| = |(\ve^{n},(U^{n}\ve^{n})_{x})|
= \tfrac{1}{2}|(U_{x}^{n}\ve^{n},\ve^{n})| \leq C\|\ve^{n}\|^{2}\,.
\]
Hence, $(\ref{eq423})$ yields
\[
\|\ve^{n} - \tfrac{k}{2}P(U^{n}\ve^{n})_{x}\|^{2} \leq \|\ve^{n}\|^{2}
+ C\mu k\|\ve^{n}\|^{2} + Ck\|\ve^{n}\|^{2}\,,
\]
i.e. that
\[
\|\ve^{n} - \tfrac{k}{2}P(U^{n}\ve^{n})_{x}\| \leq (1 + C_{\mu}k)\|\ve^{n}\|\,,
\]
where $C_{\mu}$ is a generic polynomial in $\mu$ of degree one. Therefore $(\ref{eq421})$
becomes
\begin{equation}
\|\ve^{n+1}\| \leq \|\ve^{n}\| + C_{\mu}k(\|\ve^{n}\| + \|e^{n}\|_{1})
+ \|\delta_{1}^{n}\|\,.
\label{eq424}
\end{equation}
Hence, by $(\ref{eq422})$, $(\ref{eq424})$, and Lemma $4.2$ we obtain for $0\leq n\leq n^{*}$
\[
\|\ve^{n+1}\| + \|e^{n+1}\|_{1} \leq \|\ve^{n}\| + \|e^{n}\|_{1}
+ C_{\mu}k(\|\ve^{n}\| + \|e^{n}\|_{1}) + Ck(k+h^{3/2})\,.
\]
By Gronwall's Lemma we conclude that
\begin{equation}
\|\ve^{n}\| + \|e^{n}\|_{1} \leq C(\mu,T)(k+h^{3/2})\,, \quad
0\leq n\leq n^{*}+1\,, \label{eq425}
\end{equation}
where $C(\mu,T)=\exp((C_{0}+C_{1}\mu)T)$. Taking $h$ sufficiently small, we see from the
maximality property $(\ref{eq418})$ of $n^{*}$ that we may take $n^{*}=M-1$, and the conclusion
of the proposition follows from $(\ref{eq21})$. The case of $(\ref{cb})$ is entirely similar.
\end{proof}
\begin{remark} The estimate $(\ref{eq425})$ and Sobolev's inequality imply that
$\|e^{n}\|_{\infty} = O(k+h^{3/2})$. Therefore,
$\max_{n}\|u(t^{n})-U_{h}^{n}\|_{\infty} = O(k + h^{3/2})$.
\end{remark}
\begin{remark} Consider the {\em linearized} problem $(\ref{eq243})$. In this case, the analogous
fully discrete scheme is
\begin{align*}
H_{h}^{n+1} - H_{h}^{n} + kPU_{hx}^{n} & = 0\,, \\
U_{h}^{n+1} - U_{h}^{n} + kAH_{hx}^{n} & = 0\,,
\end{align*}
for $0\leq n\leq M-1$, with $H_{h}^{0}=I_{h}\eta^{0}$, $U_{h}^{0}=I_{h,0}u^{0}$.
Consequently, using the notation of the proof of Proposition $4.1$, we now have for
$0\leq n\leq M-1$ the error equations
\begin{align*}
\ve^{n+1} & = \ve^{n} - kPe_{x}^{n} + \delta_{1}^{n}\,,\\
e^{n+1} & = e^{n} -kA\ve_{x}^{n} + \delta_{2}^{n}\,,
\end{align*}
from which there easily follows the estimate
\[
\|\ve^{n}\| + \|e^{n}\|_{1} \leq C(k+h^{3/2})\,, \quad 0\leq n\leq M\,,
\]
and the conclusions of the analog of Proposition $4.1$, without the stability restriction
$k=O(h^{2})$. In other words, the linearized system {\em is not stiff}. This may also be verified
by examining the spectrum of the spatial discretization operator of the semidiscrete linearized
system: The latter may be written for $0\leq t\leq T$ in the form
\begin{align*}
\eta_{ht} + L_{h}u_{h} & = 0\,, \\
M_{h}u_{ht} + \widetilde{L}_{h}\eta_{h} & = 0\,,
\end{align*}
where the operators $L_{h} : S_{h,0}^{2}\to S_{h}^{2}$,
$\widetilde{L}_{h} : S_{h}^{2}\to S_{h,0}^{2}$,
$M_{h} : S_{h,0}^{2}\to S_{h,0}^{2}$ are defined by the equations
\begin{align*}
(L_{h}\psi,\phi) & = (\psi_{x},\phi) \quad \forall \psi \in S_{h,0}^{2}, \phi \in S_{h}^{2}\,, \\
(\widetilde{L}_{h}\phi,\psi) & =
(\phi_{x},\psi) \quad \forall \phi \in S_{h}^{2}, \psi \in S_{h,0}^{2}\,, \\
(M_{h}\psi,\chi) & = a(\psi,\chi) \quad \forall \psi, \chi \in S_{h,0}^{2}\,.
\end{align*}
Hence, the semidiscrete system may be written on $S_{h}^{2} \times S_{h,0}^{2}$ as
\[
A_{h} W_{ht} + B_{h}W_{h} = 0\,,
\]
where $W_{h} = [\eta_{h}, u_{h}]^{T}$, and
\[
A_{h} =
\begin{bmatrix}
I & 0 \\
0 & M_{h}
\end{bmatrix}
\,, \quad B_{h} =
\begin{bmatrix}
0 & L_{h} \\
\widetilde{L}_{h} & 0
\end{bmatrix}
\,.
\]
Therefore, the spectrum of the spatial discretization operator coincides with that of the
generalized eigenvalue problem
\begin{equation}
B_{h}V_{h} = -\lambda_{h}A_{h}V_{h}\,.
\label{eq426}
\end{equation}
Denoting the eigenfunctions as $V_{h} = [H_{h}, U_{h}]^{T}$, where $H_{h}$, $U_{h}$ are elements
of $S_{h}^{2}$, $S_{h,0}^{2}$, respectively, regarded as vector spaces over
$\mathbb{C}$, we have $L_{h}U_{h}=-\lambda_{h}H_{h}$,
$\widetilde{L}_{h}H_{h} = -\lambda_{h}M_{h}U_{h}$, from which
\begin{equation}
\begin{aligned}
(L_{h}U_{h},H_{h}) & = -\lambda_{h}\|H_{h}\|^{2}\,, \\
(\widetilde{L}_{h}H_{h},U_{h}) & = -\lambda_{h}(M_{h}U_{h},U_{h})\,,
\end{aligned}
\label{eq427}
\end{equation}
where the $L^{2}$ inner product for complex valued functions is defined as
$(f,g) = \int_{0}^{1}f(x)\overline{g(x)}dx$. Now
\[
(L_{h}U_{h},H_{h}) = (U_{hx},H_{h})=-(U_{h},H_{hx}) = -\overline{(\widetilde{L}_{h}H_{h},U_{h})}\,.
\]
Therefore, from $(\ref{eq427})$
\begin{equation}
\lambda_{h}=\frac{-2i \text{Im}(U_{hx},H_{h})}{\|H_{h}\|^{2} + a(U_{h},U_{h})}\,.
\label{eq428}
\end{equation}
We conclude that the spectrum of $(\ref{eq426})$ consists of purely imaginary eigenvalues
(and is symmetric about the origin as $-\lambda_{h}$ is also an eigenvalue corresponding to
the eigenvector $\overline{V}\!_{h}$). Moreover, it follows from $(\ref{eq428})$ that
\[
|\lambda_{h}| \leq \frac{2\|U_{hx}\|\|H_{h}\|}{\|H_{h}\|^{2} + a(U_{h},U_{h})} \leq C\,,
\]
i.e. that the spectrum is bounded by a constant $C$ independent of $h$, implying that the
linearized semidiscrete problem is not stiff. \par
As we saw in Proposition $4.1$, when the nonlinear $(\ref{cb})$ or $(\ref{scb})$ system is
discetized by the explicit Euler-Galerkin method, the mesh condition $k=O(h^{2})$ is sufficient
for stability. In a numerical experiment, we solved the nonlinear, nonhomogeneous $(\ref{cb})$
system using as exact solution $\eta(x,t)=\exp(2t)(\cos(\pi x) + x + 2)$,
$u(x,t) = \exp(xt)(\sin(\pi x) + x^{3} - x^{2})$ for $x \in [0,1]$, and discretizing by the explicit
Euler-standard Galerkin method with piecewise linear functions and fixed $N=1/h=400$. When we
integrated up to $T=1$ using $k=h^{2}$, we obtained an $L^{2}$ error for $\eta$ that was
approximately equal to $2.2090(-4)$. The accuracy degenerated when we took $k=h^{\alpha}$
with decreasing $\alpha < 2$. For example, at the temporal instance closest to $T=1$ the
computations yielded the following magnitudes of the $L^{2}$ errors of $\eta$:
\begin{itemize}
\item[] $k=h^{1.8}$ \qquad $3.8839(-4)$
\item[] $k=h^{1.6}$ \qquad $1.1257(-3)$
\item[] $k=h^{1.4}$ \qquad $3.6917(-3)$
\item[] $k=h^{1.2}$ \qquad overflow at about $t=0.8$\,.
\end{itemize}
\end{remark}
\subsection{The improved Euler method} We next study the temporal discretization of the
initial-value problems $(\ref{eq27})$, $(\ref{eq220})$ and $(\ref{eq29})$, $(\ref{eq220})$ by
the explicit, second-order accurate `improved Euler' scheme, that may be written in the case
of the o.d.e. $y'=f(t,y)$ in the two-step form
\begin{align*}
y^{n,1} & = y^{n} + \tfrac{k}{2}f(t^{n},y^{n})\,, \\
y^{n+1} & = y^{n} + k f(t^{n}+\tfrac{k}{2},y^{n,1})\,.
\end{align*}
Using notation analogous to that established in the previous paragraph, in the case of
$(\ref{cb})$ we seek for $0\leq n\leq M$ $H_{h}^{n}\in S_{h}^{2}$, $U_{h}^{n}\in S_{h,0}^{2}$,
and for $0\leq n\leq M-1$ $H_{h}^{n,1}\in S_{h}^{2}$, $U_{h}^{n,1}\in S_{h,0}^{2}$, such that
\begin{equation}
\begin{aligned}
H_{h}^{n,1} & - H_{h}^{n} + \tfrac{k}{2}PU_{hx}^{n} + \tfrac{k}{2}P(H_{h}^{n}U_{h}^{n})_{x}=0\,,\\
U_{h}^{n,1} & - U_{h}^{n} + \tfrac{k}{2}AH_{hx}^{n} + \tfrac{k}{2}A(U_{h}^{n}U_{hx}^{n}) = 0\,, \\
H_{h}^{n+1} & - H_{h}^{n} + kPU_{hx}^{n,1} + kP(H_{h}^{n,1}U_{h}^{n,1})_{x}=0\,, \\
U_{h}^{n+1} & - U_{h}^{n} + kAH_{hx}^{n,1} + kA(U_{h}^{n,1}U_{hx}^{n,1}) = 0\,,
\end{aligned}
\label{eq429}
\end{equation}
for $0\leq n\leq M-1$, with $H_{h}^{0}=I_{h}\eta^{0}$, $U_{h}^{0}=I_{h,0}u^{0}$. In the case
of $(\ref{scb})$ the analogous equations are
\begin{equation}
\begin{aligned}
H_{h}^{n,1} & - H_{h}^{n} + \tfrac{k}{2}PU_{hx}^{n} + \tfrac{k}{4}P(H_{h}^{n}U_{h}^{n})_{x}=0\,,\\
U_{h}^{n,1} & - U_{h}^{n} + \tfrac{k}{2}AH_{hx}^{n} + \tfrac{3k}{4}A(U_{h}^{n}U_{hx}^{n})
+ \tfrac{k}{4}A(H_{h}^{n}H_{hx}^{n}) = 0\,, \\
H_{h}^{n+1} & - H_{h}^{n} + kPU_{hx}^{n,1} + \tfrac{k}{2}P(H_{h}^{n,1}U_{h}^{n,1})_{x}=0\,, \\
U_{h}^{n+1} & - U_{h}^{n} + kAH_{hx}^{n,1} + \tfrac{3k}{2}A(U_{h}^{n,1}U_{hx}^{n,1})
+ \tfrac{k}{2}A(H_{h}^{n,1}H_{hx}^{n,1})= 0\,,
\end{aligned}
\label{eq430}
\end{equation}
for $0\leq n\leq M-1$, with $H_{h}^{0}=I_{h}\eta^{0}$, $U_{h}^{0}=I_{h,0}u^{0}$. In order to study
the consistency and convergence of the schemes, we let again
$H^{n} = H(t^{n})=I_{h}\eta(t^{n})$, $U^{n}=U(t^{n})=I_{h,0}u(t^{n})$, where
$(\eta, u)$ is the solution of $(\ref{cb})$ or $(\ref{scb})$, and define, in the case of
$(\ref{scb})$, $(H^{n,1}, U^{n,1}) \in S_{h}^{2}\times S_{h,0}^{2}$ for $0\leq n\leq M-1$
by the equations
\begin{equation}
\begin{aligned}
H^{n,1} & - H^{n} + \tfrac{k}{2}PU_{x}^{n} + \tfrac{k}{4}P(H^{n}U^{n})_{x} = 0\,, \\
U^{n,1} & - U^{n} + \tfrac{k}{2}AH_{x}^{n} + \tfrac{3k}{4}A(U^{n}U_{x}^{n})
+ \tfrac{k}{4}A(H^{n}H_{x}^{n}) = 0\,.
\end{aligned}
\label{eq431}
\end{equation}
In the case of $(\ref{cb})$ $H^{n,1}$, $U^{n,1}$ are defined analogously. Our consistency
result is:
\begin{lemma} Suppose that the solution $(\eta, u)$ of $(\ref{scb})$ is sufficiently smooth
and let $\lambda=k/h$. Define, for $0\leq n\leq M-1$, $\delta_{1}^{n}$, $\delta_{2}^{n}$
by the equations
\begin{align}
\delta_{1}^{n} & = H^{n+1} - H^{n} + kPU_{x}^{n,1} + \tfrac{k}{2}P(H^{n,1}U^{n,1})_{x}\,,
\label{eq432} \\
\delta_{2}^{n} & = U^{n+1} - U^{n} + kAH_{x}^{n,1} + \tfrac{3k}{2}A(U^{n,1}U_{x}^{n,1})
+ \tfrac{k}{2}A(H^{n,1}H_{x}^{n,1})\,.
\label{eq433}
\end{align}
Then, there exists a constant $C_{1}=C_{1}(\lambda)$, which is a polynomial of $\lambda$ of degree
one, such that
\[
\max_{0\leq n\leq M-1}(\|\delta_{1}^{n}\| + \|\delta_{2}^{n}\|_{1})\leq C_{1}k(k^{2}+h^{3/2})\,.
\]
The analogous result holds for $(\ref{cb})$ as well.
\end{lemma}
\begin{proof} Let $0\leq n\leq M-1$. By $(\ref{eq431})$, $(\ref{eq49})$ and $(\ref{eq410})$ we have
\begin{equation}
H^{n,1} = H^{n} + \tfrac{k}{2}H_{t}^{n} - \tfrac{k}{2}\psi^{n}\,, \quad
U^{n,1} = U^{n} + \tfrac{k}{2}U_{t}^{n} - \tfrac{k}{2}\zeta^{n}\,.
\label{eq434}
\end{equation}
From these expressions, after some algebra, we obtain
\[
H^{n,1}U^{n,1} = H^{n}U^{n} + \tfrac{k}{2}(HU)_{t}^{n} + w_{1}^{n}\,,
\]
where
\begin{equation}
w_{1}^{n}:=\tfrac{k^{2}}{4}H_{t}^{n}U_{t}^{n} - \tfrac{k}{2}(U^{n}+\tfrac{k}{2}U_{t}^{n})\psi^{n}
-\tfrac{k}{2}(H^{n}+\tfrac{k}{2}H_{t}^{n})\zeta^{n} + \tfrac{k^{2}}{4}\psi^{n}\zeta^{n}\,.
\label{eq435}
\end{equation}
Hence, by $(\ref{eq432})$, $(\ref{eq434})$, $(\ref{eq49})$, and the above we obtain
\begin{equation}
\delta_{1}^{n}=H^{n+1} - H^{n} - kH_{t}^{n} - \tfrac{k^{2}}{2}H_{tt}^{n} + k\psi^{n}
+ \tfrac{k^{2}}{2}\psi_{t}^{n} - \tfrac{k^{2}}{2}P\zeta_{x}^{n} + \tfrac{k}{2}Pw_{1x}^{n}\,.
\label{eq436}
\end{equation}
Now, $(\ref{eq435})$, in view of $(\ref{eq411})$ and the approximation and inverse properties
of $S_{h}^{2}$ and $S_{h,0}^{2}$, gives
\begin{align*}
\|w_{1}^{n}\|_{1} & \leq C(k^{2}\|H_{t}^{n}\|_{1} \|U_{t}^{n}\|_{1}
+ k\|\psi^{n}\|_{1}(\|U^{n}\|_{1} + k\|U_{t}^{n}\|_{1}) \\
& \qquad
+ k\|\zeta^{n}\|_{1}(\|H^{n}\|_{1} + k\|H_{t}^{n}\|_{1}) + k^{2}\|\psi^{n}\|_{1}\|\zeta^{n}\|_{1})\\
& \leq c(k^{2} + kh^{-1}h^{3/2}(1+ck) + k h^{2}(1+ck) + k^{2}h^{-1}h^{7/2}) \\
& \leq c(k^{2} + \lambda h^{3/2}).
\end{align*}
Therefore, by Taylor's theorem and $(\ref{eq411})$ we have
\begin{equation}
\|\delta_{1}^{n}\| \leq c(k^{3} + kh^{3/2} + k^{2}h^{3/2} + k^{2}h^{2} + k(k^{2}+\lambda h^{3/2})
\leq C_{1}k(k^{2} + h^{3/2})\,,
\label{eq437}
\end{equation}
where $C_{1}$ is a constant that is a polynomial of $\lambda$ of degree one. (Such constants
will be generically denoted by $C_{1}$ in the sequel of this proof.) In order to estimate
$\|\delta_{2}^{n}\|_{1}$ note that by $(\ref{eq434})$
\begin{equation}
U^{n,1}U_{x}^{n,1} = U^{n}U_{x}^{n} + \tfrac{k}{2}(UU_{x})_{t}^{n} + w_{2}^{n}\,,
\label{eq438}
\end{equation}
where
\[
w_{2}^{n}:= \tfrac{k^{2}}{4}U_{t}^{n}U_{tx}^{n}
- \tfrac{k}{2}\bigl( (U^{n}+\tfrac{k}{2}U_{t}^{n})\zeta^{n}\bigr)_{x}
+ \tfrac{k^{2}}{4}\zeta^{n}\zeta_{x}^{n}\,.
\]
By $(\ref{eq411})$ and the approximation properties of $S_{h,0}^{2}$ we have
\begin{equation}
\|w_{2}^{n}\| \leq C(k^{2} + kh^{2})\,.
\label{eq439}
\end{equation}
Similarly,
\begin{equation}
H^{n,1}H_{x}^{n,1} = H^{n}H_{x}^{n} + \tfrac{k}{2}(HH_{x})_{t}^{n} + w_{3}^{n}\,,
\label{eq440}
\end{equation}
where
\[
w_{3}^{n}:=\tfrac{k^{2}}{4}H_{t}^{n}H_{tx}^{n}
-\tfrac{k}{2}\bigl( (H^{n}+\tfrac{k}{2}H_{t}^{n})\psi^{n}\bigr)_{x}
+ \tfrac{k^{2}}{4}\psi^{n}\psi_{x}^{n}\,.
\]
By $(\ref{eq411})$ and the approximation and inverse properties of $S_{h}^{2}$ we have
\begin{equation}
\|w_{3}^{n}\| \leq C(k^{2} + \lambda h^{3/2})\,.
\label{eq441}
\end{equation}
By $(\ref{eq433})$, $(\ref{eq435})$, $(\ref{eq438})$, and $(\ref{eq440})$, we see now that
\[
\delta_{2}^{n} =( U^{n+1}-U^{n}-kU_{t}^{n}-\tfrac{k^{2}}{2}U_{tt}^{n})
+ k\zeta^{n} + \tfrac{k^{2}}{2}\zeta_{t}^{n}-\tfrac{k^{2}}{2}A\psi_{x}^{n}
+ \tfrac{3k}{2}Aw_{2}^{n} + \tfrac{k}{2}Aw_{3}^{n}\,.
\]
Therefore, by Taylor's theorem, $(\ref{eq411})$, $(\ref{eq46})$, $(\ref{eq439})$,
$(\ref{eq441})$,
\[
\|\delta_{2}^{n}\|_{1} \leq c(k^{3}+kh^{2}+k^{2}h^{2} + k^{2}h^{3/2}+k^{3}
+k^{2}h^{2} + k^{3}+\lambda kh^{3/2})\leq C_{1}k(k^{2} + h^{3/2})\,,
\]
which, with $(\ref{eq437})$, concludes the proof of the Lemma. The case of $(\ref{cb})$ is
entirely analogous.
\end{proof}
For the stability and convergence of the fully discrete scheme it does not suffice to suppose
that $k=O(h)$. The following result shows that the stronger condition $k=O(h^{4/3})$
is sufficient.
\begin{proposition} Suppose that the solutions $(\eta, u)$ of $(\ref{scb})$ and $(\ref{cb})$
are sufficiently smooth on $[0,T]$. Then, if $\mu = k/h^{4/3}$, there is a constant
$C=C(\mu)$, which is an increasing continuous function of $\mu$, such that
\[
\max_{0\leq n\leq M}\|H_{h}^{n} - \eta(t^{n})\| \leq C(k^{2} + h^{3/2})\,, \quad
\max_{0\leq n\leq M}\|U_{h}^{n} - u(t^{n})\|_{1} \leq C(k^{2} + h)\,.
\]
\end{proposition}
\begin{proof} We consider $(\ref{scb})$, and put $\ve^{n}=H^{n} - H_{h}^{n}$,
$e^{n}=U^{n}-U_{h}^{n}$, $\theta^{n} = H^{n,1} - H_{h}^{n,1}$, and $\xi^{n}=U^{n,1} - U_{h}^{n,1}$.
We will show that
\begin{equation}
\max_{0\leq n\leq M}(\|\ve^{n}\| + \|e^{n}\|_{1}) \leq C(k^{2} + h^{3/2})\,,
\label{eq442}
\end{equation}
from which the conclusion of the proposition follows. From $(\ref{eq430})$, $(\ref{eq432})$,
$(\ref{eq433})$ we have, for $0\leq n\leq M-1$
\begin{equation}
\ve^{n+1} = \ve^{n} - kP\xi_{x}^{n}
- \tfrac{k}{2}P(H^{n,1}U^{n,1}-H_{h}^{n,1}U_{h}^{n,1})_{x} + \delta_{1}^{n}\,,
\label{eq443}
\end{equation}
and
\begin{equation}
e^{n+1} = e^{n} - kA\theta_{x}^{n}
- \tfrac{3k}{2}A(U^{n,1}U_{x}^{n,1} - U_{h}^{n,1}U_{hx}^{n,1})
- \tfrac{k}{2}A(H^{n,1}H_{x}^{n,1} - H_{h}^{n,1}H_{hx}^{n,1}) + \delta_{2}^{n}\,.
\label{eq444}
\end{equation}
From Lemma $4.3$ we have an estimate of $\|\delta_{1}^{n}\| + \|\delta_{2}^{n}\|_{1}$.
Our goal is to obtain suitable estimates of the remaining terms of the right-hand sides of
$(\ref{eq443})$ and $(\ref{eq444})$ in terms of $\|\ve^{n}\| + \|e^{n}\|_{1}$. To do this,
note first that $(\ref{eq431})$ and $(\ref{eq430})$ give
\[
\theta^{n} = \ve^{n} - \tfrac{k}{2}Pe_{x}^{n} - \tfrac{k}{4}P(H^{n}U^{n}-H_{h}^{n}U_{h}^{n})_{x}\,.
\]
But $H^{n}U^{n} - H_{h}^{n}U_{h}^{n}=H^{n}e^{n} - \ve^{n}e^{n} + U^{n}\ve^{n}$. Therefore
\begin{equation}
\theta^{n} = \ve^{n} - \tfrac{k}{4}\rho_{1}^{n} - \tfrac{k}{2}\omega_{1}^{n}\,,
\label{eq445}
\end{equation}
where
\begin{equation}
\rho_{1}^{n}:=P(U^{n}\ve^{n})_{x}\,,
\label{eq446}
\end{equation}
and
\begin{equation}
\omega_{1}^{n}:=Pe_{x}^{n} + \tfrac{1}{2}P(H^{n}e^{n})_{x} - \tfrac{1}{2}P(\ve^{n}e^{n})_{x}\,.
\label{eq447}
\end{equation}
Similarly,
\[
\xi^{n}=e^{n}-\tfrac{k}{2}A\ve_{x}^{n}-\tfrac{3k}{4}A(U^{n}U_{x}^{n}-U_{h}^{n}U_{hx}^{n})
-\tfrac{k}{4}A(H^{n}H_{x}^{n} - H_{h}^{n}H_{hx}^{n})\,.
\]
But $U^{n}U_{x}^{n} - U_{h}U_{hx}^{n}=(U^{n}e^{n})_{x} - e^{n}e_{x}^{n}$,
$H^{n}H_{x}^{n} - H_{h}^{n}H_{hx}^{n} = (H^{n}\ve^{n})_{x} - \ve^{n}\ve_{x}^{n}$.
Therefore
\begin{equation}
\xi^{n} = e^{n} - \tfrac{k}{2}\omega_{2}^{n}\,,
\label{eq448}
\end{equation}
where
\begin{equation}
\omega_{2}^{n}:=A\ve_{x}^{n} + \tfrac{3}{2}A(U^{n}e^{n})_{x} - \tfrac{3}{2}A(e^{n}e_{x}^{n})
+ \tfrac{1}{2}A(H^{n}\ve^{n})_{x} - \tfrac{1}{2}A(\ve^{n}\ve_{x}^{n})\,.
\label{eq449}
\end{equation}
In addition, by $(\ref{eq445})$ we have
\begin{equation}
H^{n,1}U^{n,1} - H_{h}^{n,1}U_{h}^{n,1} = U^{n}\ve^{n} - \tfrac{k}{4}U^{n}\rho_{1}^{n}
-\tfrac{k}{2}U^{n}\omega_{1}^{n} + \omega_{3}^{n}\,,
\label{eq450}
\end{equation}
where
\begin{equation}
\omega_{3}^{n}:=(U^{n,1} - U^{n})\theta^{n} + H^{n,1}\xi^{n} - \theta^{n}\xi^{n}\,.
\label{eq451}
\end{equation}
From $(\ref{eq443})$, $(\ref{eq445})$-$(\ref{eq450})$ we conclude therefore that for $0\leq n\leq M-1$
\begin{equation}
\ve^{n+1} = \ve^{n} - \tfrac{k}{2}\rho_{1}^{n} + \tfrac{k^{2}}{8}\rho_{2}^{n}
-k\omega_{4}^{n} + \delta_{1}^{n}\,,
\label{eq452}
\end{equation}
where
\begin{equation}
\rho_{2}^{n}:=P(U^{n}\rho_{1}^{n})_{x}\,,
\label{eq453}
\end{equation}
and
\begin{equation}
\omega_{4}^{n}:=Pe_{x}^{n} - \tfrac{k}{2}P\omega_{2x}^{n} - \tfrac{k}{4}P(U^{n}\omega_{1}^{n})_{x}
+ \tfrac{1}{2}P\omega_{3x}^{n}\,.
\label{eq454}
\end{equation}
Finally, using the identities
$U^{n,1}U_{x}^{n,1} - U_{h}^{n,1}U_{hx}^{n,1}=(U^{n,1}\xi^{n})_{x}-\xi^{n}\xi_{x}^{n}$ and
$H^{n,1}H_{x}^{n,1} - H_{h}^{n,1}H_{hx}^{n,1}=(H^{n,1}\theta^{n})_{x} - \theta^{n}\theta_{x}^{n}$,
we obtain from $(\ref{eq444})$ that for $0\leq n\leq M-1$
\begin{equation}
e^{n+1} = e^{n} - kA\theta_{x}^{n} - \tfrac{3k}{2}A\bigl( (U^{n,1}\xi^{n})_{x} - \xi^{n}\xi_{x}^{n}\bigr)
- \tfrac{k}{2}A\bigl( (H^{n,1}\theta^{n})_{x} - \theta^{n}\theta_{x}^{n}\bigr) + \delta_{2}^{n}\,.
\label{eq455}
\end{equation}
We now estimate the various terms in the right-hand sides of $(\ref{eq452})$ and $(\ref{eq455})$.
Let $0\leq n^{*}\leq M-1$ be the maximal index for which
\begin{equation}
\|\ve^{n}\|_{1} + \|e^{n}\|_{1} \leq 1\,, \quad 0\leq n\leq n^{*}\,.
\label{eq456}
\end{equation}
Then, by $(\ref{eq447})$, the approximation properties of $S_{h}^{2}$ and $(\ref{eq456})$,
we have, for $0\leq n\leq n^{*}$
\begin{equation}
\|\omega_{1}^{n}\| \leq \|e^{n}\|_{1} + C\|H^{n}\|_{1}\|e^{n}\|_{1}
+ C\|\ve^{n}\|_{1}\|e^{n}\|_{1} \leq C\|e^{n}\|_{1}\,.
\label{eq457}
\end{equation}
By $(\ref{eq449})$, the approximation properties of $S_{h}^{2}$, $S_{h,0}^{2}$, and $(\ref{eq46})$
there follows for $0\leq n\leq n^{*}$
\begin{equation}
\|\omega_{2}^{n}\|_{1} \leq C(\|\ve^{n}\| + \|e^{n}\| + \|e^{n}\|_{1}\|e^{n}\|
+ \|\ve^{n}\| + \|\ve^{n}\|_{1}\|\ve^{n}\|) \leq C(\|\ve^{n}\| + \|e^{n}\|)\,.
\label{eq458}
\end{equation}
Hence, by $(\ref{eq448})$, for $0\leq n\leq n^{*}$
\begin{equation}
\|\xi^{n}\|_{1} \leq C(\|\ve^{n}\| + \|e^{n}\|_{1})\,.
\label{eq459}
\end{equation}
In addition, by $(\ref{eq445})$, $(\ref{eq446})$, $(\ref{eq457})$ and the inverse assumptions
we have for $0\leq n\leq n^{*}$
\begin{equation}
\|\theta^{n}\|\leq \|\ve^{n}\| + C\tfrac{k}{2}\|\ve^{n}\| + C\|e^{n}\|_{1}
\leq (1+C\lambda)\|\ve^{n}\| + C\|e^{n}\|_{1} \leq C_{\lambda}(\|\ve^{n}\|+\|e^{n}\|_{1})\,,
\label{eq460}
\end{equation}
where we have put again $\lambda=k/h$; in the sequel, $C_{\lambda}$ will denote various constants that
depend polynomially on $\lambda$. Note also that in view of $(\ref{eq456})$ we have for
$0\leq n\leq n^{*}$, from $(\ref{eq457})$, $(\ref{eq446})$ and inverse inequalities
\begin{equation}
\|\theta^{n}\|_{1} \leq \|\ve^{n}\|_{1} + Ck\|\rho_{1}^{n}\|_{1} + Ck\|\omega_{1}^{n}\|_{1}
\leq \|\ve^{n}\|_{1} + C\lambda\|\rho_{1}^{n}\| + C\lambda\|\omega_{1}^{n}\|
\leq C_{\lambda}\,.
\label{eq461}
\end{equation}
By $(\ref{eq451})$, $(\ref{eq434})$, $(\ref{eq411})$, $(\ref{eq460})$, $(\ref{eq459})$ and
$(\ref{eq456})$ we have for $0\leq n\leq n^{*}$
\begin{equation}
\begin{aligned}
\|\omega_{3}^{n}\| & \leq C(\|U^{n,1} - U^{n}\|_{1}\|\theta^{n}\| + \|H^{n,1}\| \|\xi^{n}\|_{1}
+ \|\theta^{n}\| \|\xi^{n}\|_{1}) \\
& \leq C((k+h^{2})\|\theta^{n}\| + \|\xi^{n}\|_{1} + \|\theta^{n}\|\|\xi^{n}\|_{1}) \\
& \leq C_{\lambda}(\|\ve^{n}\| + \|e^{n}\|_{1})\,.
\end{aligned}
\label{eq462}
\end{equation}
Also, by $(\ref{eq460})$, $(\ref{eq459})$, $(\ref{eq461})$, for $0\leq n\leq n^{*}$
\begin{equation}
\begin{aligned}
\|\omega_{3}^{n}\|_{1} & \leq C(\|U^{n,1} - U^{n}\|_{1}\|\theta^{n}\|_{1}
+ \|H^{n,1}\|_{1} \|\xi^{n}\|_{1} + \|\theta^{n}\|_{1} \|\xi^{n}\|_{1}) \\
& \leq (k+h^{2})h^{-1}C_{\lambda}(\|\ve^{n}\| + \|e^{n}\|_{1})
+ C_{\lambda}(\|\ve^{n}\| + \|e^{n}\|_{1}) + C_{\lambda}(\|\ve^{n}\| + \|e^{n}\|_{1})\\
& \leq C_{\lambda}(\|\ve^{n}\| + \|e^{n}\|_{1})\,.
\end{aligned}
\label{eq463}
\end{equation}
Hence, by $(\ref{eq454})$, $(\ref{eq458})$, $(\ref{eq457})$, $(\ref{eq462})$ and the inverse
inequalities we have for $0\leq n\leq n^{*}$
\[
\|\omega_{4}^{n}\| \leq \|e^{n}\|_{1} + Ck(\|\ve^{n}\| + \|e^{n}\|) + C\lambda \|e^{n}\|_{1}
+ C_{\lambda}(\|\ve^{n}\| + \|e^{n}\|_{1}) \leq C_{\lambda}(\|\ve^{n}\| + \|e^{n}\|_{1})\,.
\]
Therefore, in the right-hand side of $(\ref{eq452})$ we have for $0\leq n\leq n^{*}$ in view
of Lemma $4.3$,
\begin{equation}
\|-k\omega_{4}^{n} + \delta_{1}^{n}\| \leq C_{\lambda}k(\|\ve^{n}\| + \|e^{n}\|_{1})
+ C_{\lambda}k(k+h^{3/2})\,.
\label{eq464}
\end{equation}
We embark now upon obtaining a sharp $L^{2}$-estimate of the remaining term
$\ve^{n} - \tfrac{k}{2}\rho_{1}^{n} + \tfrac{k^{2}}{8}\rho_{2}^{n}$ in $(\ref{eq452})$. We have
\[
\|\ve^{n} - \tfrac{k}{2}\rho_{1}^{n} + \tfrac{k^{2}}{8}\rho_{2}^{n}\|^{2}
=\|\ve^{n}\|^{2} + \tfrac{k^{2}}{4}\|\rho_{1}^{n}\|^{2} + \tfrac{k^{4}}{64}\|\rho_{2}^{n}\|^{2}
-k(\ve^{n},\rho_{1}^{n}) + \tfrac{k^{2}}{4}(\ve^{n},\rho_{2}^{n})
- \tfrac{k^{3}}{8}(\rho_{1}^{n},\rho_{2}^{n})\,.
\]
Now, by $(\ref{eq446})$
\[
(\ve^{n},\rho_{1}^{n}) = (\ve^{n},(U^{n}\ve^{n})_{x}) = \tfrac{1}{2}(U_{x}^{n}\ve^{n},\ve^{n})\,.
\]
Also, by $(\ref{eq453})$, $(\ref{eq446})$
\begin{align*}
(\ve^{n},\rho_{2}^{n}) & = - (\ve_{x}^{n},U^{n}\rho_{1}^{n})
= - \|\rho_{1}^{n}\|^{2} + (U_{x}^{n}\ve^{n},\rho_{1}^{n})\,, \\
(\rho_{1}^{n},\rho_{2}^{n}) & = \tfrac{1}{2}(U_{x}^{n}\rho_{1}^{n},\rho_{1}^{n})\,.
\end{align*}
We conclude that
\begin{equation}
\|\ve^{n} - \tfrac{k}{2}\rho_{1}^{n} + \tfrac{k^{2}}{8}\rho_{2}^{n}\|^{2} =
\|\ve^{n}\|^{2} + \tfrac{k^{4}}{64}\|\rho_{2}^{n}\|^{2} - \tfrac{k}{2}(U_{x}^{n}\ve^{n},\ve^{n})
+ \tfrac{k^{2}}{4}(U_{x}^{n}\ve^{n},\rho_{1}^{n})
- \tfrac{k^{3}}{16}(U_{x}^{n}\rho_{1}^{n},\rho_{1}^{n})\,.
\label{eq465}
\end{equation}
Now, by the approximation and inverse properties of $S_{h,0}^{2}$ we have by $(\ref{eq446})$,
$(\ref{eq453})$
\begin{align*}
|(U_{x}^{n}\ve^{n},\ve^{n})| & \leq C\|\ve^{n}\|^{2}\,, \\
|(U_{x}^{n}\ve^{n},\rho_{1}^{n})| & \leq C\|\ve^{n}\|\|\rho_{1}^{n}\| \leq Ch^{-1}\|\ve^{n}\|^{2}\,, \\
|(U_{x}^{n}\rho_{1}^{n},\rho_{1}^{n})| & \leq C\|\rho_{1}^{n}\|^{2}\leq Ch^{-2}\|\ve^{n}\|^{2}\,,\\
\|\rho_{2}^{n}\| & \leq Ch^{-2}\|\ve^{n}\|\,.
\end{align*}
Inserting these estimates in $(\ref{eq465})$ and recalling that $\mu=k/h^{4/3}$ we are led to
the inequality
\[
\|\ve^{n} - \tfrac{k}{2}\rho_{1}^{n} + \tfrac{k^{2}}{8}\rho_{2}^{n}\|^{2}
\leq (1 + Ck\mu^{3} + Ck + Ck\lambda + Ck\lambda^{2})\|\ve^{n}\|^{2}
\leq (1 + C_{\mu}k)\|\ve^{n}\|^{2}\,,
\]
and, hence,
\begin{equation}
\|\ve^{n} - \tfrac{k}{2}\rho_{1}^{n} + \tfrac{k^{2}}{8}\rho_{2}^{n}\| \leq (1+C_{\mu}k)\|\ve^{n}\|\,,
\label{eq466}
\end{equation}
where, by $C_{\mu}$ we denote a constant depending polynomially on $\mu$; we have replaced
$C_{\lambda}$'s by $C_{\mu}$'s since $\lambda=h^{1/3}\mu\leq\mu$.
We finally obtain from $(\ref{eq466})$, $(\ref{eq464})$ and $(\ref{eq452})$, for $0\leq n\leq n^{*}$,
that
\begin{equation}
\|\ve^{n+1}\| \leq \|\ve^{n}\| + C_{\mu}k(\|\ve^{n}\| + \|e^{n}\|_{1})
+ C_{\mu}k(k^{2} + h^{3/2})\,.
\label{eq467}
\end{equation}
We now estimate the $\|\cdot\|_{1}$ norm of the various terms in the right-hand side
of $(\ref{eq455})$. For $0\leq n\leq n^{*}$, by $(\ref{eq46})$, $(\ref{eq434})$, $(\ref{eq456})$,
and $(\ref{eq459})$ we have
\[
\|A\bigl( (U^{n,1}\xi^{n})_{x} - \xi^{n}\xi_{x}^{n}\bigr)\|_{1}
\leq C\|U^{n,1}\xi^{n} - \tfrac{1}{2}(\xi^{n})^{2}\| \leq C \|\xi^{n}\|_{1}
\leq C(\|\ve^{n}\| + \|e^{n}\|_{1})\,.
\]
Similarly, for $0\leq n\leq n^{*}$, by $(\ref{eq460})$, we have
\[
\|A\bigl( (H^{n,1}\theta^{n})_{x} - \theta^{n}\theta_{x}^{n}\bigr)\|_{1}
\leq C\|H^{n,1}\theta^{n} - \tfrac{1}{2}(\theta^{n})^{2}\| \leq C_{\lambda} \|\theta^{n}\|
\leq C_{\lambda}(\|\ve^{n}\| + \|e^{n}\|_{1})\,.
\]
Therefore, using $(\ref{eq455})$, $(\ref{eq460})$ and Lemma $4.3$ we have for $0\leq n\leq n^{*}$
\begin{equation}
\begin{aligned}
\|e^{n+1}\|_{1} & \leq \|e^{n}\|_{1} + Ck\|\theta^{n}\|
+ C_{\lambda}k(\|\ve^{n}\| + \|e^{n}\|_{1}) + C_{\lambda}k(k^{2} + h^{3/2}) \\
& \leq \|e^{n}\|_{1} + C_{\lambda}k(\|\ve^{n}\| + \|e^{n}\|_{1}) + C_{\lambda}k(k^{2} + h^{3/2})\,.
\end{aligned}
\label{eq468}
\end{equation}
Adding now $(\ref{eq467})$ and $(\ref{eq468})$, we conclude for $0\leq n\leq n^{*}$ that
\[
\|\ve^{n+1}\| + \|e^{n+1}\|_{1} \leq (1+C_{\mu}k)(\|\ve^{n}\| + \|e^{n}\|_{1})
+ C_{\mu}k(k^{2} + h^{3/2})\,.
\]
Using Gronwall's Lemma and taking $h$ sufficiently small we conclude that $n^{*}$ may be taken equal to
$M-1$ and there holds that
\[
\|\ve^{n}\| + \|e^{n}\|_{1} \leq \exp(C_{\mu}T)(k^{2} + h^{3/2})\,, \quad
0\leq n\leq M\,,
\]
i.e. that $(\ref{eq442})$ is valid; the conclusion of the proposition follows.
\end{proof}
\begin{remark} Arguing as in Remark $4.1$ we have again
$\|U_{h}^{n} - u(t^{n})\|_{\infty} = O(k^{2} + h^{3/2})$.
\end{remark}
\begin{remark} If we repeat the numerical experiment in Remark $4.2$ using the improved Euler method
for time stepping, we obtain, for $N=1/h=400$, the evolution of the $L^{2}$ errors
$\|H_{h}^{n} - \eta(t^{n})\|$ when $k=h$ and $k=h^{4/3}$ shown in Table \ref{tbl41}.
These results suggest that the condition $k=h^{4/3}$ is probably necessary as well for the
stability of the scheme. If we repeat the experiment using $k=h$ and the fourth-order explicit
classical R-K scheme for time stepping (this scheme, as will be shown in the next paragraph, is
stable when $k=h$,) we obtain errors approximately equal to those of the last column of Table \ref{tbl41}:
For example at $t^{n}=0.95$ and $t^{n}=1.0$ the analogous errors are equal to $0.1810(-3)$ and
$0.1954(-3)$, respectively. This implies that the errors of the last column of Table \ref{tbl41}
are essentially due to the spatial discretization.
\def\baselinestretch{1}
\scriptsize
\begin{table}[h]
\begin{center}
\begin{tabular}[h]{ | c | c || c | c | }\hline
$t^{n}$ &  $k=h$         & $t^{n}$   &  $k=h^{4/3}$ \\ \hline
$0.05$  &  $0.9594(-5)$  & $0.05090$ &  $0.9833(-5)$ \\ \hline
$0.1$   &  $0.1857(-4)$  & $0.10179$ &  $0.1890(-4)$ \\ \hline
$0.3$   &  $0.5436(-4)$  & $0.30198$ &  $0.5394(-4)$ \\ \hline
$0.5$   &  $0.8538(-4)$  & $0.50217$ &  $0.8393(-4)$ \\ \hline
$0.7$   &  $0.1211(-3)$  & $0.70236$ &  $0.1206(-3)$ \\ \hline
$0.8$   &  $0.3949(-3)$  & $0.80075$ &  $0.1448(-3)$ \\ \hline
$0.825$ &  $0.2214(-2)$  & $0.82450$ &  $0.1507(-3)$ \\ \hline
$0.85$  &  $0.1445(-1)$  & $0.85165$ &  $0.1572(-3)$ \\ \hline
$0.9$   &  $0.8082$      & $0.90254$ &  $0.1691(-3)$ \\ \hline
$0.95$  &  $0.7706(+18)$ & $0.95005$ &  $0.1811(-3)$ \\ \hline
$1.0$   &  $overflow$    & $1.00026$ &  $0.1963(-3)$ \\ \hline
\end{tabular}
\end{center}
\caption{$L^{2}$-errors $\|H_{h}^{n} - \eta(t^{n})\|$ for the standard Galerkin method
with piecewise linear continuous functions with $h=1/400$ and improved Euler time stepping
with $k=h$ and $k=h^{4/3}$; example of Remark $4.2$.}
\label{tbl41}
\end{table}
\normalsize
\end{remark}
\subsection{Fourth-order Runge-Kutta scheme with cubic splines}
Our third example is time stepping with the classical, fourth-order accurate four-stage
explicit Runge-Kutta scheme, written in the case of the o.d.e. $y'=f(t,y)$ in the three-step form
\begin{align*}
y^{n,1} & = y^{n} + \tfrac{k}{2}f(t^{n}+\tfrac{k}{2},y^{n})\,, \\
y^{n,2} & = y^{n} + \tfrac{k}{2}f(t^{n}+\tfrac{k}{2},y^{n,1})\,, \\
y^{n,3} & = y^{n} + k f(t^{n}+k,y^{n,2})\,, \\
y^{n+1} & = y^{n} + k\bigl( \tfrac{1}{6}f(t^{n},y^{n}) + \tfrac{1}{3}f(t^{n}+\tfrac{k}{2},y^{n,1})
+ \tfrac{1}{3}f(t^{n}+\tfrac{k}{2},y^{n,2}) + \tfrac{1}{6}f(t^{n}+k,y^{n,3})\bigr)\,.
\end{align*}
We will couple this scheme with a space discretization that uses cubic splines on a uniform mesh
on $[0,1]$. We recall that the semidiscrete problem can be written, e.g. in the case of the $(\ref{scb})$,
in the form $(\ref{eq38})$ or, equivalently, as follows. We seek $(\eta_{h}, u_{h})$
$\in C^{1}(0,T; S_{h}^{4}\times S_{h,0}^{4})$ such that
\begin{equation}
\begin{aligned}
\begin{aligned}
\eta_{ht} & + Pu_{hx} + \tfrac{1}{2}P(\eta_{h}u_{h})_{x} = 0\,, \\
u_{ht} & + A\eta_{hx} + \tfrac{3}{2}A(u_{h}u_{hx}) + \tfrac{1}{2}A(\eta_{h}\eta_{hx}) = 0\,,
\end{aligned}
\quad 0\leq t\leq T\,, \\
\eta_{h}(0) = I_{h}\eta_{0}\,, \quad u_{h}(0) = R_{h}u_{0}\,,\hspace{111pt} &
\end{aligned}
\label{eq469}
\end{equation}
where $P : L^{2}\to S_{h}^{4}$ is the $L^{2}$-projection operator onto $S_{h}^{4}$,
$A : L^{2} \to S_{h,0}^{4}$ is defined as in $(\ref{eq45})$, but now in $S_{h,0}^{4}$,
$I_{h}$ is the interpolant in $S_{h}^{4}$, and $R_{h}$ is introduced in Paragraph $3.1$.
The semidiscrete scheme for the $(\ref{cb})$ system is defined analogously. \par
We start with an estimation of the truncation errors of the semidiscrete equations when applied to
$I_{h}\eta(t)$ and $R_{h}u(t)$ similar to that of Lemma $4.1$.
\begin{lemma} Suppose that the solution $(\eta_{h}, u_{h})$ of $(\ref{scb})$ is sufficiently smooth
in $[0,T]$. Let $H(t) = I_{h}\eta(t)$, $U(t) = R_{h}u(t)$ and define $\psi = \psi(t) \in S_{h}^{4}$,
$\zeta=\zeta(t) \in S_{h,0}^{4}$, for $0\leq t\leq T$ by
\begin{align}
H_{t}  + PU_{x} + \tfrac{1}{2}P(HU)_{x} & = \psi\,, \label{eq470}\\
U_{t} + AH_{x} + \tfrac{3}{2}A(UU_{x}) + \tfrac{1}{2}A(HH_{x}) & = \zeta\,. \label{eq471}
\end{align}
Then, for $j=0$, $1$, $2$, $3$
\begin{equation}
\|\partial_{t}^{j}\psi\| \leq Ch^{3.5}\sqrt{\ln 1/h}\,, \qquad
\|\partial_{t}^{j}\zeta\|_{1} \leq Ch^{4}\,,
\label{eq472}
\end{equation}
hold for $0\leq t\leq T$. An analogous result holds for $(\ref{cb})$ as well.
\end{lemma}
\begin{proof} Subtracting the equations
\begin{align*}
P\eta_{t} + Pu_{x} + \tfrac{1}{2}P(\eta u)_{x} & = 0\,, \\
H_{t} + PU_{x} + \tfrac{1}{2}P(HU)_{x} & = \psi\,,
\end{align*}
we obtain, putting $\rho:=\eta - H$, $\sigma:=u - U$
\[
P\rho_{t} + P\sigma_{x} + \tfrac{1}{2}P(\eta u - HU)_{x} = -\psi\,.
\]
Since
\[
\eta u - HU = \eta(u-U) +U(\eta - H) = \eta(u-U)-(u-U)(\eta-H) + u(\eta-H)=\eta\sigma + u\rho-\rho\sigma\,,
\]
we have
\[
P\rho_{t} + P\sigma_{x} + \tfrac{1}{2}P(\eta\sigma)_{x}+\tfrac{1}{2}P(u\rho)_{x}
-\tfrac{1}{2}P(\rho\sigma)_{x}=-\psi\,,
\]
and conclude from Lemmas $3.2$ and $3.4$ that
\[
\|\partial_{t}^{j}\psi\| \leq C h^{3.5}\sqrt{\ln 1/h}\,, \quad j=0,1,2,3\,,
\]
which is the first estimate in $(\ref{eq472})$
By the definition of $A$ we have
\[
A(u_{t} - \tfrac{1}{3}u_{txx}) = U_{t}\,.
\]
Subtracting now the equations
\begin{align*}
A(u_{t}-\tfrac{1}{3}u_{txx}) + A\eta_{x} + \tfrac{3}{2}A(uu_{x}) + \tfrac{1}{2}A(\eta\eta_{x}) & =0\,,\\
U_{t} + AH_{x} + \tfrac{3}{2}A(UU_{x}) + \tfrac{1}{2}A(HH_{x}) & = \zeta\,,
\end{align*}
we obtain
\[
A\rho_{x} + \tfrac{3}{2}A(uu_{x}-UU_{x}) + \tfrac{1}{2}A(\eta\eta_{x}-HH_{x})=-\zeta\,.
\]
But
\[
uu_{x} - UU_{x} = (u\sigma)_{x} - \sigma\sigma_{x}\,, \quad
\eta\eta_{x}-HH_{x} = (\eta\rho)_{x} - \rho\rho_{x}\,.
\]
Hence
\[
A\rho_{x}+\tfrac{3}{2}\bigl(A(u\sigma)_{x}-A(\sigma\sigma_{x})\bigr)
+\tfrac{1}{2}\bigl(A(\eta\rho)_{x}-A(\rho\rho_{x})\bigr) = -\zeta\,,
\]
and from $(\ref{eq46})$ and the approximation properties of $S_{h}^{4}$, $S_{h,0}^{4}$ we conclude
\[
\|\partial_{t}^{j}\zeta\|_{1}\leq C h^{4}\,,
\]
for $j=0,1,2,3$, thus proving the Lemma. The $(\ref{cb})$ case is entirely similar.
\end{proof}
We now define the fully discrete scheme. We consider only the case of $(\ref{scb})$, that of $(\ref{cb})$
being analogous. We let as usual $M$ be a positive integer, $k=T/M$ and $t^{n}=nk$, for \,$n=0,1,\dots,M$.
For $0\leq n\leq M$ we seek $(H_{h}^{n}, U_{h}^{n}) \in S_{h}^{4}\times S_{h,0}^{4}$
approximations of $\eta(t^{n})$, $u(t^{n})$, and for $0\leq n\leq M-1$
$(H_{h}^{n,j}, U_{h}^{n,j}) \in S_{h}^{4}\times S_{h,0}^{4}$, $j=1,2,3$, such that for $0\leq n\leq M-1$
\begin{equation}
\begin{aligned}
H_{h}^{n,j} - H_{h}^{n} + k a_{j}P\bigl(U_{h}^{n,j-1} + \tfrac{1}{2}H_{h}^{n,j-1}U_{h}^{n,j-1}\bigr)_{x}& =0\,, \\
U_{h}^{n,j} - U_{h}^{n} + ka_{j}A\bigl(H_{hx}^{n,j-1} + \tfrac{3}{2}U_{h}^{n,j-1}U_{hx}^{n,j-1}
+\tfrac{1}{2}H_{h}^{n,j-1}H_{hx}^{n,j-1}\bigr) & = 0\,,
\end{aligned}
\label{eq473}
\end{equation}
for $j=1,2,3$ and
\begin{equation}
\begin{aligned}
H_{h}^{n+1} - H_{h}^{n} + kP\Bigl[ \sum_{j=1}^{4}b_{j}\bigl(U_{h}^{n,j-1}+
\tfrac{1}{2}H_{h}^{n,j-1}U_{h}^{n,j-1}\bigr)\Bigr]_{x} & = 0\,, \\
U_{h}^{n+1} - U_{h}^{n} + kA\Bigl[\sum_{j=1}^{4}b_{j}\bigl(H_{hx}^{n,j-1}+\tfrac{3}{2}U_{hx}^{n,j-1}U_{h}^{n,j-1}
+ \tfrac{1}{2}H_{hx}^{n,j-1}H_{h}^{n,j-1}\bigr)\Bigr] & = 0\,,
\end{aligned}
\label{eq474}
\end{equation}
where
\[
H_{h}^{n,0} = H_{h}^{n}\,, \quad U_{h}^{n,0} = U_{h}^{n}\,, \quad a_{1}=a_{2}=1/2\,, \quad a_{3}=1\,, \quad b_{1}=b_{4}=1/6\,,
\quad b_{2}=b_{3}=1/3\,,
\]
and
\[
H_{h}^{0} = \eta_{h}(0)\,, \quad U_{h}^{0}= u_{h}(0)\,.
\]
We first investigate the consistency of this scheme. We define
$H$, $U$ as in Lemma $4.4$ and put $H^{n}=H(t^{n})$, $U^{n}=U(t^{n})$. We let $V^{n,j} \in S_{h}^{4}$,
$W^{n,j} \in S_{h,0}^{4}$, for $j=0$, $1$, $2$, $3$ be defined for $0 \leq n \leq M-1$ by the equations
\begin{equation}
\begin{aligned}
V^{n,j} - H^{n} + ka_{j}P\bigl(W^{n,j-1} + \tfrac{1}{2}V^{n,j-1}W^{n,j-1}\bigr)_{x} & = 0\,, \\
W^{n,j} - U^{n} + ka_{j}A\bigl(V_{x}^{n,j-1} + \tfrac{3}{2}W^{n,j-1}W_{x}^{n,j-1} +
\tfrac{1}{2}V^{n,j-1}V_{x}^{n,j-1}\bigr)&=0\,,
\end{aligned}
\label{eq475}
\end{equation}
and
\[
V^{n,0} = H^{n}\,, \quad W^{n,0} = U^{n}\,.
\]
Before proving our consistency estimates we introduce some notation that will simplify the computations.
Let
\[
\Phi = U + \tfrac{1}{2}HU\,, \quad F = H_{x} + \tfrac{3}{2}UU_{x} + \tfrac{1}{2}HH_{x}\,,
\quad \Phi^{n} = \Phi(t^{n})\,, \quad F^{n} = F(t^{n})\,,
\]
and consider the functions
\begin{align*}
\Phi_{h}^{n,j} & = U_{h}^{n,j} + \tfrac{1}{2}H_{h}^{n,j}U_{h}^{n,j}\,,\\
F_{h}^{n,j} & = H_{hx}^{n,j} + \tfrac{3}{2}U_{h}^{n,j}U_{hx}^{n,j} + \tfrac{1}{2}H_{h}^{n,j}H_{hx}^{n,j}\,,
\end{align*}
for $j=0$, $1$, $2$, $3$. With a slight abuse of notation we put for $j=0$, $1$, $2$, $3$
\begin{align*}
\Phi^{n,j} & = W^{n,j} + \tfrac{1}{2}V^{n,j}W^{n,j}\,,\\
F^{n,j} & = V_{x}^{n,j} + \tfrac{3}{2}W^{n,j}W_{x}^{n,j} + \tfrac{1}{2}V^{n,j}V_{x}^{n,j}\,.
\end{align*}
We may then write $(\ref{eq470})$ and $(\ref{eq471})$ as
\begin{equation}
H_{t} + P\Phi_{x} = \psi\,,
\label{eq476}
\end{equation}
\begin{equation}
U_{t} + AF = \zeta\,,
\label{eq477}
\end{equation}
and $(\ref{eq473})$, $(\ref{eq474})$, respectively, as
\begin{equation}
\begin{aligned}
H_{h}^{n,j} - H_{h}^{n} + ka_{j}P\Phi_{hx}^{n,j-1} & = 0\,,\\
U_{h}^{n,j} - U_{h}^{n} + ka_{j}AF_{h}^{n,j-1} & = 0\,,
\end{aligned}
\label{eq478}
\end{equation}
for $j=1$, $2$, $3$, and
\begin{equation}
\begin{aligned}
H_{h}^{n+1} - H_{h}^{n} + kP\bigl[\sum_{j=1}^{4}b_{j}\Phi_{h}^{n,j-1}\bigr]_{x}=0\,,\\
U_{h}^{n+1} - U_{h}^{n} + kA\bigl[\sum_{j=1}^{4}b_{j}F_{h}^{n,j-1}\bigr] = 0\,.
\end{aligned}
\label{eq479}
\end{equation}
Finally, we may write the ralations $(\ref{eq475})$ in the form
\begin{equation}
\begin{aligned}
V^{n,j} - H^{n} + ka_{j}P\Phi_{x}^{n,j-1} &= 0\,,\\
W^{n,j} - U^{n} + ka_{j}AF^{n,j-1} & = 0\,,
\end{aligned}
\label{eq480}
\end{equation}
for $j=1$, $2$, $3$.
\begin{lemma} Suppose that the solution $(\eta, u)$ of $(\ref{scb})$ is sufficiently smooth and let
$\lambda = k/h$. Define, for $0\leq n\leq M-1$, $\delta_{1}^{n}$ and $\delta_{2}^{n}$ by the equations
\begin{equation}
\delta_{1}^{n} = H^{n+1} - H^{n} + kP\bigl[\sum_{j=1}^{4}b_{j}\Phi^{n,j-1}\bigr]_{x}\,,
\label{eq481}
\end{equation}
\begin{equation}
\delta_{2}^{n} = U^{n+1} - U^{n} + kA\bigl[\sum_{j=1}^{4}b_{j}F^{n,j-1}\bigr]\,.
\label{eq482}
\end{equation}
Then, there exists a constant $C_{\lambda}$ which is a polynomial of $\lambda$ with nonnegative coefficients,
such that
\[
\max_{0\leq n\leq M-1} \|\delta_{1}^{n}\| \leq C_{\lambda}k(k^{4} + h^{3.5}\sqrt{\ln 1/h})\,,\quad
\max_{0\leq n\leq M-1} \|\delta_{2}^{n}\|_{1} \leq C_{\lambda}k(k^{4} + h^{4})\,.
\]
\end{lemma}
\begin{proof}
We need to find formulas for the functions $V^{n,j}$, $W^{n,j}$, for $j=1$, $2$, $3$ up to some remainder terms.
From the first equation of $(\ref{eq477})$ we have
\[
V^{n,1} - H^{n} + a_{1}kP\Phi_{x}^{n} = 0\,,
\]
whence from $(\ref{eq473})$
\begin{equation}
V^{n,1} = H^{n} + a_{1}kH_{t}^{n} - a_{1}k\psi^{n}\,.
\label{eq483}
\end{equation}
Also, from the second equation of $(\ref{eq480})$ we obtain
\[
W^{n,1} - U^{n} + a_{1}kAF^{n} = 0\,,
\]
and so from $(\ref{eq477})$
\begin{equation}
W^{n,1} = U^{n} + a_{1}k U_{t}^{n} - a_{1}k\zeta^{n}\,.
\label{eq484}
\end{equation}
Thus
\begin{equation}
V^{n,1}W^{n,1} = H^{n}U^{n} + a_{1}k(HU)_{t}^{n} + a_{1}^{2}k^{2}H_{t}^{n}U_{t}^{n} + v_{1}^{n}\,,
\label{eq485}
\end{equation}
where, in view of $(\ref{eq472})$ and inverse properties
\[
\|v_{1}^{n}\|_{1} \leq C_{\lambda}h^{3.5}\sqrt{\ln 1/h}\,.
\]
In addition, from $(\ref{eq484})$, $(\ref{eq485})$ we obtain
\[
W^{n,1} + \tfrac{1}{2}V^{n,1}W^{n,1} = U^{n}+\tfrac{1}{2}H^{n}U^{n}+
a_{1}k\bigl(U_{t}^{n} + \tfrac{1}{2}(HU)_{t}^{n}\bigr) + \tfrac{a_{1}^{2}}{2}k^{2}H_{t}^{n}U_{t}^{n}+v_{2}^{n}\,,
\]
i.e.
\begin{equation}
\Phi^{n,1} = \Phi^{n} + a_{1}k\Phi_{t}^{n} + \tfrac{a_{1}^{2}}{2}k^{2}H_{t}^{n}U_{t}^{n}+v_{2}^{n}\,,
\label{eq486}
\end{equation}
where
\[
\|v_{2}^{n}\|_{1} \leq C_{\lambda}h^{3.5}\sqrt{\ln 1/h}\,.
\]
Since now
\[
V^{n,2} = H^{n} - a_{2}kP\Phi_{x}^{n,1}\,,
\]
we have from $(\ref{eq486})$ that
\[
V^{n,2} = H^{n} - a_{2}kP\Phi_{x}^{n} - a_{1}a_{2}k^{2}
P\Phi_{xt}^{n} - \tfrac{a_{1}^{2}a_{2}}{2}k^{3}P(H_{t}^{n}U_{t}^{n})_{x}
- a_{2}kP(v_{2}^{n})_{x}\,,
\]
from which, in view of $(\ref{eq473})$, $(\ref{eq472})$, and inverse properties there follows that
\begin{equation}
V^{n,2} = H^{n} + a_{2}kH_{t}^{n}+a_{1}a_{2}k^{2}H_{tt}^{n}-
\tfrac{a_{1}^{2}a_{2}}{2}k^{3}P(H_{t}^{n}U_{t}^{n})_{x} + \psi_{1}^{n}\,,
\label{eq487}
\end{equation}
with
\[
\|\psi_{1}^{n}\| \leq C_{\lambda}h^{4}\,,
\]
and
\[
\|\psi_{1}^{n}\|_{1} \leq C_{\lambda}h^{3.5}\sqrt{\ln 1/h}\,.
\]
In addition, from $(\ref{eq484})$
\begin{align*}
W^{n,1}W_{x}^{n,1} & = (U^{n} + a_{1}kU_{t}^{n}-a_{1}k\zeta^{n})(U_{x}^{n}+a_{1}kU_{tx}^{n}
-a_{1}k\zeta_{x}^{n})\\
& = U^{n}U_{x}^{n} + a_{1}k(UU_{x})_{t}^{n} + a_{1}^{2}k^{2}U_{t}^{n}U_{tx}^{n} + w_{1}^{n}\,,
\end{align*}
where, in view of $(\ref{eq472})$
\[
\|w_{1}^{n}\| \leq C k h^{4}\,.
\]
Similarly, from $(\ref{eq480})$ there follows
\begin{align*}
V^{n,1}V_{x}^{n,1} & = (H^{n} + a_{1}kH_{t}^{n} - a_{1}k\psi^{n})(H_{x}^{n} + a_{1}kH_{tx}^{n} - a_{1}k\psi_{x}^{n})\\
& = H^{n}H_{x}^{n} + a_{1}k(HH_{x})_{t}^{n} + a_{1}^{2}k^{2}H_{t}^{n}H_{tx}^{n} + w_{2}^{n}\,,
\end{align*}
where, in view of $(\ref{eq472})$ and inverse properties
\[
\|w_{2}^{n}\|_{-1} \leq C_{\lambda}h^{4}\,,
\]
and
\[
\|w_{2}^{n}\| \leq C_{\lambda}h^{3.5}\sqrt{\ln 1/h}\,.
\]
Therefore, from the above relations and $(\ref{eq483})$ we have
\[
V_{x}^{n,1} + \tfrac{3}{2}W^{n,1}W_{x}^{n,1} + \tfrac{1}{2}V^{n,1}V_{x}^{n,1} =
F^{n} +  a_{1}kF_{t}^{n} + a_{1}^{2}k^{2}(\tfrac{3}{2}U_{t}^{n}U_{tx}^{n}
+ \tfrac{1}{2}H_{t}^{n}H_{tx}^{n}) + w_{3}^{n}\,,
\]
i.e.
\begin{equation}
F^{n,1} = F^{n} +  a_{1}kF_{t}^{n} + a_{1}^{2}k^{2}(\tfrac{3}{2}U_{t}^{n}U_{tx}^{n} +
\tfrac{1}{2}H_{t}^{n}H_{tx}^{n}) + w_{3}^{n}\,,
\label{eq488}
\end{equation}
where
\[
w_{3}^{n} = -k\psi_{1x}^{n} + \tfrac{3}{2}w_{1}^{n} + \tfrac{1}{2}w_{2}^{n}\,.
\]
Since
\[
W^{n,2} = U^{n} - a_{2}kAF^{n,1}\,,
\]
there follows from $(\ref{eq488})$, $(\ref{eq477})$, and previous estimates that
\begin{equation}
W^{n,2} = U^{n} + a_{2}kU_{t}^{n} + a_{1}a_{2}k^{2}U_{tt}^{n}
- a_{1}^{2}a_{2}k^{3}\bigl[\tfrac{3}{2}A(U_{t}^{n}U_{tx}^{n})
+ \tfrac{1}{2}A(H_{t}^{n}H_{tx}^{n})\bigr] + \zeta_{1}^{n}\,,
\label{eq489}
\end{equation}
where
\[
\|\zeta_{1}^{n}\|_{1} \leq C_{\lambda}kh^{4}\,.
\]
We now have from $(\ref{eq487})$, $(\ref{eq489})$, the fact that $a_{1}=a_{2}$, and similar estimates
to those used above, that
\begin{align*}
V^{n,2}W^{n,2} = & \,\,H^{n}U^{n} + a_{2}k(HU)_{t}^{n}+a_{1}a_{2}k^{2}(H^{n}U_{tt}^{n}+H_{t}^{n}U_{t}^{n}+H_{tt}^{n}U^{n})
+ a_{1}a_{2}^{2}k^{3}(H_{t}^{n}U_{tt}^{n} + H_{tt}^{n}U_{t}^{n})\\
& - a_{1}^{2}a_{2}k^{3}\bigl[
\tfrac{3}{2}H^{n}A(U_{t}^{n}U_{tx}^{n}) + \tfrac{1}{2}H^{n}A(H_{t}^{n}H_{tx}^{n})
-\tfrac{1}{2}U^{n}P(H_{t}^{n}U_{t}^{n})_{x}\bigr] + v_{3}^{n}\,,
\end{align*}
where
\[
\|v_{3}^{n}\|_{1} \leq C_{\lambda}h^{3.5}\sqrt{\ln 1/h}
\]
follows from similar, as above, estimates and from the stability of $P$ in $H^{1}$. Rewriting the previous
equation as
\begin{align*}
V^{n,2}W^{n,2} = & \,\,H^{n}U^{n} + a_{2}k(HU)_{t}^{n} + a_{1}a_{2}k^{2}(HU)_{tt}^{n}
- a_{1}a_{2}k^{2}H_{t}^{n}U_{t}^{n} + a_{1}a_{2}^{2}k^{3}(H_{t}U_{t})_{t}^{n} \\
& -a_{1}^{2}a_{2}k^{3}\bigl[\tfrac{3}{2}H^{n}A(U_{t}^{n}U_{tx}^{n}) + \tfrac{1}{2}H^{n}A(H_{t}^{n}H_{tx}^{n})
+\tfrac{1}{2}U^{n}P(H_{t}^{n}U_{t}^{n})_{x}\bigr] + v_{3}^{n}\,,
\end{align*}
we have, by the definition of $\Phi^{n,2}$ and $(\ref{eq489})$
\begin{equation}
\begin{aligned}
\Phi^{n,2} = & \,\,\Phi^{n} + a_{2}k\Phi_{t}^{n} + a_{1}a_{2}k^{2}\Phi_{tt}^{n}
- \tfrac{a_{1}a_{2}}{2}k^{2}H_{t}^{n}U_{t}^{n} + \tfrac{a_{1}a_{2}^{2}}{2}k^{3}(H_{t}U_{t})_{t}^{n}\\
& - a_{1}^{2}a_{2}k^{3}\bigl[\tfrac{3}{2}A(U_{t}^{n}U_{tx}^{n}) + \tfrac{1}{2}A(H_{t}^{n}H_{tx}^{n})
+ \tfrac{3}{4}H^{n}A(U_{t}^{n}U_{tx}^{n}) + \tfrac{1}{4}H^{n}A(H_{t}^{n}H_{tx}^{n})\\
&\qquad \qquad  + \tfrac{1}{4}U^{n}P(H_{t}^{n}U_{t}^{n})_{x}\bigr] + \zeta_{1}^{n} + \tfrac{1}{2}v_{3}^{n}\,.
\end{aligned}
\label{eq490}
\end{equation}
Since now
\[
V^{n,3} = H^{n} - a_{3}k P\Phi^{n,2}_{x} = H^{n} - kP\Phi^{n,2}_{x}\,,
\]
we have from $(\ref{eq490})$ and $(\ref{eq476})$ that
\begin{equation}
V^{n,3} = H^{n} + k H_{t}^{n} + a_{2}k^{2}H_{tt}^{n} + a_{1}a_{2}k^{3}H_{ttt}^{n}
+ \tfrac{a_{1}a_{2}}{2}k^{3}P(H_{t}^{n}U_{t}^{n})_{x} + \psi_{2}^{n}\,,
\label{eq491}
\end{equation}
where
\[
\|\psi_{2}^{n}\| \leq C_{\lambda}h^{4}\,,
\]
and (using the fact that $\|Af\|_{2}\leq C\|f\|$ for $f\in L^{2}$),
\[
\|\psi_{2}^{n}\|_{1} \leq C_{\lambda}h^{3.5}\sqrt{\ln 1/h}\,.
\]
In addition, from $(\ref{eq489})$
\begin{align*}
W^{n,2} W_{x}^{n,2}= & \,\Bigl[U^{n} + a_{2}kU_{t}^{n} + a_{1}a_{2}k^{2}U_{tt}^{n}
- a_{1}^{2}a_{2}k^{3}\bigl[\tfrac{3}{2}A(U_{t}^{n}U_{tx}^{n})
+ \tfrac{1}{2}A(H_{t}^{n}H_{tx}^{n})\bigr] + \zeta_{1}^{n}\Bigr]\cdot \\
&\,\Bigl[U_{x}^{n} + a_{2}kU_{tx}^{n} + a_{1}a_{2}k^{2}U_{ttx}^{n}
- a_{1}^{2}a_{2}k^{3}\bigl[\tfrac{3}{2}A(U_{t}^{n}U_{tx}^{n})
+ \tfrac{1}{2}A(H_{t}^{n}H_{tx}^{n})\bigr]_{x} + \zeta_{1x}^{n}\Bigr]\,,
\end{align*}
and so
\begin{equation}
\begin{aligned}
W^{n,2} W_{x}^{n,2}=&\,\,U^{n}U_{x}^{n} + a_{2}k(UU_{x})_{t}^{n}+a_{1}a_{2}k^{2}(UU_{x})_{tt}^{n}
- a_{1}a_{2}k^{2}U_{t}^{n}U_{tx}^{n} + a_{1}a_{2}^{2}k^{3}(U_{t}^{n}U_{tt}^{n})_{x}\\
& - a_{1}^{2}a_{2}k^{3}\bigl[\tfrac{3}{2}U^{n}A(U_{t}^{n}U_{tx}^{n}) +
\tfrac{1}{2}U^{n}A(H_{t}^{n}H_{tx}^{n})\bigr]_{x} + w_{4}^{n}\,,
\end{aligned}
\label{eq492}
\end{equation}
where $\|w_{4}^{n}\| \leq C_{\lambda}h^{4}$. In addition, from $(\ref{eq487})$
\begin{align*}
V^{n,2} V_{x}^{n,2}=&\,\bigl[H^{n} + a_{2}kH_{t}^{n} + a_{1}a_{2}k^{2}H_{tt}^{n}-
\tfrac{a_{1}^{2}a_{2}}{2}k^{3}P(H_{t}^{n}U_{t}^{n})_{x} + \psi_{1}^{n}\bigr]\cdot\\
&\,\bigl[H_{x}^{n} + a_{2}kH_{tx}^{n}+a_{1}a_{2}k^{2}H_{ttx}^{n}-
\tfrac{a_{1}^{2}a_{2}}{2}k^{3}\bigl(P(H_{t}^{n}U_{t}^{n})_{x}\bigr)_{x} + \psi_{1x}^{n}\bigr]\,.
\end{align*}
Hence,
\begin{equation}
\begin{aligned}
V^{n,2} V_{x}^{n,2}=&\,\,H^{n}H_{x}^{n} + a_{2}k(HH_{x})_{t}^{n} + a_{1}a_{2}k^{2}(HH_{x})_{tt}^{n}
- a_{1}a_{2}k^{2}H_{t}^{n}H_{tx}^{n} \\ &+ a_{1}a_{2}^{2}k^{3}(H_{t}^{n}H_{tt}^{n})_{x}
- \tfrac{a_{1}^{2}a_{2}}{2}k^{3}\bigl[H^{n}P(H_{t}^{n}U_{t}^{n})_{x}\bigr]_{x} + v_{4}^{n}\,,
\end{aligned}
\label{eq493}
\end{equation}
where
\[
\|v_{4}^{n}\|_{-1} \leq C_{\lambda}h^{4}\,,
\]
and
\[
\|v_{4}^{n}\| \leq C_{\lambda}h^{3.5}\sqrt{\ln 1/h}\,.
\]
From the definition of $F^{n,2}$,  $(\ref{eq487})$, $(\ref{eq492})$ and $(\ref{eq493})$
we have
\begin{equation}
\begin{aligned}
F^{n,2} = &\,\,F^{n} + a_{2}kF_{t}^{n}
+ a_{1}a_{2}k^{2}F_{tt}^{n} - a_{1}a_{2}k^{2}(\tfrac{3}{2}U_{t}^{n}U_{tx}^{n}+\tfrac{1}{2}H_{t}^{n}H_{tx}^{n})\\
& + a_{1}a_{2}^{2}k^{3}
\bigl[\tfrac{3}{2}U_{t}^{n}U_{tt}^{n} + \tfrac{1}{2}H_{t}^{n}H_{tt}^{n} -
\tfrac{1}{2}P(H_{t}^{n}U_{t}^{n})_{x} - \tfrac{9}{4}U^{n}A(U_{t}^{n}U_{tx}^{n})\\
& \qquad \qquad
-\tfrac{3}{4}U^{n}A(H_{t}^{n}H_{tx}^{n}) - \tfrac{1}{4}H^{n}P(H_{t}^{n}U_{t}^{n})_{x}\bigr]_{x} + w_{5}^{n}\,,
\end{aligned}
\label{eq494}
\end{equation}
where
\[
w_{5}^{n} = \psi_{1x}^{n} + \tfrac{3}{2}w_{4}^{n} + \tfrac{1}{2} v_{4}^{n}\,.
\]
Hence, since
\[
W^{n,3} = U^{n} - kAF^{n,2}\,,
\]
we obtain
\begin{equation}
\begin{aligned}
W^{n,3} = &\,\,U^{n} + k U_{t}^{n} + a_{2}k^{2}U_{tt}^{n} + a_{1}a_{2}k^{3}U_{ttt}^{n} \\
&+ a_{1}a_{2}k^{3}\bigl[\tfrac{3}{2}A(U_{t}^{n}U_{tx}^{n}) + \tfrac{1}{2}A(H_{t}^{n}H_{tx}^{n})\bigr] + \zeta_{2}^{n}\,,
\end{aligned}
\label{eq495}
\end{equation}
where
\[
\|\zeta_{2}^{n}\|_{1} \leq C_{\lambda} kh^{4} + C k^{4}\,.
\]
From $(\ref{eq491})$, $(\ref{eq495})$ we get
\begin{align*}
V^{n,3}W^{n,3} = &\,H^{n}U^{n} + k(HU)_{t}^{n} + a_{2}k^{2}(HU)_{tt}^{n} +
a_{1}a_{2}k^{3}(HU)_{ttt}^{n} - a_{1}a_{2}k^{3}(H_{t}U_{t})_{t}^{n}\\
&\,+a_{1}a_{2}k^{3}\bigl[ \tfrac{3}{2}H^{n}A(U_{t}^{n}U_{tx}^{n}) +
\tfrac{1}{2}H^{n}A(H_{t}^{n}H_{tx}^{n}) + \tfrac{1}{2}U^{n}P(H_{t}^{n}U_{t}^{n})_{x}\bigr]
+ v_{5}^{n}\,,
\end{align*}
where
\[
\|v_{5}^{n}\|_{1} \leq C_{\lambda}h^{3.5}\sqrt{\ln 1/h}\,.
\]
Hence, from the definition of $\Phi^{n,3}$, $(\ref{eq495})$ and the last relation, we see that
\begin{equation}
\begin{aligned}
\Phi^{n,3} = & \,\,\Phi^{n} + k\Phi_{t}^{n} + a_{2}k^{2}\Phi_{tt}^{n}
+ a_{1}a_{2}k^{3}\Phi_{ttt}^{n} - \tfrac{a_{1}a_{2}}{2}k^{3}(H_{t}U_{t})_{t}^{n}\\
&+ a_{1}a_{2}k^{3}\bigl[ \tfrac{3}{2}A(U_{t}^{n}U_{tx}^{n}) + \tfrac{1}{2}A(H_{t}^{n}H_{tx}^{n})
+ \tfrac{3}{4}H^{n}A(U_{t}^{n}U_{tx}^{n}) \\
& \qquad \qquad + \tfrac{1}{4}H^{n}A(H_{t}^{n}H_{tx}^{n}) + \tfrac{1}{4}U^{n}P(H_{t}^{n}U_{t}^{n})_{x}\bigr]
+ \zeta_{2}^{n} + \tfrac{1}{2}v_{5}^{n}\,.
\end{aligned}
\label{eq496}
\end{equation}
We are now ready to estimate $\delta_{1}^{n}$. By $(\ref{eq490})$ and $(\ref{eq496})$ we have
\begin{align*}
b_{3}\Phi^{n,2} + b_{4}\Phi^{n,3}
= &\,\,(b_{3} + b_{4})\Phi^{n} + (b_{3}a_{2}+b_{4})k\Phi_{t}^{n}
+ (a_{1}a_{2}b_{3} + a_{2}b_{4})k^{2}\Phi_{tt}^{n}\\
&+ a_{1}a_{2}b_{4}k^{3}\Phi_{ttt}^{n}
- \tfrac{a_{1}a_{2}b_{3}}{2}k^{2}H_{t}^{n}U_{t}^{n} +
b_{3}(\zeta_{1}^{n} + \tfrac{1}{2}v_{3}^{n}) + b_{4}(\zeta_{2}^{n} + \tfrac{1}{2}v_{5}^{n})\,.
\end{align*}
In addition, in view of $(\ref{eq486})$
\[
b_{1}\Phi^{n} + b_{2}\Phi^{n,1} = (b_{1}+b_{2})\Phi^{n}
+ a_{1}b_{2}k\Phi_{t}^{n} + \tfrac{a_{1}^{2}b_{2}}{2}k^{2}H_{t}^{n}U_{t}^{n} + b_{2}v_{2}^{n}\,.
\]
Hence,
\[
\sum_{j=1}^{4}b_{j}\Phi^{n,j-1} = \Phi^{n}
+ \tfrac{k}{2}\Phi_{t}^{n} + \tfrac{k^{2}}{6}\Phi_{tt}^{n} + \tfrac{k^{3}}{24}\Phi_{ttt}^{n} + v_{6}^{n}\,,
\]
where
\[
v_{6}^{n} = b_{2}v_{2}^{n} + b_{3}(\zeta_{1}^{n} + \tfrac{1}{2}v_{3}^{n}) +
b_{4}(\zeta_{2}^{n} + \tfrac{1}{2}v_{5}^{n})\,.
\]
Therefore, by $(\ref{eq481})$ and $(\ref{eq476})$ we obtain
\[
\delta_{1}^{n} = H^{n+1} - H^{n} - kH_{t}^{n} - \tfrac{k^{2}}{2}H_{tt}^{n}-\tfrac{k^{3}}{6}H_{ttt}^{n}
- \tfrac{k^{4}}{24}H_{tttt}^{n} + kP(v_{6}^{n})_{x} + v_{7}^{n}\,,
\]
where, by $(\ref{eq472})$
\[
\|v_{7}^{n}\| \leq C_{\lambda}kh^{3.5}\sqrt{\ln 1/h}\,.
\]
Since, in addition
\[
\|v_{6}^{n}\|_{1} \leq C_{\lambda}h^{3.5}\sqrt{\ln 1/h}\,,
\]
we conclude that
\[
\|\delta_{1}^{n}\| \leq C_{\lambda}k(k^{4} + h^{3.5}\sqrt{\ln 1/h})\,.
\]
To estimate $\delta_{2}^{n}$ we must also compute $F^{n,3}$. To this end, using $(\ref{eq495})$ we have
after some algebra
\begin{equation}
\begin{aligned}
W^{n,3}W_{x}^{n,3} =&\,\,U^{n}U_{x}^{n} + k(UU_{x})_{t}^{n} + a_{2}k^{2}(UU_{x})_{tt}^{n}
+ a_{1}a_{2}k^{3}(UU_{x})_{ttt}^{n} - a_{1}a_{2}k^{3}(U_{t}^{n}U_{tt}^{n})_{x}\\
&\, + a_{1}a_{2}k^{3}[ \tfrac{3}{2}U^{n}A(U_{t}^{n}U_{tx}^{n}) + \tfrac{1}{2}U^{n}A(H_{t}^{n}H_{tx}^{n})]_{x}
+ w_{6}^{n}\,,
\end{aligned}
\label{eq497}
\end{equation}
where
\[
\|w_{6}^{n}\| \leq C_{\lambda}h^{4}\,.
\]
Similarly, from $(\ref{eq491})$, we have
\begin{align*}
V^{n,3}V_{x}^{n,3}=&\, H^{n}H_{x}^{n} + k(HH_{x})_{t}^{n} + a_{2}k^{2}(HH_{x})_{tt}^{n}
+ a_{1}a_{2}k^{3}(HH_{x})_{ttt}^{n} \\ \,& - a_{1}a_{2}k^{3}(H_{t}^{n}H_{tt}^{n})_{x}
+ \tfrac{a_{1}a_{2}}{2}k^{3}[H^{n}P(H_{t}^{n}U_{t}^{n})_{x}]_{x} + v_{7}^{n}\,,
\end{align*}
where
\[
\|v_{7}^{n}\|_{-1} \leq C_{\lambda}h^{4}\,, \qquad \|v_{7}^{n}\| \leq C_{\lambda}h^{3.5}\sqrt{\ln 1/h}\,.
\]
Hence, from the definition of $F^{n,3}$, $(\ref{eq491})$, $(\ref{eq497})$ and the last relation,
we obtain
\begin{equation}
\begin{aligned}
F^{n,3} = &\,
F^{n} + kF_{t}^{n} + a_{2}k^{2}F_{tt}^{n} + a_{1}a_{2}k^{3}F_{ttt}^{n}
- a_{1}a_{2}k^{3}[\tfrac{3}{2}U_{t}^{n}U_{tt}^{n} + \tfrac{1}{2}H_{t}^{n}H_{tt}^{n}]_{x}\\
&\, + a_{1}a_{2}k^{3}[\tfrac{1}{2}P(H_{t}^{n}U_{t}^{n})_{x} + \tfrac{9}{4}U^{n}A(U_{t}^{n}U_{tx}^{n})
+ \tfrac{3}{4}U^{n}A(H_{t}^{n}H_{tx}^{n}) \\
&\qquad \qquad + \tfrac{1}{4}H^{n}P(H_{t}^{n}U_{t}^{n})_{x}\bigr]_{x}
+ w_{7}^{n}\,,
\end{aligned}
\label{eq498}
\end{equation}
where
\[
w_{7}^{n} = \psi_{2x}^{n} + \tfrac{3}{2}w_{6}^{n} + \tfrac{1}{2}v_{7}^{n}\,.
\]
We may now estimate $\delta_{2}^{n}$. By $(\ref{eq494})$ and $(\ref{eq498})$ we obtain
\begin{align*}
b_{3}F^{n,2} + b_{4}F^{n,3}
& = (b_{3}+b_{4})F^{n} + (a_{2}b_{3}+b_{4})kF_{t}^{n} + (a_{1}a_{2}b_{3} + a_{2}b_{4})k^{2}F_{tt}^{n}
+ a_{1}a_{2}b_{4}k^{3}F_{ttt}^{n} \\
&\quad - a_{1}a_{2} b_{3}k^{2}(\tfrac{3}{2}U_{t}^{n}U_{tx}^{n} + \tfrac{1}{2}H_{t}^{n}H_{tx}^{n})
+ b_{3}w_{5}^{n} + b_{4}w_{7}^{n}\,.
\end{align*}
Also, in view of $(\ref{eq488})$
\[
b_{1}F^{n} + b_{2}F^{n,1} =
(b_{1} + b_{2})F^{n} + a_{1}b_{2}kF_{t}^{n}
+ a_{1}^{2}b_{2}k^{2}(\tfrac{3}{2}U_{t}^{n}U_{tx}^{n} + \tfrac{1}{2}H_{t}H_{tx}^{n}) + b_{2}w_{3}^{n}\,.
\]
Therefore
\[
\sum_{j=1}^{4}b_{j}F^{n,j-1} = F^{n} + \tfrac{k}{2}F_{t}^{n} + \tfrac{k^{2}}{6}F_{tt}^{n} + \tfrac{k^{3}}{24}F_{ttt}^{n} + w_{8}^{n}\,,
\]
where
\[
w_{8}^{n} = b_{2}w_{3}^{n} + b_{3}w_{5}^{n} + b_{4}w_{7}^{n}\,.
\]
Hence, from $(\ref{eq482})$ and $(\ref{eq477})$ we obtain
\[
\delta_{2}^{n} = U^{n+1} - U^{n} - kU_{t}^{n} - \tfrac{k^{2}}{2}U_{tt}^{n}
- \tfrac{k^{3}}{6}U_{ttt}^{n} - \tfrac{k^{4}}{24}U_{tttt}^{n} + kAw_{8}^{n} + w_{9}^{n}\,,
\]
where, by $(\ref{eq472})$
\[
\|w_{9}^{n}\|_{1} \leq Ckh^{4}\,.
\]
Since we also have, by previous estimates,
\[
\|w_{8}^{n}\|_{-1} \leq C_{\lambda}h^{4}\,,
\]
we finally conclude that
\[
\|\delta_{2}^{n}\|_{1} \leq C_{\lambda}k(k^{4} + h^{4})\,,
\]
and the proof of the Lemma $4.5$ is completed.
\end{proof}
We show stability and convergence of the fully discrete scheme in the course of the proof of
the following result.
\begin{proposition} Suppose that the solution $(\eta, u)$ of $(\ref{scb})$ is sufficiently smooth
on $[0,T]$. Let $\lambda=k/h$ and $(H_{h}^{n}, U_{h}^{n})$ be the solution of $(\ref{eq473})$-$(\ref{eq474})$.
Then, there exists a positive constant $\lambda_{0}$ and a constant $C$ independent of $k$ and h, such that
for $\lambda \leq \lambda_{0}$,
\[
\max_{0\leq n\leq M}\|\eta(t^{n}) - H_{h}^{n}\| \leq C(k^{4} + h^{3.5}\sqrt{\ln 1/h})\,, \qquad \quad
\max_{0\leq n\leq M}\|u(t^{n}) - U_{h}^{n}\|_{1} \leq C(k^{4} + h^{3})\,.
\]
\end{proposition}
\begin{proof} It suffices to show that
\[
\max_{0\leq n\leq M}(\|H^{n} - H_{h}^{n}\| + \|U^{n} - U_{h}^{n}\|_{1})
\leq C(k^{4} + h^{3.5}\sqrt{\ln 1/h})\,.
\]
Let
\[
\ve^{n} = H^{n} - H_{h}^{n}\,, \quad e^{n} = U^{n} -U_{h}^{n}\,,
\quad \ve^{n,j} = V^{n,j} - H_{h}^{n,j}\,, \quad e^{n,j} = W^{n,j} - U_{h}^{n,j}\,,
\quad j=1,2,3\,.
\]
Then, by $(\ref{eq475})$ and $(\ref{eq478})$, we have for $j=1,2,3$
\begin{equation}
\ve^{n,j} = \ve^{n} - ka_{j}P\bigl(\Phi^{n,j-1} - \Phi_{h}^{n,j-1}\bigr)_{x}\,,
\label{eq499}
\end{equation}
\begin{equation}
e^{n,j}  = e^{n} - ka_{j}A\bigl(F^{n,j-1} - F_{h}^{n,j-1}\bigr)\,,
\label{eq4100}
\end{equation}
and by $(\ref{eq479})$, $(\ref{eq481})$, and $(\ref{eq482})$
\begin{equation}
\ve^{n+1} = \ve^{n} - k\sum_{j=1}^{4}b_{j}P\bigl(\Phi^{n,j-1} - \Phi_{h}^{n,j-1}\bigr)_{x}
+ \delta_{1}^{n}\,,
\label{eq4101}
\end{equation}
\begin{equation}
e^{n+1} = e^{n} - k\sum_{j=1}^{4}b_{j}A\bigl(F^{n,j-1} - F_{h}^{n,j-1}\bigr)
+ \delta_{2}^{n}\,.
\label{eq4102}
\end{equation}
We embark upon the proof by deriving suitable estimates for $\ve^{n,1}$ and $e^{n,1}$. To this end,
putting $\Phi_{h}^{n}:=\Phi_{h}^{n,0}$, we have
\[
\Phi^{n} - \Phi_{h}^{n} = e^{n} + \tfrac{1}{2}(H^{n}U^{n} - H_{h}^{n}U_{h}^{n})\,,
\]
and since
\begin{align*}
H^{n}U^{n} - H_{h}^{n}U_{h}^{n} & = H^{n}(U^{n} - U_{h}^{n}) -
(U^{n} - U_{h}^{n})(H^{n} - H_{h}^{n}) + U^{n}(H^{n}-H_{h}^{n}) \\
& = H^{n}e^{n} - \ve^{n}e^{n} + U^{n}\ve^{n}\,,
\end{align*}
we will have
\[
\Phi^{n} - \Phi_{h}^{n} = e^{n} + \tfrac{1}{2}(H^{n}e^{n} - \ve^{n}e^{n}) + \tfrac{1}{2}U^{n}\ve^{n}\,,
\]
i.e.
\[
\Phi^{n} - \Phi_{h}^{n} = \widetilde{\gamma}_{1}^{n} + \tfrac{1}{2}U^{n}\ve^{n}\,,
\]
where
\[
\widetilde{\gamma}_{1}^{n} = e^{n} + \tfrac{1}{2}(H^{n}e^{n} - \ve^{n}e^{n})\,.
\]
Therefore,
\begin{equation}
P(\Phi^{n} - \Phi_{h}^{n})_{x} = \gamma_{1}^{n} + \tfrac{1}{2}\rho_{1}^{n}\,,
\label{eq4103}
\end{equation}
where
\begin{equation}
\gamma_{1}^{n} = P\widetilde{\gamma}_{1x}^{n}\,, \qquad \quad \rho_{1}^{n} = P(U^{n}\ve^{n})_{x}\,.
\label{eq4104}
\end{equation}
Hence, from $(\ref{eq499})$
\begin{equation}
\ve^{n,1} = \ve^{n} - a_{1}k\gamma_{1}^{n} - \tfrac{a_{1}}{2}k\rho_{1}^{n}\,.
\label{eq4105}
\end{equation}
Let $0\leq n^{*}\leq M-1$ be the maximal integer for which
\[
\|\ve^{n}\|_{1} + \|e^{n}\|_{1} \leq 1\,, \qquad 0\leq n\leq n^{*}\,.
\]
Then, for $n \leq n^{*}$,
\[
\|\widetilde{\gamma}_{1}^{n}\|_{1} \leq C(\|\ve^{n}\| + \|e^{n}\|_{1})\,,
\]
and, in view of $(\ref{eq4104})$,
\begin{equation}
\|\gamma_{1}^{n}\| \leq C(\|\ve^{n}\| + \|e^{n}\|_{1})\,.
\label{eq4106}
\end{equation}
Therefore, by $(\ref{eq499})$, for $0\leq n\leq n^{*}$,
\begin{equation}
\|\ve^{n,1}\| \leq C_{\lambda}(\|\ve^{n}\| + \|e^{n}\|_{1})\,,
\label{eq4107}
\end{equation}
and also
\begin{equation}
\|\ve^{n,1}\|_{1}\leq 1 + C_{\lambda}\|\gamma_{1}^{n}\| + C_{\lambda}\|\rho_{1}^{n}\| \leq C_{\lambda}\,,
\label{eq4108}
\end{equation}
where, as before, $C_{\lambda}$ denotes a constant that is a polynomial of $\lambda$ with positive
coefficients. \par
In order to estimate $e^{n,1}$ we note, setting $F_{h}^{n}=F_{h}^{n,0}$, that
\[
F^{n} - F_{h}^{n} = \ve_{x}^{n} + \tfrac{3}{2}(U^{n}U_{x}^{n} - U_{h}^{n}U_{hx}^{n})
+\tfrac{1}{2}(H^{n}H_{x}^{n} - H_{h}^{n}H_{hx}^{n})\,.
\]
Since
\begin{align*}
U^{n}U_{x}^{n} - U_{h}^{n}U_{hx}^{n} & = U^{n}(U_{x}^{n} - U_{hx}^{n}) -
(U_{x}^{n}-U_{hx}^{n})(U^{n} - U_{h}^{n}) + U_{x}^{n}(U^{n} - U_{h}^{n}) \\
& = (U^{n}e^{n})_{x} - e^{n}e_{x}^{n}\,,
\end{align*}
and
\[
H^{n}H_{x}^{n} - H_{h}^{n}H_{hx}^{n} = (H^{n}\ve^{n})_{x} - \ve^{n}\ve_{x}^{n}\,,
\]
we have
\[
F^{n} - F_{h}^{n} = \ve_{x}^{n} + \tfrac{3}{2}(U^{n}e^{n})_{x} - \tfrac{3}{2}e^{n}e_{x}^{n}
+ \tfrac{1}{2}(H^{n}\ve^{n})_{x} - \tfrac{1}{2}\ve^{n}\ve_{x}^{n}\,,
\]
and therefore, for $n \leq n^{*}$
\begin{equation}
\|F^{n} - F_{h}^{n}\|_{-1} \leq C (\|\ve^{n}\| + \|e^{n}\|)\,.
\label{eq4109}
\end{equation}
Recalling the expression for $e^{n,1}$ in $(\ref{eq4101})$ we have for $n \leq n^{*}$
\begin{equation}
\|e^{n,1}\|_{1} \leq C(\|\ve^{n}\| + \|e^{n}\|_{1})\,.
\label{eq4110}
\end{equation}
We now turn to deriving estimates for $\ve^{n,2}$ and $e^{n,2}$. With this aim in mind, we note
for later reference that it follows from the definition of $V^{n,j}$, $W^{n,j}$,
$F^{n,j}$ that for $0\leq n\leq M-1$, $j=0,1,2,3$
\[
\|V^{n,j}\|_{1} + \|W^{n,j}\|_{1} \leq C_{\lambda}\,, \qquad \|F^{n,j}\|_{-1} \leq C_{\lambda}\,.
\]
Since now
\[
\Phi^{n,1} - \Phi_{h}^{n,1} = e^{n,1} + \tfrac{1}{2}(V^{n,1}W^{n,1} - H_{h}^{n,1}U_{h}^{n,1})\,,
\]
and
\begin{align*}
V^{n,1}W^{n,1} - H_{h}^{n,1}U_{h}^{n,1} & = V^{n,1}(W^{n,1} - U_{h}^{n,1})
-(W^{n,1} - U_{h}^{n,1})(V^{n,1} - H_{h}^{n,1}) + W^{n,1}(V^{n,1} - H_{h}^{n,1}) \\
& = V^{n,1}e^{n,1} - e^{n,1}\ve^{n,1} + W^{n,1}\ve^{n,1}\,,
\end{align*}
using $(\ref{eq480})$ we have by $(\ref{eq4105})$
\begin{align*}
W^{n,1}\ve^{n,1} & = (U^{n} - a_{1}kAF^{n})(\ve^{n} - a_{1}k\gamma_{1}^{n} -\tfrac{a_{1}k}{2}\rho_{1}^{n})\\
& = U^{n}\ve^{n} - \tfrac{a_{1}k}{2}U^{n}\rho_{1}^{n}-a_{1}kU^{n}\gamma_{1}^{n}
-a_{1}k(AF^{n})\ve^{n,1}\,,
\end{align*}
and so
\[
\Phi^{n,1} - \Phi_{h}^{n,1} = \tfrac{1}{2}U^{n}\ve^{n} - \tfrac{a_{1}k}{4}U^{n}\rho_{1}^{n}
+ \widetilde{\gamma}_{2}^{n}\,,
\]
where
\begin{equation}
\widetilde{\gamma}_{2}^{n} = e^{n,1} + \tfrac{1}{2}V^{n,1}e^{n,1} -\tfrac{1}{2}e^{n,1}\ve^{n,1}
- \tfrac{a_{1}k}{2}U^{n}\gamma_{1}^{n} -\tfrac{a_{1}k}{2}(AF^{n})\ve^{n,1}\,.
\label{eq4111}
\end{equation}
Recalling the definition of $\rho_{1}^{n}$ in $(\ref{eq4104})$ we obtain
\begin{equation}
P(\Phi^{n,1} - \Phi_{h}^{n,1})_{x} = \tfrac{1}{2}\rho_{1}^{n} - \tfrac{a_{1}k}{4}\rho_{2}^{n}
 + \gamma_{2}^{n}\,,
\label{eq4112}
\end{equation}
where
\begin{equation}
\gamma_{2}^{n} = P(\widetilde{\gamma}_{2}^{n})_{x}\,,
\qquad \quad \rho_{2}^{n} = P(U^{n}\rho_{1}^{n})_{x}\,.
\label{eq4113}
\end{equation}
By $(\ref{eq4111})$ and $(\ref{eq4108})$ we deduce that
\[
\|\widetilde{\gamma}_{2}^{n}\|_{1} \leq C_{\lambda}\|e^{n,1}\|_{1}
+ Ck\|\gamma_{1}^{n}\|_{1} + C_{\lambda}k\|\ve^{n,1}\|_{1}\,.
\]
Hence, by $(\ref{eq4110})$ and $(\ref{eq4106})$, $(\ref{eq4107})$ we have
\[
\|\widetilde{\gamma}_{2}^{n}\|_{1} \leq C_{\lambda}(\|\ve^{n}\| + \|e^{n}\|_{1})\,,
\]
and therefore, by $(\ref{eq4113})$ we see that
\begin{equation}
\|\gamma_{2}^{n}\| \leq C_{\lambda}(\|\ve^{n}\| + \|e^{n}\|_{1})\,,
\label{eq4114}
\end{equation}
holds for $n\leq n^{*}$. Now, by $(\ref{eq480})$ and $(\ref{eq478})$
\[
\ve^{n,2} = V^{n,2} - H_{h}^{n,2} = \ve^{n} - a_{2}kP(\Phi^{n,1} - \Phi_{h}^{n,1})_{x}\,,
\]
and from $(\ref{eq4112})$
\begin{equation}
\ve^{n,2} = \ve^{n} - \tfrac{a_{2}k}{2}\rho_{1}^{n} + \tfrac{a_{1}a_{2}k^{2}}{4}\rho_{2}^{n}
- a_{2}k\gamma_{2}^{n}\,.
\label{eq4115}
\end{equation}
Thus,
\[
\|\ve^{n,2}\| \leq \|\ve^{n}\| + Ck\|\rho_{1}^{n}\| + Ck^{2}\|\rho_{2}^{n}\| + Ck\|\gamma_{2}^{n}\|\,,
\]
and using $(\ref{eq4113})$, $(\ref{eq4104})$ and $(\ref{eq4114})$ we conclude that for
$n \leq n^{*}$
\begin{equation}
\|\ve^{n,2}\| \leq C_{\lambda}(\|\ve^{n}\| + \|e^{n}\|_{1})\,,
\label{eq4116}
\end{equation}
and
\begin{equation}
\|\ve^{n,2}\|_{1} \leq C_{\lambda}\,.
\label{eq4117}
\end{equation}
We next consider $e^{n,2}$. Since
\[
F^{n,1} - F_{h}^{n,1} = \ve_{x}^{n,1} + \tfrac{3}{2}(W^{n,1}W_{x}^{n,1} - U_{h}^{n,1}U_{hx}^{n,1})
+ \tfrac{1}{2}(V^{n,1}V_{x}^{n,1} - H_{h}^{n,1}H_{hx}^{n,1})\,,
\]
\begin{align*}
W^{n,1}W_{x}^{n,1} - U_{h}^{n,1}U_{hx}^{n,1} & = W^{n,1}(W_{x}^{n,1} - U_{hx}^{n,1}) -
(W_{x}^{n,1}  - U_{hx}^{n,1})(W^{n,1} - U_{h}^{n,1}) + W_{x}^{n,1}(W^{n,1} - U_{h}^{n,1})\\
& = (W^{n,1}e^{n,1})_{x} - e_{x}^{n,1}e^{n,1}\,,
\end{align*}
and
\[
V^{n,1}V_{x}^{n,1} - H_{h}^{n,1}H_{hx}^{n,1} = (V^{n,1}\ve^{n,1})_{x} - \ve_{x}^{n,1}\ve^{n,1}\,,
\]
we obtain by $(\ref{eq4108})$, for $n\leq n^{*}$
\[
\|F^{n,1} - F_{h}^{n,1}\|_{-1} \leq C_{\lambda}(\|\ve^{n,1}\| + \|e^{n,1}\|_{1})\,,
\]
giving, in view of $(\ref{eq4107})$ and $(\ref{eq4110})$ for $n\leq n^{*}$
\begin{equation}
\|F^{n,1} - F_{h}^{n,1}\|_{-1} \leq C_{\lambda}(\|\ve^{n}\| + \|e^{n}\|_{1})\,.
\label{eq4118}
\end{equation}
Therefore, by $(\ref{eq4100})$ we conclude that for $n\leq n^{*}$
\begin{equation}
\|e^{n,2}\|_{1} \leq C_{\lambda}(\|\ve^{n}\| + \|e^{n}\|_{1})\,.
\label{eq4119}
\end{equation}
\par
To obtain analogous bounds for $\ve^{n,3}$, $e^{n,3}$ note that
\[
\Phi^{n,2} - \Phi_{h}^{n,2} = e^{n,2} + \tfrac{1}{2}(V^{n,2}W^{n,2} - H_{h}^{n,2}U_{h}^{n,2})\,,
\]
and
\begin{align*}
V^{n,2}W^{n,2} - H_{h}^{n,2}U_{h}^{n,2} & = V^{n,2}(W^{n,2} - U_{h}^{n,2})
-(W^{n,2} - U_{h}^{n,2})(V^{n,2} - H_{h}^{n,2}) + W^{n,2}(V^{n,2} - H_{h}^{n,2}) \\
& = V^{n,2}e^{n,2} - e^{n,2}\ve^{n,2} + W^{n,2}\ve^{n,2}\,.
\end{align*}
By $(\ref{eq480})$ and $(\ref{eq4115})$ we see that
\begin{align*}
W^{n,2}\ve^{n,2} & = (U^{n} - a_{2}kAF^{n,1})(\ve^{n} -\tfrac{a_{2}k}{2}\rho_{1}^{n}
+ \tfrac{a_{1}a_{2}k^{2}}{4}\rho_{2}^{n} - a_{2}k\gamma_{2}^{n}) \\
& = U^{n}\ve^{n} - \tfrac{a_{2}k}{2}U^{n}\rho_{1}^{n} + \tfrac{a_{1}a_{2}k^{2}}{4}U^{n}\rho_{2}^{n}
- a_{2}kU^{n}\gamma_{2}^{n} - a_{2}k(AF^{n,1})\ve^{n,2}\,.
\end{align*}
Therefore
\[
\Phi^{n,2} - \Phi_{h}^{n,2} = \tfrac{1}{2}U^{n}\ve^{n} - \tfrac{a_{2}k}{4}U^{n}\rho_{1}^{n}
+ \tfrac{a_{1}a_{2}k^{2}}{8}U^{n}\rho_{2}^{n} + \widetilde{\gamma}_{3}^{n}\,,
\]
where
\begin{equation}
\widetilde{\gamma}_{3}^{n} = e^{n,2} + \tfrac{1}{2}V^{n,2}e^{n,2} - \tfrac{1}{2}e^{n,2}\ve^{n,2}
- \tfrac{a_{2}k}{2}U^{n}\gamma_{2}^{n} - \tfrac{a_{2}k}{2}(AF^{n,1})\ve^{n,2}\,.
\label{eq4120}
\end{equation}
So, by $(\ref{eq4104})$ and $(\ref{eq4113})$ we have
\begin{equation}
P(\Phi^{n,2} - \Phi_{h}^{n,2})_{x} = \tfrac{1}{2}\rho_{1}^{n}  - \tfrac{a_{2}k}{4}\rho_{2}^{n}
+ \tfrac{a_{1}a_{2}k^{2}}{8}\rho_{3}^{n} + \gamma_{3}^{n}\,,
\label{eq4121}
\end{equation}
where
\begin{equation}
\gamma_{3}^{n} = P(\widetilde{\gamma}_{3}^{n})_{x}\,, \qquad \quad \rho_{3}^{n}=P(U^{n}\rho_{2}^{n})_{x}\,.
\label{eq4122}
\end{equation}
Now, by $(\ref{eq4120})$ and $(\ref{eq4117})$ we have for $n\leq n^{*}$
\[
\|\widetilde{\gamma}_{3}^{n}\|_{1} \leq C_{\lambda}\|e^{n,2}\|_{1} + Ck\|\gamma_{2}^{n}\|_{1}
+ C_{\lambda}k\|\ve^{n,2}\|_{1}\,.
\]
Hence, by $(\ref{eq4119})$, $(\ref{eq4114})$ and $(\ref{eq4116})$,
\[
\|\widetilde{\gamma}_{3}^{n}\|_{1} \leq C_{\lambda}(\|\ve^{n}\| + \|e^{n}\|_{1})\,,
\]
and thus
\begin{equation}
\|\gamma_{3}^{n}\| \leq C_{\lambda}(\|\ve^{n}\| + \|e^{n}\|_{1})\,,
\label{eq4123}
\end{equation}
for $n\leq n^{*}$. Now
\[
\ve^{n,3} = V^{n,3} - H_{h}^{n,3} = \ve^{n} - kP(\Phi^{n,2} - \Phi_{h}^{n,2})_{x}\,,
\]
and from $(\ref{eq4121})$
\begin{equation}
\ve^{n,3} = \ve^{n} - \tfrac{k}{2}\rho_{1}^{n} + \tfrac{a_{2}k^{2}}{4}\rho_{2}^{n}
-\tfrac{a_{1}a_{2}k^{3}}{8}\rho_{3}^{n} - k\gamma_{3}^{n}\,.
\label{eq4124}
\end{equation}
Therefore,
\[
\|\ve^{n,3}\| \leq \|\ve^{n}\| + Ck\|\rho_{1}^{n}\| + Ck^{2}\|\rho_{2}^{n}\| +
Ck^{3}\|\rho_{3}^{n}\| + Ck\|\gamma_{3}^{n}\|\,,
\]
and by $(\ref{eq4122})$, $(\ref{eq4113})$, $(\ref{eq4104})$ and $(\ref{eq4123})$ we deduce that
for $n\leq n^{*}$
\begin{equation}
\|\ve^{n,3}\| \leq C_{\lambda}(\|\ve^{n}\| + \|e^{n}\|_{1})\,,
\label{eq4125}
\end{equation}
and also
\begin{equation}
\|\ve^{n,3}\|_{1} \leq C_{\lambda}\,.
\label{eq4126}
\end{equation}
We consider now $e^{n,3}$. Since
\[
F^{n,2} - F_{h}^{n,2} = \ve_{x}^{n,2} + \tfrac{3}{2}(W^{n,2}W_{x}^{n,2} - U_{h}^{n,2}U_{hx}^{n,2})
+ \tfrac{1}{2}(V^{n,2}V_{x}^{n,2} - H_{h}^{n,2}H_{hx}^{n,2})\,,
\]
\begin{align*}
W^{n,2}W_{x}^{n,2} - U_{h}^{n,2}U_{hx}^{n,2} & = W^{n,2}(W_{x}^{n,2} - U_{hx}^{n,2}) -
(W_{x}^{n,2}  - U_{hx}^{n,2})(W^{n,2} - U_{h}^{n,2}) + W_{x}^{n,2}(W^{n,2} - U_{h}^{n,2})\\
& = (W^{n,2}e^{n,2})_{x} - e_{x}^{n,2}e^{n,2}\,,
\end{align*}
and
\[
V^{n,2}V_{x}^{n,2} - H_{h}^{n,2}H_{hx}^{n,2} = (V^{n,2}\ve^{n,2})_{x} - \ve_{x}^{n,2}\ve^{n,2}\,,
\]
we have by $(\ref{eq4117})$ for $n \leq n^{*}$
\[
\|F^{n,2} - F_{h}^{n,2}\|_{-1} \leq C_{\lambda}(\|\ve^{n,2}\| + \|e^{n,2}\|_{1})\,,
\]
and from $(\ref{eq4116})$ and $(\ref{eq4119})$, for $n \leq n^{*}$
\begin{equation}
\|F^{n,2} - F_{h}^{n,2}\|_{-1} \leq C_{\lambda}(\|\ve^{n}\| + \|e^{n}\|_{1})\,.
\label{eq4127}
\end{equation}
Therefore, we conclude by $(\ref{eq4100})$, for $n \leq n^{*}$
\begin{equation}
\|e^{n,3}\|_{1} \leq C_{\lambda}(\|\ve^{n}\| + \|e^{n}\|_{1})\,.
\label{eq4128}
\end{equation}
We have now reached the final stage of the estimation of the errors of the fully discrete scheme.
First, we estimate $e^{n+1}$:
Since
\[
F^{n,3} - F_{h}^{n,3} = \ve_{x}^{n,3} + \tfrac{3}{2}(W^{n,3}W_{x}^{n,3} - U_{h}^{n,3}U_{hx}^{n,3})
+ \tfrac{1}{2}(V^{n,3}V_{x}^{n,3} - H_{h}^{n,3}H_{hx}^{n,3})\,,
\]
\[
W^{n,3}W_{x}^{n,3} - U_{h}^{n,3}U_{hx}^{n,3} = (W^{n,3}e^{n,3})_{x} - e_{x}^{n,3}e^{n,3}\,,
\]
and
\[
V^{n,3}V_{x}^{n,3} - H_{h}^{n,3}H_{hx}^{n,3} =(V^{n,3}\ve^{n,3})_{x} - \ve_{x}^{n,3}\ve^{n,3}\,,
\]
we have from $(\ref{eq4126})$ for $n \leq n^{*}$,
\[
\|F^{n,3} - F_{h}^{n,3}\|_{-1} \leq C_{\lambda}(\|\ve^{n,3}\| + \|e^{n,3}\|_{1})\,,
\]
and therefore, from $(\ref{eq4125})$, $(\ref{eq4128})$
\[
\|F^{n,3} - F_{h}^{n,3}\|_{-1} \leq C_{\lambda}(\|\ve^{n}\| + \|e^{n}\|_{1})\,.
\]
Thus, by $(\ref{eq4102})$, $(\ref{eq4109})$, $(\ref{eq4118})$, $(\ref{eq4127})$, and above
inequality, we obtain for $n\leq n^{*}$
\begin{equation}
\|e^{n+1}\|_{1} \leq \|e^{n}\|_{1} + C_{\lambda}k(\|\ve^{n}\| + \|e^{n}\|_{1}) + \|\delta_{2}^{n}\|_{1}\,,
\label{eq4129}
\end{equation}
which is the required recursive inequality for $\|e^{n+1}\|_{1}$.
To obtain an analogous relation for $\|\ve^{n+1}\|$ requires more work. Observe that
\[
\Phi^{n,3} - \Phi_{h}^{n,3} = e^{n,3} + \tfrac{1}{2}(V^{n,3}W^{n,3} - H_{h}^{n,3}U_{h}^{n,3})\,,
\]
and
\begin{align*}
V^{n,3}W^{n,3} - H_{h}^{n,3}U_{h}^{n,3} & = V^{n,3}(W^{n,3} - U_{h}^{n,3})
-(W^{n,3} - U_{h}^{n,3})(V^{n,3} - H_{h}^{n,3}) + W^{n,3}(V^{n,3} - H_{h}^{n,3}) \\
& = V^{n,3}e^{n,3} - e^{n,3}\ve^{n,3} + W^{n,3}\ve^{n,3}\,.
\end{align*}
By $(\ref{eq480})$ and $(\ref{eq4124})$ we obtain
\begin{align*}
W^{n,3}\ve^{n,3} & = (U^{n} - kAF^{n,2})(\ve^{n} -\tfrac{k}{2}\rho_{1}^{n}
+ \tfrac{a_{2}k^{2}}{4}\rho_{2}^{n} - \tfrac{a_{1}a_{2}k^{3}}{8}\rho_{3}^{n}
- k\gamma_{3}^{n}) \\
& = U^{n}\ve^{n} - \tfrac{k}{2}U^{n}\rho_{1}^{n} + \tfrac{a_{2}k^{2}}{4}U^{n}\rho_{2}^{n}
- \tfrac{a_{1}a_{2}k^{3}}{8}U^{n}\rho_{3}^{n}
- kU^{n}\gamma_{3}^{n} - k(AF^{n,2})\ve^{n,3}\,.
\end{align*}
Thus
\[
\Phi^{n,3} - \Phi_{h}^{n,3} = \tfrac{1}{2}U^{n}\ve^{n} - \tfrac{k}{4}U^{n}\rho_{1}^{n}
+ \tfrac{a_{2}k^{2}}{8}U^{n}\rho_{2}^{n} - \tfrac{a_{1}a_{2}k^{3}}{16}U^{n}\rho_{3}^{n}
+ \widetilde{\gamma}_{4}^{n}\,,
\]
where
\begin{equation}
\widetilde{\gamma}_{4}^{n} = e^{n,3} + \tfrac{1}{2}V^{n,3}e^{n,3} - \tfrac{1}{2}e^{n,3}\ve^{n,3}
- \tfrac{k}{2}U^{n}\gamma_{3}^{n} - \tfrac{k}{2}(AF^{n,2})\ve^{n,3}\,.
\label{eq4130}
\end{equation}
Hence, by $(\ref{eq4104})$, $(\ref{eq4113})$, and $(\ref{eq4132})$ we obtain
\begin{equation}
P(\Phi^{n,3} - \Phi_{h}^{n,3})_{x} = \tfrac{1}{2}\rho_{1}^{n}  - \tfrac{k}{4}\rho_{2}^{n}
+ \tfrac{a_{2}k^{2}}{8}\rho_{3}^{n} -\tfrac{a_{1}a_{2}k^{3}}{16}\rho_{4}^{n}
+ \gamma_{4}^{n}\,,
\label{eq4131}
\end{equation}
where
\begin{equation}
\gamma_{4}^{n} = P(\widetilde{\gamma}_{4}^{n})_{x}\,, \qquad \quad \rho_{4}^{n}=P(U^{n}\rho_{3}^{n})_{x}\,.
\label{eq4132}
\end{equation}
By $(\ref{eq4130})$ and $(\ref{eq4126})$ we obtain for $n\leq n^{*}$
\[
\|\widetilde{\gamma}_{4}^{n}\|_{1} \leq C_{\lambda}\|e^{n,3}\|_{1} + Ck\|\gamma_{3}^{n}\|_{1}
+ C_{\lambda}k\|\ve^{n,3}\|_{1}\,,
\]
from which, using $(\ref{eq4128})$, $(\ref{eq4123})$ and $(\ref{eq4125})$ we see that for $n\leq n^{*}$
\[
\|\widetilde{\gamma}_{4}^{n}\|_{1} \leq C_{\lambda}(\|\ve^{n}\| + \|e^{n}\|_{1})\,,
\]
and
\begin{equation}
\|\gamma_{4}^{n}\| \leq C_{\lambda}(\|\ve^{n}\| + \|e^{n}\|_{1})\,.
\label{eq4133}
\end{equation}
\par
From $(\ref{eq4103})$ and $(\ref{eq4112})$ we obtain
\begin{align*}
b_{1}P(\Phi^{n} - \Phi_{h}^{n})_{x} + b_{2}P(\Phi^{n,1}-\Phi_{h}^{n,1})_{x} &=
\tfrac{1}{2}(b_{1}+b_{2})\rho_{1}^{n} - \tfrac{a_{1}b_{2}k}{4}\rho_{2}^{n} + b_{1}\gamma_{1}
+ b_{2}\gamma_{2}\\
& = \tfrac{1}{4}\rho_{1}^{n} - \tfrac{k}{24}\rho_{2}^{n} + \tfrac{1}{6}\gamma_{1}^{n}
+ \tfrac{1}{3}\gamma_{2}^{n}\,.
\end{align*}
From $(\ref{eq4121})$ and $(\ref{eq4131})$ we see that
\begin{align*}
b_{3}P(\Phi^{n,2} - \Phi_{h}^{n,2})_{x} + b_{4}P(\Phi^{n,3}-\Phi_{h}^{n,3})_{x} &=
\tfrac{1}{2}(b_{3}+b_{4})\rho_{1}^{n} - \tfrac{(b_{3}a_{2}+b_{4})k}{4}\rho_{2}^{n}
+ \tfrac{(b_{3}a_{1}a_{2} + b_{4}a_{2})k^{2}}{8}\rho_{3}^{n} \\
& \qquad - \tfrac{b_{4}a_{1}a_{2}k^{3}}{16}\rho_{4}^{n}
+ b_{3}\gamma_{3}^{n} + b_{4}\gamma_{4}^{n}\\
&= \tfrac{1}{4}\rho_{1}^{n} - \tfrac{k}{12}\rho_{2}^{n} + \tfrac{k^{2}}{48}\rho_{3}^{n}
-\tfrac{k^{3}}{24\cdot 16}\rho_{4}^{n} + \tfrac{1}{3}\gamma_{3}^{n} + \tfrac{1}{6}\gamma_{4}^{n}\,.
\end{align*}
Therefore, in view of $(\ref{eq4101})$, we have
\begin{equation}
\ve^{n+1}=\ve^{n} - \tfrac{k}{2}\rho_{1}^{n} + \tfrac{k^{2}}{8}\rho_{2}^{n}
- \tfrac{k^{3}}{48}\rho_{3}^{n} +\tfrac{k^{4}}{24\cdot 16}\rho_{4}^{n} - k\gamma_{5}^{n}
+ \delta_{1}^{n}\,,
\label{eq4134}
\end{equation}
where
\[
\gamma_{5}^{n} = \tfrac{1}{6}(\gamma_{1}^{n} + \gamma_{4}^{n}) +
\tfrac{1}{3}(\gamma_{2}^{n}+\gamma_{3}^{n})\,,
\]
satisfies, in view of $(\ref{eq4106})$, $(\ref{eq4114})$, $(\ref{eq4123})$ and $(\ref{eq4133})$,
the inequality
\begin{equation}
\|\gamma_{5}^{n}\| \leq C_{\lambda}(\|\ve^{n}\| + \|e^{n}\|_{1})\,,
\label{eq4135}
\end{equation}
for $n\leq n^{*}$. We begin now the estimation of
\begin{equation}
\sigma^{n}:=\ve^{n} - \tfrac{k}{2}\rho_{1}^{n} + \tfrac{k^{2}}{8}\rho_{2}^{n}
- \tfrac{k^{3}}{48}\rho_{3}^{n} +\tfrac{k^{4}}{24\cdot 16}\rho_{4}^{n}\,.
\label{eq4136}
\end{equation}
We have
\begin{align*}
\|\sigma^{n}\|^{2} & = \|\ve^{n}\|^{2} - k(\ve^{n},\rho_{1}^{n}) +
\tfrac{k^{2}}{4}\bigl[\|\rho_{1}^{n}\|^{2} + (\ve^{n},\rho_{2}^{n})\bigr]
-\tfrac{k^{3}}{24}\bigl[(\ve^{n},\rho_{3}^{n}) + 3(\rho_{1}^{n},\rho_{2}^{n})\bigr] \\
&\quad + \tfrac{k^{4}}{64}\bigl[\|\rho_{2}^{n}\|^{2} + \tfrac{1}{3}(\ve^{n},\rho_{4}^{n})
+\tfrac{4}{3}(\rho_{1}^{n},\rho_{3}^{n})\bigr]
-\tfrac{k^{5}}{24\cdot 16}\bigl[(\rho_{1}^{n},\rho_{4}^{n}) + 2(\rho_{2}^{n},\rho_{3}^{n})\bigr]\\
&\quad + \tfrac{k^{6}}{48^{2}}\bigl[\|\rho_{3}^{n}\|^{2} + \tfrac{3}{2}(\rho_{2}^{n},\rho_{4}^{n})\bigr]
-\tfrac{k^{7}}{24^{2}\cdot 16}(\rho_{3}^{n},\rho_{4}^{n}) +
\tfrac{k^{8}}{24^{2}\cdot 16^{2}}\|\rho_{4}^{n}\|^{2}\,,
\end{align*}
or
\begin{equation}
\|\sigma^{n}\|^{2} =\|\ve^{n}\|^{2} - k f_{1}^{n} + \tfrac{k^{2}}{4}f_{2}^{n}
-\tfrac{k^{3}}{24}f_{3}^{n} + \tfrac{k^{4}}{64}f_{4}^{n}
-\tfrac{k^{5}}{24\cdot 16}f_{5}^{n} + \tfrac{k^{6}}{48^{2}}f_{6}^{n}
-\tfrac{k^{7}}{24^{2}\cdot 16}f_{7}^{n} + \tfrac{k^{8}}{24^{2}\cdot 16^{2}}\|\rho_{4}^{n}\|^{2}\,,
\label{eq4137}
\end{equation}
where the numbers $f_{1}^{n},\dots,f_{7}^{n}$ have been defined in the obvious manner.
By $(\ref{eq4104})$ we have
\[
f_{1}^{n} = (\ve^{n},\rho_{1}^{n}) = -(U^{n}\ve_{x}^{n},\ve^{n})
=\tfrac{1}{2}(U_{x}^{n}\ve^{n},\ve^{n})\,,
\]
and therefore
\begin{equation}
|f_{1}^{n}| \leq C \|\ve^{n}\|^{2}\,.
\label{eq4138}
\end{equation}
By $(\ref{eq4113})$ and $(\ref{eq4104})$
\[
(\ve^{n},\rho_{2}^{n}) = - (U^{n}\ve_{x}^{n},\rho_{1}^{n})=-((U^{n}\ve^{n})_{x},\rho_{1}^{n})
+ (U_{x}^{n}\ve^{n},\rho_{1}^{n}) = - \|\rho_{1}^{n}\|^{2} + (U_{x}^{n}\ve^{n},\rho_{1}^{n})\,.
\]
Since
\[
f_{2}^{n} = \|\rho_{1}^{n}\|^{2} + (\ve^{n},\rho_{2}^{n})\,,
\]
we conclude that
\begin{equation}
f_{2}^{n} =  (U_{x}^{n}\ve^{n},\rho_{1}^{n}) \leq \tfrac{C}{h}\|\ve^{n}\|^{2}\,.
\label{eq4139}
\end{equation}
Now
\[
(\rho_{1}^{n},\rho_{2}^{n}) = \tfrac{1}{2}(U_{x}^{n}\rho_{1}^{n},\rho_{1}^{n})\,.
\]
In addition, using $(\ref{eq4122})$, $(\ref{eq4104})$, and $(\ref{eq4113})$ we have
\[
(\ve^{n},\rho_{3}^{n}) = -(U^{n}\ve_{x}^{n},\rho_{2}^{n}) = - (\rho_{1}^{n},\rho_{2}^{n})
+ (U_{x}^{n}\ve^{n},\rho_{2}^{n}) = -\tfrac{1}{2}(U_{x}^{n}\rho_{1}^{n},\rho_{1}^{n})
+ (U_{x}^{n}\ve^{n},\rho_{2}^{n})\,.
\]
Hence, since
\[
f_{3}^{n} = (\ve^{n},\rho_{3}^{n}) + 3(\rho_{1}^{n},\rho_{2}^{n})\,,
\]
we have
\[
f_{3}^{n} = (U_{x}^{n}\rho_{1}^{n},\rho_{1}^{n}) + (U_{x}^{n}\ve^{n},\rho_{2}^{n})\,,
\]
from which
\begin{equation}
|f_{3}^{n}| \leq \tfrac{C}{h^{2}}\|\ve^{n}\|^{2}\,.
\label{eq4140}
\end{equation}
Now
\[
(\rho_{1}^{n},\rho_{3}^{n}) = -(U^{n}\rho_{1x}^{n},\rho_{2}^{n})=
-((U^{n}\rho_{1}^{n})_{x},\rho_{2}^{n}) + (U_{x}^{n}\rho_{1}^{n},\rho_{2}^{n})
= - \|\rho_{2}^{n}\|^{2} + (U_{x}^{n}\rho_{1}^{n},\rho_{2}^{n})\,,
\]
and by $(\ref{eq4132})$
\begin{align*}
(\ve^{n},\rho_{4}^{n}) & = - (U^{n}\ve_{x}^{n},\rho_{3}^{n}) = -((U^{n}\ve^{n})_{x},\rho_{3}^{n})
+ (U_{x}^{n}\ve^{n},\rho_{3}^{n}) = - (\rho_{1}^{n},\rho_{3}^{n})
+ (U_{x}^{n}\ve^{n},\rho_{3}^{n})\\
& = \|\rho_{2}^{n}\|^{2} - (U_{x}^{n}\rho_{1}^{n},\rho_{2}^{n}) + (U_{x}^{n}\ve^{n},\rho_{3}^{n})\,.
\end{align*}
So, since
\[
f_{4}^{n} = \|\rho_{2}^{n}\|^{2} + \tfrac{1}{3}(\ve^{n},\rho_{4}^{n}) +
\tfrac{4}{3}(\rho_{1}^{n},\rho_{3}^{n})\,,
\]
we have
\[
f_{4}^{n} = (U_{x}^{n}\rho_{1}^{n},\rho_{2}^{n}) + \tfrac{1}{3}(U_{x}^{n}\ve^{n},\rho_{3}^{n})\,,
\]
and thus
\begin{equation}
|f_{4}^{n}| \leq \tfrac{C}{h^{3}}\|\ve^{n}\|^{2}\,.
\label{eq4141}
\end{equation}
We also have
\[
(\rho_{2}^{n},\rho_{3}^{n}) = \tfrac{1}{2}(U_{x}^{n}\rho_{2}^{n},\rho_{2}^{n})\,,
\]
and
\begin{align*}
(\rho_{1}^{n},\rho_{4}^{n}) & = - (U^{n}\rho_{1x}^{n},\rho_{3}^{n})=
-((U^{n}\rho_{1}^{n})_{x},\rho_{3}^{n}) + (U_{x}^{n}\rho_{1}^{n},\rho_{3}^{n})
= -(\rho_{2}^{n},\rho_{3}^{n}) + (U_{x}^{n}\rho_{1}^{n},\rho_{3}^{n}) \\
& = -\tfrac{1}{2}(U_{x}^{n}\rho_{2}^{n},\rho_{2}^{n}) + (U_{x}^{n}\rho_{1}^{n},\rho_{3}^{n})\,.
\end{align*}
Since
\[
f_{5}^{n} = (\rho_{1}^{n},\rho_{4}^{n}) + 2(\rho_{2}^{n},\rho_{3}^{n})\,,
\]
we have
\[
f_{5}^{n} = \tfrac{1}{2}(U_{x}^{n}\rho_{2}^{n},\rho_{2}^{n}) + (U_{x}^{n}\rho_{1}^{n},\rho_{3}^{n})\,,
\]
and therefore
\begin{equation}
|f_{5}^{n}| \leq \tfrac{C}{h^{4}}\|\ve^{n}\|^{2}\,.
\label{eq4142}
\end{equation}
Similarly,
\[
(\rho_{2}^{n},\rho_{4}^{n}) = -(U^{n}\rho_{2x}^{n},\rho_{3}^{n}) =
-((U^{n}\rho_{2}^{n})_{x},\rho_{3}^{n}) + (U_{x}^{n}\rho_{2}^{n},\rho_{3}^{n}) =
-\|\rho_{3}^{n}\|^{2} + (U_{x}^{n}\rho_{2}^{n},\rho_{3}^{n})\,.
\]
Since
\[
f_{6}^{n} = \|\rho_{3}^{n}\|^{2} + \tfrac{3}{2}(\rho_{2}^{n},\rho_{4}^{n})\,,
\]
we have
\begin{equation}
f_{6}^{n} = \tfrac{3}{2}(U_{x}^{n}\rho_{2}^{n},\rho_{3}^{n}) - \tfrac{1}{2}\|\rho_{3}^{n}\|^{2} \leq
\tfrac{C}{h^{5}}\|\ve^{n}\|^{2} - \tfrac{1}{2}\|\rho_{3}^{n}\|^{2}\,.
\label{eq4143}
\end{equation}
Finally,
\[
f_{7}^{n} = (\rho_{3}^{n},\rho_{4}^{n}) = \tfrac{1}{2}(U_{x}^{n}\rho_{3}^{n},\rho_{3}^{n})\,,
\]
and thus
\begin{equation}
|f_{7}^{n}| \leq \tfrac{C}{h^{6}} \|\ve^{n}\|^{2}\,.
\label{eq4144}
\end{equation}
Now, from the inequalities $(\ref{eq4138})$-$(\ref{eq4144})$ and the fact that
\[
\|\rho_{4}^{n}\|\leq \tfrac{\widetilde{C}}{h}\|\rho_{3}^{n}\|\,,
\]
where $\widetilde{C}$ is a constant independent of $k$ and $h$, we obtain in $(\ref{eq4137})$
\[
\|\sigma^{n}\|^{2} \leq (1+C_{\lambda}k)\|\ve^{n}\|^{2} + \tfrac{k^{6}}{2\cdot 48^{2}}
(\tfrac{\widetilde{C}\lambda^{2}}{32} - 1)\|\rho_{3}^{n}\|^{2}\,.
\]
Thus, for $\lambda \leq \lambda_{0}=\sqrt{32/\widetilde{C}}$ we will have the following estimate of
$\|\sigma^{n}\|$
\[
\|\sigma^{n}\| \leq (1+C_{\lambda}k)\|\ve^{n}\|\,.
\]
Therefore, from $(\ref{eq4134})$, $(\ref{eq4136})$ and $(\ref{eq4135})$, we finally obtain
\[
\|\ve^{n+1}\| \leq \|\ve^{n}\| + C_{\lambda}k(\|\ve^{n}\| + \|e^{n}\|_{1}) + \|\delta_{1}^{n}\|\,,
\]
for $n=0,1,\dots,n^{*}$. From this inequality, $(\ref{eq4129})$ and Lemma  $4.5$ we obtain
for $n\leq n^{*}$
\[
\|\ve^{n+1}\| + \|e^{n+1}\|_{1} \leq (1+C_{\lambda}k)(\|\ve^{n}\| + \|e^{n}\|_{1})
+Ck(k^{4} + h^{3.5}\sqrt{\ln 1/h})\,.
\]
Therefore, by Gronwall's Lemma, and $(\ref{eq469})$ we see that
\[
\|\ve^{n}\| + \|e^{n}\|_{1} \leq C_{\lambda}(k^{4} + h^{3.5}\sqrt{\ln 1/h})\,,
\]
for $n=0,1,\dots,n^{*}$. Hence, if $k$ and $h$ are sufficiently small, the maximality of $n^{*}$
contradicted and we may take $n^{*}=M-1$. The conclusion of the proposition follows.
\end{proof}
\begin{remark} Arguing as in Remark $4.1$ we have again
$\|U_{h}^{n} - u(t^{n})\|_{\infty} = O(k^{4} + h^{3.5}\sqrt{\ln 1/h})$.
\end{remark}
\begin{remark}
The analogous results of Lemma $4.5$ of Proposition $4.3$ hold for the system $(\ref{cb})$ as well.
\end{remark}
%
\section{A non-standard Galerkin method}
In this section we analyze a semidiscrete Galerkin approximation of $(\ref{cb})$ and $(\ref{scb})$
which is non-standard in the sense that it approximates $\eta$ by piecewise linear, continuous
functions and $u$ by piecewise quadratic, $C_{0}^{1}$ functions, defined on the same mesh. (The method
may be generalized using the analogous pairs of higher-order splines but here we confine ourselves to the
low-order case.) In the analysis, we consider only the uniform mesh case, where a series of
superaccuracy results for the error of the $L^{2}$-projection onto the space of piecewise linear
continuous functions affords proving optimal-order error estimates.
\subsection{Preliminaries and the basic results} Using the notation of paragraph $2.2$,
in addition to $S_{h}^{2}$ we let
\[
S_{h,0}^{3}:=\{\phi \in C^{1} : \phi\big|_{[x_{j},x_{j+1}]} \in \mathbb{P}_{2}\,,\,\, 1\leq j\leq N\,,\,\,
\phi(0)=\phi(1)=0\}\,.
\]
In this section we shall denote by $a(\cdot,\cdot)$ the usual bilinear form
$a(\psi,\chi)=(\psi,\chi) + \tfrac{1}{3}(\psi',\chi')$ defined for $\psi$, $\chi \in S_{h,0}^{3}$ and let
$R_{h} : H_{0}^{1} \to S_{h,0}^{3}$ be the elliptic projection operator defined for $v\in H_{0}^{1}$
by $((R_{h}v)',\chi') = (v',\chi')$\, $\forall \chi \in S_{h,0}^{3}$. The system $(\ref{scb})$ is
discretized as follows: Seek $\eta_{h} : [0,T]\to S_{h}^{2}$, $u_{h} : [0,T]\to S_{h,0}^{3}$ such that
for $0\leq t\leq T$
\begin{align}
(\eta_{ht},\phi) + (u_{hx},\phi) + \tfrac{1}{2}((\eta_{h}u_{h})_{x},\phi) & = 0\,, \quad
\forall \phi \in S_{h}^{2} \label{eq51}\\
a(u_{ht},\chi) + (\eta_{hx},\chi) + \tfrac{3}{2}(u_{h}u_{hx},\chi)
+ \tfrac{1}{2}(\eta_{h}\eta_{hx},\chi) & = 0\,, \quad \forall \chi \in S_{h,0}^{3}
\label{eq52}
\end{align}
and
\begin{equation}
\eta_{h}(0) = P\eta_{0}\,, \qquad u_{h}(0) = R_{h}u_{0}\,,
\label{eq53}
\end{equation}
where $P : L^{2}\to S_{h}^{2}$ is the $L^{2}$-projection operator onto $S_{h}^{2}$. (We will treat
only the case of $(\ref{scb})$, that of $(\ref{cb})$ being similar.) The o.d.e. initial-value problem
represented by $(\ref{eq51})$-$(\ref{eq53})$ has a unique solution for any $T < \infty$. In fact,
taking $\phi=\eta_{h}$ in $(\ref{eq51})$ and $\chi=u_{h}$ in $(\ref{eq52})$ and adding the resulting
equations, we obtain for $t\geq 0$
\begin{equation}
\|\eta_{h}(t)\|^{2} + \|u_{h}(t)\|_{1}^{2} = \|\eta_{h}(0)\|^{2} + \|u_{h}(0)\|_{1}^{2}\,,
\label{eq54}
\end{equation}
and the desired conclusion follows by standard o.d.e. theory. \par
The plan of the section is as follows: We will start by stating the preliminary Lemmas $5.1$-$5.4$
with the aid of which we will prove our main error estimate, Theorem $5.1$. The proofs of the Lemmas
depend on some superaccuracy estimates of the $L^{2}$ projection error $v-Pv$ for smooth
functions $v$; these are proved in \S$5.2$. (Lemmas $5.5$-$5.10$). Recall first the following standard
estimates, \cite{ddw},\cite{w}. Suppose that $\eta \in C^{2}$ and $u \in C_{0}^{3}$. Then, if
$\rho=\eta - P\eta$, $\sigma=u-R_{h}u$, we have
\[
\|\rho\| + h\|\rho\|_{1}\leq Ch^{2}\,, \quad \|\sigma\|+h\|\sigma\|_{1}\leq Ch^{3}\,,
\quad \|\sigma\|_{\infty} + h\|\rho\|_{\infty}\leq Ch^{3}\,.
\]
\begin{lemma}
Let $\eta \in C^{4}$, $v \in C_{0}^{2}$. If $\rho=\eta - P\eta$ and $\zeta \in S_{h}^{2}$ is such that
\[
(\zeta,\phi) = ((v\rho)',\phi) \quad \forall \phi \in S_{h}^{2}\,,
\]
then $\|\zeta\| \leq Ch^{3}$.
\end{lemma}
\begin{proof} The result follows from Lemma $5.9$ and Lemma $2.1(ii)$.
\end{proof}
\begin{lemma} Let $\eta \in C^{2}$ and $v \in C_{0}^{1}$. If $\rho=\eta - P\eta$
and $\zeta \in S_{h}^{2}$ is such that
\[
(\zeta,\phi) = (v\rho,\phi) \quad \forall \phi \in S_{h}^{2}\,,
\]
then $\|\zeta\| \leq Ch^{3}$.
\end{lemma}
\begin{proof}
The result follows from Lemma $5.10$ and Lemma $2.1(ii)$.
\end{proof}
\begin{lemma} Let $\eta \in C^{2}$ and $u \in C_{0}^{3}$. If $\rho=\eta - P\eta$,
$\sigma =u-R_{h}u$, then
\begin{itemize}
\item[(i)] $(\rho',\psi) = 0 \quad \forall \psi \in S_{h,0}^{3}$
\item[(ii)] $(\sigma',\phi) = 0 \quad \forall \phi \in S_{h}^{2}$\,.
\end{itemize}
\end{lemma}
\begin{proof}
The identity $(i)$ follows from the observation that $(\rho',\psi) = (\rho,\psi') = 0$, since
$\psi \in S_{h,0}^{3}$ implies $\psi' \in S_{h}^{2}$.
In order to prove $(ii)$, we use an argument similar to that of Wahlbin, $[30$,p$.5]$. Let
$\{\phi_{i}\}_{i=1}^{N+1}$ be the usual hat function basis of $S_{h}^{2}$ associated with the
uniform mesh $x_{i} = (i-1)h$, $1\leq i\leq N+1$, i.e. with the property that
$\phi_{i}(x_{j})=\delta_{ij}$. Then, the function $\psi_{i}$ defined for $1\leq i\leq N+1$ on $[0,1]$
by $\psi_{i}(x)=\int_{0}^{x}\phi_{i}d\xi - x\int_{0}^{1}\phi_{i}d\xi$ belongs to $S_{h,0}^{3}$ and
satisfies $\psi_{i}' = \phi_{i}-\int_{0}^{1}\phi_{i}d\xi$. Since $(\sigma',1)=0$, it follows that
$(\sigma',\phi_{i})=(\sigma',\psi_{i}')=0$. Hence, $(\sigma',\phi)=0$ for any $\phi \in S_{h}^{2}$
and the proof of $(ii)$ is completed. (Since $Pu'=(R_{h}u)'$ as it may be easily established, $(ii)$
implies that $\sigma'=u'-Pu'$.)
\end{proof}
\begin{lemma} Let $\eta \in C^{1}$ and $u \in C_{0}^{3}$. If $\sigma=u-R_{h}u$ and $\zeta \in S_{h}^{2}$
is such that
\[
(\zeta,\phi) = ((\eta\sigma)',\phi) \quad \forall \phi \in S_{h}^{2}\,,
\]
then $\|\zeta\| \leq Ch^{3}$.
\end{lemma}
\begin{proof} Consider $\gamma_{i}=((\eta\sigma)',\phi_{i})$. Then, by the final remark in the proof
of Lemma $5.3$,
\[
\gamma_{i} = (\eta'\sigma,\phi_{i}) + (\eta\sigma',\phi_{i}) = (\eta'\sigma,\phi_{i})
+ (\eta(u'-Pu'),\phi_{i})\,.
\]
Now, since $\|\sigma\|_{\infty} \leq Ch^{3}$, we have that
$|(\eta'\sigma,\phi_{i})| \leq C\|\sigma\|_{\infty}\int_{0}^{1}\phi_{i}dx\leq Ch^{4}$. In addition,
by Lemma $5.10$ we obtain that $|(\eta(u'-Pu'),\phi_{i})| \leq Ch^{4}$. We conclude that
$|\gamma| = (\sum_{i=1}^{N+1}\gamma_{i}^{2})^{1/2} \leq Ch^{3.5}$, from which the conclusion of the
Lemma follows in view of Lemma $2.1(ii)$.
\end{proof}
With the aid of these lemmas we may establish the basic result of this section:
\begin{theorem} Suppose that the solution $(\eta,u)$ of $(\ref{scb})$ is such that $\eta \in C(0,T;C^{4})$,
$\eta_{t} \in C(0,T;C^{2})$, $u$, $u_{t} \in C(0,T;C_{0}^{3})$ and let $(\eta_{h},u_{h})$ be the solution
of the semidiscrete problem $(\ref{eq51})$-$(\ref{eq53})$. Then
\begin{align*}
(i)\,\,\,\,\, & \max_{0\leq t\leq T}(\|u(t)-u_{h}(t)\| + h\|\eta(t)-\eta_{h}(t)\|)\leq Ch^{3}\,, \\
(ii)\,\,\,\,\, & \max_{0\leq t\leq T}(\|u(t)-u_{h}(t)\|_{\infty}
+ h\|\eta(t)-\eta_{h}(t)\|_{\infty})\leq Ch^{3}\,.
\end{align*}
\end{theorem}
\begin{proof}
Let $\rho=\eta - P\eta$, $\theta=P\eta-\eta_{h}$, $\sigma=u-R_{h}u$, and $\xi=R_{h}u-u_{h}$. Then
$\eta - \eta_{h}=\rho+\theta$, $u-u_{h}=\sigma+\xi$ and, by Lemma $5.3$
\begin{align}
(\theta_{t},\phi) + (\xi_{x},\phi) + \tfrac{1}{2}((\eta u-\eta_{h}u_{h})_{x},\phi)& =0 \quad
\forall \phi \in S_{h}^{2}\,, \label{eq55} \\
a(\xi_{t},\psi) + (\theta_{x},\psi)+\tfrac{3}{2}(uu_{x}-u_{h}u_{hx},\psi)+
\tfrac{1}{2}(\eta\eta_{x}-\eta_{h}\eta_{hx},\psi) & = -(\sigma_{t},\psi) \quad \forall \psi\in S_{h,0}^{3}\,,
\label{eq56}
\end{align}
for $t \in [0,T]$. Since
\[
\eta u - \eta_{h}u_{h} = \eta (u-u_{h}) + u (\eta - \eta_{h}) - (\eta-\eta_{h})(u-u_{h})
= \eta(\sigma+\xi) + u (\rho + \theta) - (\rho+\theta)(\sigma+\xi)\,,
\]
it follows that
\[
\eta u - \eta_{h}u_{h} = \eta\sigma + u\rho - \theta\xi + f\,,
\]
where $f = \eta\xi + u\theta - \rho\sigma-\rho\xi-\theta\sigma$. In addition,
\begin{align*}
\eta\eta_{x} - \eta_{h}\eta_{hx} & = \eta(\eta_{x}-\eta_{hx})+\eta_{x}(\eta-\eta_{h})
-(\eta-\eta_{h})(\eta_{x}-\eta_{hx}) \\
& = \eta(\rho_{x} + \theta_{x}) + \eta_{x}(\rho+\theta) - (\rho+\theta)(\rho_{x}+\theta_{x})\\
& = (\eta\rho)_{x} + (\eta\theta)_{x} - (\rho\theta)_{x}-\rho\rho_{x} - \theta\theta_{x}\,,
\end{align*}
whence
\[
\eta\eta_{x} - \eta_{h}\eta_{hx} =  (\eta\rho)_{x} -\theta\theta_{x} + g_{x}\,,
\]
where $g =  \eta\theta - \rho\theta -\tfrac{1}{2}\rho^{2}$. Finally,
\begin{align*}
uu_{x} - u_{h}u_{hx} & = u(u_{x}-u_{hx}) + u_{x}(u-u_{h})-(u-u_{h})(u_{x}-u_{hx}) \\
& = u(\sigma_{x} + \xi_{x}) + u_{x}(\sigma + \xi)-(\sigma+\xi)(\sigma_{x}+\xi_{x})\\
& = (u\sigma)_{x} + (u\xi)_{x} - (\sigma\xi)_{x}-\sigma\sigma_{x}-\xi\xi_{x}\,,
\end{align*}
i.e.
\[
uu_{x} - u_{h}u_{hx} = \hat{f}_{x}\,,
\]
where $\hat{f} = u\sigma + u\xi - \sigma\xi -\tfrac{1}{2}\sigma^{2} - \tfrac{1}{2}\xi^{2}$.
Thus we rewrite $(\ref{eq55})$ and $(\ref{eq56})$ as
\begin{align*}
(\theta_{t},\phi) + (\xi_{x},\phi) + \tfrac{1}{2}((\eta\sigma)_{x},\phi) +
\tfrac{1}{2}((u\rho)_{x},\phi)
- \tfrac{1}{2}((\theta\xi)_{x},\phi) +
\tfrac{1}{2}(f_{x},\phi) & = 0 \quad \forall \phi \in S_{h}^{2}\,, \\
a(\xi_{t},\psi) + (\theta_{x},\psi) + \tfrac{3}{2}(\hat{f}_{x},\psi)
+\tfrac{1}{2}((\eta\rho)_{x},\psi) -
\tfrac{1}{2}((\theta\theta)_{x},\psi) +
\tfrac{1}{2}(g_{x},\psi) & = -(\sigma_{t},\psi)
\quad \forall \psi \in S_{h,0}^{3}\,.
\end{align*}
Putting $\phi = \theta$, $\chi = \xi$ in the above and adding the resulting equations, we have for
$0\leq t\leq T$
\begin{equation}
\begin{aligned}
\tfrac{1}{2}\tfrac{d}{dt}(\|\theta\|^{2} + \|\xi\|_{1}^{2}) & + \tfrac{1}{2}((\eta\sigma)_{x},\theta)
+ \tfrac{1}{2}((u\rho)_{x},\theta) + \tfrac{1}{2}(f_{x},\theta) \\
& + \tfrac{1}{2}((\eta\rho)_{x},\xi)
+ \tfrac{3}{2}(\hat{f}_{x},\xi) + \tfrac{1}{2}(g_{x},\xi)
= -(\sigma_{t},\xi)\,.
\end{aligned}
\label{eq57}
\end{equation}
We estimate the various terms in $(\ref{eq57})$ as follows. By Lemma $5.4$ we have
\[
|((\eta\sigma)_{x},\theta)| \leq C h^3 \|\theta\|\,,
\]
and by Lemma $5.1$
\[
|((u\rho)_{x},\theta)| \leq C h^{3}\|\theta\|\,.
\]
Using integration by parts and the standard estimates for $\rho$ and $\sigma$ we also have
\[
|(f_{x},\theta)| \leq  C\|\xi\|_{1}\|\theta\| + Ch^{4}\|\theta\| + C \|\theta\|^{2}\,.
\]
In addition, from Lemma $5.2$, since $\xi_{x} \in S_{h}^{2}$
\[
|((\eta\rho)_{x},\xi)|  = |(\eta\rho,\xi_{x})| \leq C h^{3} \|\xi\|_{1}\,.
\]
Finally, it follows from integration by parts and standard estimates for $\sigma$ that
\[
|(\hat{f}_{x},\xi)| = |(\hat{f},\xi_{x})| \leq C h^{3}\|\xi\|_{1} + C \|\xi\|^{2}\,,
\]
and by standard estimates for $\rho$ that
\[
|(g_{x},\xi)| = |(g,\xi_{x})| \leq C \|\theta\|\|\xi\|_{1} + Ch^{4}\|\xi\|_{1}\,.
\]
Hence, from $(\ref{eq57})$ and these estimates we deduce for $0\leq t\leq T$
\[
\dfrac{1}{2}\dfrac{d}{dt}(\|\theta\|^{2} + \|\xi\|_{1}^{2}) \leq
C(h^{6} + \|\theta\|^{2} + \|\xi\|_{1}^{2})\,,
\]
and, consequently, by Gronwall's Lemma and $(\ref{eq53})$ that
\[
\|\theta\| + \|\xi\|_{1} \leq C h^3\,, \quad 0\leq t\leq T\,,
\]
and the conclusions of the Theorem follow in view of the approximation properties of $S_{h}^{2}$
and $S_{h,0}^{3}$, Sobolev's theorem and the $L^{\infty}$-$L^{2}$ inverse inequality in $S_{h}^{2}$.
\end{proof}
\begin{remark} In view of Theorems $1.3.1$ and $1.3.2$ of \cite{w} the conclusions of Theorem $5.1$
remain valid if in place of $R_{h}u_{0}$ we use the elliptic projection of $u_{0}$ that is induced
by the bilinear form $a(\cdot,\cdot)$ as was done in previous sections.
\end{remark}
\subsection{Superaccuracy estimates for the error of the $L^{2}$ projection}
The proofs of Lemmas $5.1$, $5.2$, and $5.4$ depend on some superaccuracy properties
of the error of the $L^{2}$ projection onto $S_{h}^{2}$ of a smooth function, in the case of
uniform mesh. \par
Suppose that $\eta \in C^{3}$, $u \in C_{0}^{2}$, and let $P$ be the
$L^{2}$ projection onto $S_{h}^{2}$ and $\rho = \eta - P\eta$. To estimate
$(u\rho,\phi_{i}')$, where $\{\phi_{i}\}$ is the standard basis of $S_{h}^{2}$,
it suffices to estimate the integrals
\[
\int_{I_{i}}u\rho dx\,,
\]
where $I_{i}=(x_{i},x_{i+1})$, $i=1,2,\dots,N$. Putting $x_{i+1/2} = (x_{i}+x_{i+1})/2$, we will have
\[
\int_{I_{i}} u\rho dx = u(x_{i+1/2})\int_{I_{i}} \rho dx
+ u'(x_{i+1/2})\int_{I_{i}}(x-x_{i+1/2})\rho(x) dx +
\tfrac{1}{2} \int_{I_{i}} (x-x_{i+1/2})^{2}\rho(x) u''(y_{i}) dx\,,
\]
for some $y_{i}=y_{i}(x) \in I_{i}$. Since $\|\rho\|_{\infty}=O(h^{2})$, the third integral in the
right-hand side is of $O(h^{5})$. Now, if $x \in I_{i}$
\[
\rho(x) = \rho(x_{i+1/2}) + (x - x_{i+1/2})\rho'(x_{i+1/2}) +
\tfrac{1}{2} (x - x_{i+1/2})^{2} \eta''(x_{i+1/2}) +
\tfrac{1}{6} (x - x_{i+1/2})^{3}\eta'''(\tau_{i})\,,
\]
for some $\tau_{i}=\tau_{i}(x) \in I_{i}$. Hence
\[
\int_{I_{i}}(x-x_{i+1/2})\rho(x)dx = \rho'(x_{i+1/2})\int_{I_{i}}(x-x_{i+1/2})^{2}dx + O(h^{5})
= \tfrac{h^{3}}{12}\rho'(x_{i+1/2}) + O(h^{5})\,.
\]
Thus
\[
\int_{I_{i}} u\rho dx = u(x_{i+1/2})\int_{I_{i}} \rho dx
+ \tfrac{h^{3}}{12}\rho'(x_{i+1/2})u'(x_{i+1/2}) + O(h^{5})\,,
\]
and, consequently, we must estimate the integral of $\rho$ on $I_{i}$ and the derivative of
$\rho$ at the midpoint of $I_{i}$. Both of these quantities turn out to be superaccurate as
we prove by elementary techniques in the following two lemmas:
\begin{lemma} Let $\eta \in C^{3}$. Then
\[
\max_{1\leq i\leq N} |\rho'(x_{i+1/2})| \leq Ch^{2}\,.
\]
\end{lemma}
\begin{proof} For $x \in I_{i}$ we have
\[
\rho'(x) = \eta'(x) - (P\eta)'(x) = \eta'(x) - \dfrac{c_{i+1} - c_{i}}{h}\,,
\]
i.e.
\begin{equation}
\rho'(x_{i+1/2}) = \eta'(x_{i+1/2}) - \dfrac{c_{i+1} - c_{i}}{h}\,, \quad
1 \leq i\leq N\,,
\label{eq58}
\end{equation}
where $c_{i}$ is the $i$th component of the solution $c$ of the linear system $Gc=b$,
where $G_{ij} = (\phi_{j},\phi_{i})$, $1\leq i,j \leq N+1$ and $b_{i} = (\eta,\phi_{i})$,
$i=1,2,\dots,N+1$. The equations of the system may be written explicitly as
\begin{equation}
\begin{aligned}
2c_{1} + c_{2} & = 6 b_{1}/h\,, \\
c_{1} + 4 c_{2} + c_{3} & = 6 b_{2}/h\,, \\
c_{i-1} + 4c_{i} + c_{i+1} & = 6 b_{i}/h \,, \quad i=3,4,\dots,N-1\,, \\
c_{N-1} + 4c_{N} + c_{N+1} & = 6 b_{N}/h\,, \\
c_{N} + 2c_{N+1} & = 6 b_{N+1}/h\,.
\end{aligned}
\label{eq59}
\end{equation}
Hence
\begin{align*}
3(c_{2} - c_{1}) + (c_{3} - c_{2}) & = 6(b_{2} - 2b_{1})/h\,, \\
(c_{i} - c_{i-1}) + 4 (c_{i+1} - c_{i}) & + (c_{i+2} - c_{i+1}) = 6(b_{i+1}-b_{i})/h\,,
\quad 2\leq i\leq N-1
\end{align*}
and
\[
(c_{N} - c_{N-1}) + 3(c_{N+1} - c_{N}) = 6(2b_{N+1} - b_{N})/h\,.
\]
Therefore we obtain the following equations for the differences $(c_{i+1}-c_{i})/h$:
\begin{align*}
3(c_{2} - c_{1})/h + (c_{3} - c_{2})/h & = 6(b_{2} - 2b_{1})/h^{2}\,,  \\
(c_{i} - c_{i-1})/h + 4 (c_{i+1} - c_{i})/h +
(c_{i+2} - c_{i+1})/h & = 6(b_{i+1}-b_{i})/h^{2}\,,\quad i=2,3,\dots,N-1\,,\\
(c_{N} - c_{N-1})/h + 3(c_{N+1} - c_{N})/h & = 6(2b_{N+1} - b_{N})/h^{2}\,.
\end{align*}
Hence, if we put $\ve'_{i}:=\rho'(x_{i+1/2})$, $i=1,2,\dots,N$, then $(\ref{eq58})$ implies that
$\ve' = (\ve_{i}')$ is the solution of the linear system
\begin{equation}
A\varepsilon'=r'\,,
\label{eq510}
\end{equation}
where $A$ is the $N\times N$ tridiagonal matrix with $A_{11}=A_{NN}=3$,
$A_{ii}=4$, $2\leq i \leq N-1$, and $A_{ij}=1$ if $|i-j|=1$, and $r'=(r'_{1},r'_{2},\dots,r'_{N})^{T}$,
where
\begin{equation}
\begin{aligned}
r'_{1} & = 3\eta'(x_{3/2}) + \eta'(x_{5/2}) - 6(b_{2} - 2b_{1})/h^{2}\,, \\
r'_{i} & = \eta'(x_{i-1/2}) + 4\eta'(x_{i+1/2}) + \eta'(x_{i+3/2})-6(b_{i+1}-b_{i})/h^{2}\,,
\quad i=2,3,\dots,N-1\,,\\
r'_{N} & = \eta'(x_{N-1/2}) + 3\eta'(x_{N+1/2}) - 6(2b_{N+1} - b_{N})/h^{2}\,.
\end{aligned}
\label{eq511}
\end{equation}
We will now show that there exists a constant $C$ depending only on $\|\eta'''\|_{\infty}$, such that
\begin{equation}
\max_{1\leq i\leq N}|r'_{i}| \leq Ch^{2}\,.
\label{eq512}
\end{equation}
To estimate $r'_{1}$ note that
\begin{align*}
b_{2} - 2b_{1} = \int_{I_{1}\cup I_{2}}\eta\phi_{2}dx - 2\int_{I_{1}}\eta\phi_{1}dx & =
\tfrac{3}{h}\int_{I_{1}}(x-h)\eta dx - \tfrac{1}{h}\int_{I_{2}}(x-h)\eta dx
+ \int_{I_{1}\cup I_{2}}\eta dx \\
& =: \tfrac{3}{h}J_{1} - \tfrac{1}{h}J_{2} + J_{3}\,.
\end{align*}
Using Taylor's theorem we have
\begin{align*}
J_{1} & = \int_{I_{1}}(x-h)\eta dx= -\tfrac{h^{2}}{2}\eta(h) + \tfrac{h^3}{3}\eta'(h) -
\tfrac{h^4}{8}\eta''(h) + O(h^5)\,, \\
J_{2} & = \int_{I_{2}}(x-h)\eta dx = \tfrac{h^2}{2}\eta(h) + \tfrac{h^3}{3}\eta'(h) +
\tfrac{h^4}{8}\eta''(h) + O(h^5)\,, \\
J_{3} & = \int_{I_{1}\cup I_{2}}\eta dx = 2h\eta(h) + \tfrac{h^3}{3}\eta''(h) + O(h^{4})\,,
\end{align*}
and thus
\[
b_{2} - 2b_{1} = \tfrac{2h^2}{3}\eta'(h) - \tfrac{h^3}{6}\eta''(h) + O(h^4)\,.
\]
In addition,
\[
3\eta'(x_{3/2}) + \eta'(x_{5/2}) = 4\eta'(h) - h\eta''(h) + O(h^{2})\,.
\]
Therefore, by $(\ref{eq511})$, we see that $r'_{1}=O(h^{2})$.
For $r'_{N}$ we have
\begin{align*}
2b_{N+1} - b_{N} & = 2\int_{I_{N}}\eta\phi_{N+1}dx - \int_{I_{N-1}\cup I_{N}}\eta\phi_{N}dx \\
& = \tfrac{3}{h}\int_{I_{N}}(x-x_{N})\eta dx -
\tfrac{1}{h}\int_{I_{N-1}}(x-x_{N})\eta dx -\int_{I_{N-1}\cup I_{N}}\eta dx \\
& = : \tfrac{3}{h}J_{1} - \tfrac{1}{h}J_{2} - J_{3}\,.
\end{align*}
By Taylor's theorem
\begin{align*}
J_{1} & =\tfrac{h^2}{2}\eta(x_{N}) + \tfrac{h^3}{3}\eta'(x_{N}) + \tfrac{h^4}{8}\eta''(x_{N}) + O(h^5)\,, \\
J_{2} & = - \tfrac{h^2}{2}\eta(x_{N}) + \tfrac{h^3}{3}\eta'(x_{N}) - \tfrac{h^4}{8}\eta''(x_{N}) + O(h^5)\,, \\
J_{3} & = 2h\eta(x_{N}) + \tfrac{h^3}{3}\eta''(x_{N}) + O(h^4)\,,
\end{align*}
and thus
\[
2b_{N+1} - b_{N} = \tfrac{2h^2}{3}\eta'(x_{N}) + \tfrac{h^3}{6}\eta''(x_{N}) + O(h^4)\,.
\]
Hence,by $(\ref{eq511})$
\begin{align*}
r'_{N} & = \eta'(x_{N-1/2}) + 3\eta'(x_{N+1/2}) - \tfrac{6}{h^{2}}(2b_{N+1} - b_{N})  \\
& = 4\eta'(x_{N}) + h\eta''(x_{N}) - 4\eta'(x_{N}) - h\eta''(x_{N}) + O(h^2) = O(h^2)\,.
\end{align*}
For $r'_{i}$, $2\leq i\leq N-1$, we have
\begin{equation}
\begin{aligned}
b_{i+1} - b_{i}
& = \int_{I_{i}\cup I_{i+1}}\eta\phi_{i+1}dx - \int_{I_{i-1}\cup I_{i}}\eta\phi_{i}dx \\
& = \tfrac{2}{h}\int_{I_{i}}(x - x_{i+1/2})\eta dx -
\tfrac{1}{h}\int_{I_{i+1}}(x - x_{i+3/2})\eta dx -
\tfrac{1}{h}\int_{I_{i-1}}(x - x_{i-1/2})\eta dx \\
& \,\,\,\,\, + \tfrac{1}{2}\int_{I_{i+1}}\eta dx - \tfrac{1}{2}\int_{I_{i-1}}\eta dx
=: \tfrac{2}{h}J_{1} -\tfrac{1}{h}J_{2} - \tfrac{1}{h}J_{3} + \tfrac{1}{2}J_{4} - \tfrac{1}{2}J_{5}\,.
\end{aligned}
\label{eq513}
\end{equation}
By Taylor's theorem we obtain
\begin{equation}
J_{1} = \eta'(x_{i+1/2})\int_{I_{i}}(x-x_{i+1/2})^2 dx + O(h^5)
= \tfrac{h^3}{12}\eta'(x_{i+1/2})+O(h^5)\,.
\label{eq514}
\end{equation}
Similarly
\begin{align}
J_{2} & = \tfrac{h^3}{12}\eta'(x_{i+3/2}) + O(h^{5})\,, \label{eq515} \\
J_{3} & = \tfrac{h^3}{12}\eta'(x_{i-1/2}) + O(h^5)\,, \label{eq516}
\end{align}
\begin{align*}
J_{4} & = h \eta(x_{i+3/2}) + \tfrac{h^3}{24}\eta''(x_{i+3/2}) + O(h^4)\,,  \\
J_{5} & = h \eta(x_{i-1/2}) + \tfrac{h^3}{24}\eta''(x_{i-1/2}) + O(h^4)\,.
\end{align*}
Since now
\begin{equation}
J_{4} - J_{5} = 2h^2\eta'(x_{i+1/2}) + O(h^4)\,.
\label{eq517}
\end{equation}
Substituting $(\ref{eq514})$-$(\ref{eq517})$ into $(\ref{eq513})$ we obtain
\[
b_{i+1} - b_{i} = \tfrac{7h^2}{6}\eta'(x_{i+1/2}) - \tfrac{h^2}{12}\eta'(x_{i+3/2})
-\tfrac{h^2}{12}\eta'(x_{i-1/2}) + O(h^4)\,,
\]
and finally, by $(\ref{eq511})$
\begin{align*}
r'_{i} & = \eta'(x_{i-1/2}) + 4 \eta'(x_{i+1/2}) + \eta'(x_{i+3/2})-\tfrac{6}{h^2}(b_{i+1} - b_{i}) \\
& = \tfrac{3}{2}(\eta'(x_{i+3/2}) - 2\eta'(x_{i+1/2}) + \eta'(x_{i-1/2})) + O(h^2) = O(h^{2})\,.
\end{align*}
We conclude therefore that $(\ref{eq512})$ holds with a constant $C=C(\|\eta'''\|_{\infty})$.
Writing the linear system $(\ref{eq59})$ as $\tfrac{1}{4}A\ve' = \tfrac{1}{4}r'$ with
$\tfrac{1}{4}A = I - B$, where $B$ is a $N\times N$ matrix with
$\|B\|_{\infty}=1/2$. Therefore, $\|(I-B)^{-1}\|_{\infty} \leq 2$ and consequently
\[
\max_{1\leq i \leq N} |\ve'_{i}| \leq \tfrac{1}{2}\max_{1\leq i \leq N}|r'_{i}| \leq C h^{2}\,.
\]
\end{proof}
\begin{lemma} Let $\eta \in C^3$. Then for some constant $C=C(\|\eta'''\|_{\infty})$ we have
\[
\max_{1\leq i \leq N} \bigl|\int_{I_{i}}\rho dx\bigr| \leq Ch^4\,.
\]
\end{lemma}
\begin{proof} We have
\[
\int_{I_{i}}\rho dx = \int_{I_{i}}\eta dx - \tfrac{h}{2}(c_{i+1} + c_{i})\,, \quad i=1,2,\dots,N\,,
\]
where $c_{i}$ is the $i$-th component of the solution $c$ of the linear system $Gc=b$ with $G$ and $b$
as in the previous Lemma. From the equation $(\ref{eq59})$ we see that
\begin{align*}
5 \cdot \tfrac{h}{2}(c_{1} + c_{2}) + \tfrac{h}{2}(c_{2} + c_{3}) & = 3(2b_{1} + b_{2})\,,\\
\tfrac{h}{2}(c_{i-1} + c_{i}) + 4\cdot \tfrac{h}{2}(c_{i} + c_{i+1}) +
\tfrac{h}{2}(c_{i+1} + c_{i+2}) & = 3(b_{i+1}+b_{i})\,, \quad 2\leq i\leq N-1\,,\\
\tfrac{h}{2}(c_{N-1} + c_{N}) + 5\cdot\tfrac{h}{2}(c_{N} + c_{N+1}) & = 3(b_{N} + 2b_{N+1})\,.
\end{align*}
Hence, if
\[
\ve_{i} := \int_{I_{i}}\rho dx\,, \quad i=1,2,\dots,N\,,
\]
then $\ve = (\ve_{1},\ve_{2},\dots,\ve_{N})^{T}$ is the solution of the system
\begin{equation}
\Gamma\varepsilon = r\,,
\label{eq518}
\end{equation}
where $\Gamma$ is the tridiagonal $N\times N$ matrix  with $\Gamma_{11}=\Gamma_{NN}=5$,
$\Gamma_{ii}=4$, $i=2,3,\dots,N-1$, and $\Gamma_{ij} = 1$, when $|i-j|=1$, and where
$r = (r_{1},r_{2},\dots,r_{N})^{T}$ with
\begin{equation}
\begin{aligned}
r_{1} & = 5\int_{I_{1}}\eta dx + \int_{I_{2}}\eta dx - 3(2b_{1} + b_{2})\,, \\
r_{i} & = \int_{I_{i-1}}\eta dx + 4\int_{I_{i}}\eta dx + \int_{I_{i+1}}\eta dx - 3(b_{i+1} + b_{i})\,,
\quad i=2,3,\dots,N-1\,, \\
r_{N} & = \int_{I_{N-1}}\eta dx + 5\int_{I_{N}}\eta dx - 3(b_{N} + 2b_{N+1})\,.
\end{aligned}
\label{eq519}
\end{equation}
We will show that $r_{i}=O(h^4)$,  for $i=1,2,\dots,N$\,. \\
For $r_{1}$, note that
\[
2b_{1} + b_{2} = 2\int_{I_{1}}\eta\phi_{1}dx + \int_{I_{1}\cup I_{2}}\eta\phi_{2}dx
= \int_{I_{1}\cup I_{2}}\eta dx - \tfrac{1}{h}\int_{I_{1}\cup I_{2}}(x-h)\eta dx\,.
\]
Hence, by Taylor's theorem
\begin{equation}
2b_{1} + b_{2} = 2h\eta(h) - \tfrac{2h^2}{3}\eta'(h) + \tfrac{h^3}{3}\eta''(h) + O(h^4)\,.
\label{eq520}
\end{equation}
Also
\begin{align*}
\int_{I_{1}}\eta dx & = h\eta(h) - \tfrac{h^2}{2}\eta'(h) + \tfrac{h^3}{6}\eta''(h) + O(h^4)\,, \\
\int_{I_{2}}\eta dx & = h\eta(h) + \tfrac{h^2}{2}\eta'(h) + \tfrac{h^3}{6}\eta''(h) + O(h^4)\,.
\end{align*}
From these relations and $(\ref{eq519})$ and $(\ref{eq520})$ we obtain
\begin{align*}
r_{1} & = 5\bigl( h\eta(h) - \tfrac{h^2}{2}\eta'(h) + \tfrac{h^3}{6}\eta''(h)\bigr) +
\bigl( h\eta(h) + \tfrac{h^2}{2}\eta'(h) + \tfrac{h^3}{6}\eta''(h)\bigr) \\
&\,\,\,\,\, - 3\bigl( 2h\eta(h) - \tfrac{2h^2}{3}\eta'(h) + \tfrac{h^3}{3}\eta''(h)\bigr) + O(h^4)
=O(h^{4})\,.
\end{align*}
For $r_{N}$ we observe
\[
b_{N} + 2b_{N+1} = \int_{I_{N-1}\cup I_{N}}\eta\phi_{N}dx + 2\int_{I_{N}}\eta\phi_{N+1}dx
= \tfrac{1}{h} \int_{I_{N-1}\cup I_{N}}(x-x_{N})\eta dx + \int_{I_{N-1}\cup I_{N}}\eta dx\,.
\]
Hence, by Taylor's theorem
\begin{equation}
b_{N} + 2b_{N+1} =
2h\eta(x_{N}) + \tfrac{2h^2}{3}\eta'(x_{N}) + \tfrac{h^3}{3}\eta''(x_{N}) + O(h^4)\,.
\label{eq521}
\end{equation}
Also
\begin{align*}
\int_{I_{N-1}}\eta dx & = h\eta(x_{N}) - \tfrac{h^2}{2}\eta'(x_{N}) + \tfrac{h^3}{6}\eta''(x_{N}) + O(h^4)\,, \\
\int_{I_{N}}\eta dx & = h\eta(x_{N}) + \tfrac{h^2}{2}\eta'(x_{N}) + \tfrac{h^3}{6}\eta''(x_{N}) + O(h^4)\,.
\end{align*}
From these relations and $(\ref{eq519})$ and $(\ref{eq521})$ we see that
\begin{align*}
r_{N} & = \bigl( h\eta(x_{N}) - \tfrac{h^2}{2}\eta'(x_{N}) + \tfrac{h^3}{6}\eta''(x_{N})\bigr) +
5\bigl( h\eta(x_{N}) + \tfrac{h^2}{2}\eta'(x_{N}) + \tfrac{h^3}{6}\eta''(x_{N})\bigr) \\
&\,\,\,\,\, - 3\bigl( 2h\eta(x_{N}) + \tfrac{2h^2}{3}\eta'(x_{N}) + \tfrac{h^3}{3}\eta''(x_{N})\bigr) + O(h^4)
=O(h^{4})\,.
\end{align*}
For $r_{i}$, $i=2,3,\dots,N-1$, we have
\begin{align*}
b_{i} + b_{i+1} & = \int_{I_{i-1}\cup I_{i}}\eta\phi_{i}dx +  \int_{I_{i}\cup I_{i+1}}\eta\phi_{i}dx \\
& = \tfrac{1}{h}\int_{I_{i-1}}(x-x_{i-1})\eta dx
+ \int_{I_{i}}\eta dx+ \tfrac{1}{h}\int_{I_{i+1}}(x_{i+2}-x)\eta dx\,.
\end{align*}
Thus, by $(\ref{eq519})$ we have, for $2\leq i\leq N-1$
\begin{equation}
r_{i} = \int_{I_{i-1}\cup I_{i}\cup I_{i+1}}\eta dx - \tfrac{3}{h}\int_{I_{i-1}}(x-x_{i-1})\eta dx
- \tfrac{3}{h}\int_{I_{i+1}}(x_{i+2}-x)\eta dx =:J_{1} - \tfrac{3}{h}J_{2} - \tfrac{3}{h}J_{3}\,.
\label{eq522}
\end{equation}
By Taylor's theorem
\begin{align*}
J_{2} & = \tfrac{h^2}{2}\eta(x_{i-1}) + \tfrac{h^3}{3}\eta'(x_{i-1})+\tfrac{h^4}{8}\eta''(x_{i-1}) + O(h^5)\,,\\
J_{3} & = \tfrac{h^2}{2}\eta(x_{i+2}) - \tfrac{h^3}{3}\eta'(x_{i+2})+\tfrac{h^4}{8}\eta''(x_{i+2}) + O(h^5)\,,
\end{align*}
i.e.
\begin{align*}
J_{2} + J_{3} & = \tfrac{h^2}{2}(\eta(x_{i-1}) + \eta(x_{i+2})) + \tfrac{h^3}{3}(\eta'(x_{i-1})-\eta'(x_{i+2}))
+ \tfrac{h^4}{8}(\eta''(x_{i-1}) + \eta''(x_{i+2})) + O(h^5) \\
& = h^2\eta(x_{i+1/2}) + \tfrac{3h^4}{8}\eta''(x_{i+1/2}) + O(h^5)\,.
\end{align*}
In addition,
\[
J_{1} = 3h\eta(x_{i+1/2}) + \tfrac{1}{3}\tfrac{27h^3}{8}\eta''(x_{i+1/2}) + O(h^4)\,.
\]
Substituting in $(\ref{eq522})$ and in $(\ref{eq519})$ we see that for $2\leq i\leq N-1$
\[
r_{i} = 3h\eta(x_{i+1/2}) + \tfrac{9h^3}{8}\eta''(x_{i+1/2}) - 3h\eta(x_{i+1/2})
- \tfrac{9h^3}{8}\eta''(x_{i+1/2}) + O(h^4) = O(h^4)\,.
\]
Hence, there exists a constant $C$ depending only on $\|\eta'''\|_{\infty}$ such that
\begin{equation}
\max_{1\leq i\leq N}|r_{i}| \leq C h^{4}\,.
\label{eq523}
\end{equation}
From $(\ref{eq518})$ we obtain that $\tfrac{1}{4}\Gamma\ve = \tfrac{1}{4}r$, where
$\tfrac{1}{4}\Gamma = I - E$, and $E$ is a $N\times N$ matrix with $\|E\|_{\infty}=1/2$. Hence
$\|(I-E)^{-1}\|_{\infty} \leq 2$ and thus
\[
\max_{1\leq i \leq N} |\ve_{i}| \leq \tfrac{1}{2}\max_{1\leq i \leq N}|r_{i}| \leq C h^{4}\,.
\]
\end{proof}
If we now assume that $\eta$ is smoother, then, using a result by Demko, \cite{de}, (see also
de Boor,\cite{db}), we may show that $\int_{I_{i}}\rho dx=O(h^{5})$ provided that the endpoints
$x_{i}$, $x_{i+1}$ of the interval $I_{i}$ are at a distance of $O(h\ln 1/h)$ from the boundary
of $[0,1]$.
\begin{lemma} Let $\eta \in C^4$. Then, there exists a constant $\widetilde{C}$ such that if
\[
dist(x_{i},0)\geq \widetilde{C}h\ln 1/h\, \quad \text{and}
\quad dist(x_{i+1},1)\geq \widetilde{C}h\ln 1/h\,,
\]
then
\[
\int_{I_{i}}\rho dx = O(h^5)\,.
\]
\end{lemma}
\begin{proof}
We will first show that due to our hypothesis of increased smoothness of $\eta$, we now have
\[
r_{i} = O(h^5)\,, \quad i=2,3,\dots,N-1\,,
\]
where the $r_{i}$ were defined in $(\ref{eq519})$ and written in the form $(\ref{eq522})$. Since
$\eta \in C^4$, using Simpson's rule we approximate the integrals in $(\ref{eq522})$ by
\begin{align*}
\int_{I_{i-1}\cup I_{i}\cup I_{i+1}}\eta dx & =
\tfrac{3h}{6}\bigl(\eta(x_{i-1}) + 4\eta(x_{i+1/2})+\eta(x_{i+2})\bigr)
+ O(h^5)\,, \\
\int_{I_{i-1}}\tfrac{x-x_{i-1}}{h}\eta dx & =
\tfrac{h}{6}\bigl( 2\eta(x_{i-1/2}) + \eta(x_{i})\bigr) + O(h^5)\,, \\
\int_{I_{i+1}}\tfrac{x_{i+2}-x}{h}\eta dx & =
\tfrac{h}{6}\bigl( \eta(x_{i+1}) + 2\eta(x_{i+3/2})\bigr) + O(h^5)\,,
\end{align*}
yielding
\[
r_{i} = \tfrac{h}{2}\bigl( \eta(x_{i-1}) -2\eta(x_{i-1/2}) - \eta(x_{i})
+ 4 \eta(x_{i+1/2}) + \eta(x_{i+2})-2\eta(x_{i+3/2}) - \eta(x_{i+1})
\bigr) + O(h^5)\,.
\]
Hence, using Taylor's theorem
\begin{align*}
r_{i} & = \tfrac{h}{2}\bigl( \tfrac{h^2}{4}\eta''(x_{i-1/2}) - 2\tfrac{h^2}{4}\eta''(x_{i+1/2})
+ \tfrac{h^2}{4}\eta''(x_{i+3/2})\bigr) + O(h^5)\\
& = \tfrac{h^3}{8}\bigl( \eta''(x_{i-1/2}) - 2\eta''(x_{i+1/2}) + \eta''(x_{i+3/2})\bigr) + O(h^5)
= O(h^5)\,.
\end{align*}
We conclude that for constants $C_{1}$, $C_{2}$ independent of $h$ we have, in view
of $(\ref{eq523})$, that
\begin{equation}
|r_{1}| \leq C_{1} h^4\,, \quad \max_{2\leq i \leq N-1}|r_{i}| \leq C_{2}h^5\,, \quad
|r_{N}| \leq C_{1} h^4\,.
\label{eq524}
\end{equation}
We also recall from Lemma $5.6$, cf. $(\ref{eq518})$, that the $\ve_{i}:=\int_{I_{i}}\rho dx$,
$1\leq i\leq N$, form the solution of the system
\[
\tfrac{1}{4}\Gamma\ve=(I-E)\ve=\tfrac{1}{4}r\,,
\]
where $\Gamma$ was defined after $(\ref{eq518})$. Hence
\begin{equation}
\ve_{i} = \sum_{j=1}^{N}(\Gamma^{-1})_{ij}r_{j}= (\Gamma^{-1})_{i1}r_{1} +
\sum_{j=2}^{N-1}(\Gamma^{-1})_{ij}r_{j} + (\Gamma^{-1})_{iN}r_{N}\,.
\label{eq525}
\end{equation}
Now, using Proposition $2.1$ of \cite{de}, we see that there is a constant $C_{3}$ such that
\begin{equation}
|(\Gamma^{-1})_{i1}| \leq C_{3} 2^{-|i-1|}\,, \quad \text{and} \quad
|(\Gamma^{-1})_{iN}| \leq C_{3} 2^{-|i-N|}\,.
\label{eq526}
\end{equation}
Therefore, if $dist(x_{i},0) \geq \tfrac{1}{\ln 2} h\ln 1/h$ and
$dist(x_{i+1},1) \geq \tfrac{1}{\ln2} h \ln 1/h$, there follows that
\[
i-1  \geq \tfrac{1}{\ln 2} \ln 1/h \quad \text{and} \quad
N-i \geq \tfrac{1}{\ln 2} \ln 1/h\,,
\]
i.e.
\[
2^{-|i-1|} \leq h \quad \text{and} \quad  2^{-|i-N|} \leq h\,,
\]
and by $(\ref{eq526})$
\[
|(\Gamma^{-1})_{i1}| \leq C_{3} h \quad \text{and} \quad
|(\Gamma^{-1})_{iN}| \leq C_{3} h\,.
\]
Hence, for the indices $i$ for which
\[
1+\tfrac{1}{\ln 2}\ln 1/h \leq i \leq N - \tfrac{1}{\ln 2} \ln 1/h\,,
\]
we will have from $(\ref{eq525})$ and $(\ref{eq524})$ that
\begin{align*}
|\varepsilon_{i}| & \leq C_{3}h|r_{1}| +
\sum_{j=2}^{N-1}|(\Gamma^{-1})_{ij}||r_{i}| + C_{3} h |r_{N}| \\
& \leq C_{3}C_{1} h^{5} + \max_{1 \leq i \leq N}\sum_{j=1}^{N}|(\Gamma^{-1})_{ij}| C_{2}h^5
+ C_{3}C_{1}h^5 \leq Ch^{5}\,.
\end{align*}
\end{proof}
We now prove, as consequences of these results, several superaccurate estimates on weighted integrals
of the error of the $L^{2}$ projection.
\begin{lemma} Let $\eta \in C^4$, $u \in C_{0}^{2}$, and $\rho=\eta - P\eta$. Then,
\[
\max_{1\leq i \leq N}\bigl| \int_{I_{i}}u\rho dx\bigr| \leq Ch^5\ln 1/h\,.
\]
\end{lemma}
\begin{proof} If $1\leq i\leq N$ we have for $t_{i}=t_{i}(x) \in I_{i}$
\[
\int_{I_{i}}u\rho dx = u(x_{i+1/2}) \int_{I_{i}}\rho dx
+ u'(x_{i+1/2})\int_{I_{i}}(x-x_{i+1/2})\rho dx
+ \tfrac{1}{2}\int_{I_{i}}(x-x_{i+1/2})^2\rho u''(t_{i})dx\,.
\]
The third integral of the right-hand side is of $O(h^{5})$. For the second integral, by
Taylor's theorem we have
\[
\int_{I_{i}}(x-x_{i+1/2})\rho dx= \tfrac{h^3}{12}\rho'(x_{i+1/2}) + O(h^5)\,,
\]
and using Lemma $5.5$,
\[
\int_{I_{i}}(x-x_{i+1/2})\rho dx = O(h^5)\,.
\]
We conclude that
\begin{equation}
\int_{I_{i}}u\rho dx = u(x_{i+1/2}) \int_{I_{i}}\rho dx + O(h^5)\,.
\label{eq527}
\end{equation}
If now $i$ is such that
\[
dist(x_{i},0) \geq \widetilde{C}h\ln 1/h\,, \quad \text{and} \quad
dist(x_{i+1},1) \geq \widetilde{C}h\ln 1/h\,,
\]
where $\widetilde{C}$ is the constant in the statement of Lemma $5.7$, then
\begin{equation}
\int_{I_{i}}u\rho dx = O(h^5)\,.
\label{eq528}
\end{equation}
If $dist(x_{i},0)\leq \widetilde{C}h\ln 1/h$, we have
\[
u(x_{i+1/2}) = u(x_{i}) + \tfrac{h}{2}u'(\tau_{i})\,, \quad u(x_{i}) = x_{i}u'(y_{i})\,,
\]
for some $\tau_{i} \in (x_{i},x_{i+1/2})$ and $y_{i} \in (0,x_{i})$. Therefore
\[
|u(x_{i+1/2})| \leq Ch\ln 1/h\,,
\]
and from Lemma $5.6$ and $(\ref{eq527})$ we have for such $i$ that
\begin{equation}
\bigl| \int_{I_{i}}u\rho dx \bigr| \leq Ch^5 \ln 1/h\,.
\label{eq529}
\end{equation}
If finally $dist(x_{i+1},1)\leq \widetilde{C}h\ln 1/h$, we have
\[
u(x_{i+1/2}) = u(x_{i+1}) -\tfrac{h}{2}u'(\widetilde{\tau}_{i})\,, \quad
u(x_{i+1}) = (1-x_{i+1})u'(\widetilde{y}_{i})\,,
\]
for some $\widetilde{\tau}_{i} \in (x_{i+1/2},x_{i+1})$ and $\widetilde{y}_{i} \in (x_{i+1},1)$.
Hence
\[
|u(x_{i+1/2})| \leq Ch\ln 1/h\,,
\]
and it follows again from Lemma $5.6$ and $(\ref{eq527})$ that in this case too
\begin{equation}
\bigl| \int_{I_{i}}u\rho dx\bigr| \leq Ch^5 \ln 1/h\,.
\label{eq530}
\end{equation}
The conclusion of the Lemma follows from $(\ref{eq528})$-$(\ref{eq530})$.
\end{proof}
\begin{lemma} Let $\eta \in C^4$, $u \in C_{0}^{2}$, $\rho=\eta-P\eta$, and
\[
\beta_{i} := (u\rho,\phi'_{i})\,, \quad i=1,2,\dots,N+1\,.
\]
If $\beta=(\beta_{1},\beta_{2},\dots,\beta_{N+1})^{T}$ then
\[
|\beta| \leq C h^{3.5}\,,
\]
where $|\cdot|$ is the $l_{2}$ norm in $\mathbb{R}^{N+1}$.
\end{lemma}
\begin{proof} We have
\begin{align*}
\beta_{1} & = -\tfrac{1}{h}\int_{I_{1}}u\rho dx\,, \\
\beta_{i} & = \tfrac{1}{h}\int_{I_{i-1}}u\rho dx - \tfrac{1}{h}\int_{I_{i}}u\rho dx\,, \quad
i=2,3,\dots,N\,, \\
\beta_{N+1} & = \tfrac{1}{h}\int_{I_{N}}u\rho dx\,.
\end{align*}
If $dist(x_{i-1},0) \geq \widetilde{C}h\ln 1/h$ and $dist(x_{i+1},1) \geq \widetilde{C}h\ln 1/h$,
i.e. if $i \in K$, where $K$ are the integers in the interval
$[2+\widetilde{C}\ln 1/h, N-\widetilde{C}\ln 1/h]$, then by $(\ref{eq528})$
\begin{equation}
|\beta_{i}| \leq C h^4\,.
\label{eq531}
\end{equation}
If on the other hand $dist(x_{i-1},0) \leq \widetilde{C}h\ln 1/h$
or $dist(x_{i+1},1) \leq \widetilde{C}h\ln 1/h$, i.e. if $i \in K_{1}$, where $K_{1}$ are the integers
in $[1, 1+\widetilde{C}\ln 1/h] \cup [N+1-\widetilde{C}\ln 1/h, N+1]$, then, by $(\ref{eq529})$ and
$(\ref{eq530})$
\[
|\beta_{i}| \leq C h^4 \ln 1/h\,.
\]
Therefore,
\begin{align*}
|\beta|^{2} = \sum_{i=1}^{N+1}|\beta_{i}|^{2} & =\sum_{i\in K_{1}}|\beta_{i}|^{2}+
\sum_{i\in K}|\beta_{i}|^{2} \\
& \leq (1 + 2\widetilde{C}\ln 1/h)Ch^8(\ln 1/h)^{2} + C(N-1-2\widetilde{C}\ln 1/h)h^8
\leq Ch^{7}\,.
\end{align*}
\end{proof}
\begin{lemma} Let $w \in C^2$, $v \in C^{1}$ and
και
\[
\beta_{i} = (v(w-Pw),\phi_{i})\,, \quad i=1,2,\dots,N+1\,.
\]
If $\beta=(\beta_{1},\beta_{2},\dots,\beta_{N+1})^{T}$, then
\[
|\beta| \leq C h^{3.5}\,,
\]
where $|\cdot|$ is the $l_{2}$ norm in $\mathbb{R}^{N+1}$.
\end{lemma}
\begin{proof} For $x\in I_{1}$ we have
\[
v(x) = v(h) + O(h)\,,
\]
and since $\|w-Pw\|_{\infty} = O(h^2)$,
\[
v(x)(w-Pw)(x) = v(h)(w-Pw)(x) + O(h^3)\,,
\]
and therefore
\[
\beta_{1} = O(h^4)\,.
\]
If $x \in I_{i-1}\cup I_{i}$, $i=2,3,\dots,N$, we have
\[
v(x) = v(x_{i}) + O(h)\,,
\]
whence
\[
v(x)(w-Pw)(x) = v(x_{i})(w-Pw)(x) + O(h^3)\,,
\]
and so
\[
\beta_{i} = O(h^4)\,.
\]
Similarly, $\beta_{N+1} = O(h^4)$, and finally $|\beta| = O(h^{3.5})$.
\end{proof}
\subsection{Numerical experiments} We considered first the $(\ref{cb})$ system, that we
discretized by the nonstandard method analyzed in the previous two sections using the subspaces
$S_{h}^{2}$ and $S_{h,0}^{3}$ for approximating $\eta$ and $u$, respectively. (We considered the
nonhomogeneous system with a suitable right-hand side so that the solution is
$\eta=\exp(2t)(\cos(\pi x) + x^{2} +2)$, $u=\exp(xt)(\sin(\pi x) + x^{3} -x^{2})$.) We integrated
the semidiscrete system up to $T=1$ by the fourth-order accurate explicit Runge-Kutta method
analyzed in paragraph $4.3$ with time step $k=h/10$; the temporal discretization error was
negligible in comparison to the spatial error.
\def\baselinestretch{1}
\captionsetup[subtable]{labelformat=empty,position=top,margin=1pt,singlelinecheck=false}
\scriptsize
\begin{table}[h]
\subfloat[$L^{2}$-errors]{
\begin{tabular}[h]{ | c | c | c | c | c | }\hline
$N$   &    $\eta$      &  $order$  &      $u$       &   $order$   \\ \hline
$40$  &  $0.1250(-2)$  &           &  $0.4057(-5)$  &            \\ \hline
$60$  &  $0.5555(-3)$  &  $2.001$  &  $0.1199(-5)$  &   $3.008$   \\ \hline
$80$  &  $0.3124(-3)$  &  $2.000$  &  $0.5051(-6)$  &   $3.004$   \\ \hline
$100$ &  $0.1999(-3)$  &  $2.000$  &  $0.2585(-6)$	&   $3.002$   \\ \hline
$120$ &	 $0.1388(-3)$  &  $2.000$  &  $0.1495(-6)$	&   $3.001$   \\ \hline
$140$ &	 $0.1020(-3)$  &  $2.000$  &  $0.9416(-7)$	&   $3.001$   \\ \hline
$160$ &	 $0.7810(-4)$  &  $2.000$  &  $0.6307(-7)$	&   $3.001$   \\ \hline
$180$ &  $0.6170(-4)$  &  $2.000$  &  $0.4429(-7)$  &   $3.001$   \\ \hline
\end{tabular}
}\qquad
\subfloat[$L^{\infty}$-errors]{
\begin{tabular}[h]{ | c | c | c | c | c | }\hline
$N$   &    $\eta$      &  $order$  &      $u$       &   $order$   \\ \hline
$40$  &  $0.3342(-2)$  &           &  $0.5955(-5)$  &            \\ \hline
$60$  &  $0.1485(-2)$  &  $2.000$  &  $0.1780(-5)$  &   $2.979$   \\ \hline
$80$  &  $0.8357(-3)$  &  $1.999$  &  $0.7539(-6)$  &   $2.985$   \\ \hline
$100$ &  $0.5349(-3)$  &  $2.000$  &  $0.3870(-6)$	&   $2.989$   \\ \hline
$120$ &	 $0.3714(-3)$  &  $2.000$  &  $0.2243(-6)$	&   $2.991$   \\ \hline
$140$ &	 $0.2729(-3)$  &  $2.000$  &  $0.1414(-6)$	&   $2.992$   \\ \hline
$160$ &	 $0.2089(-3)$  &  $2.000$  &  $0.9483(-7)$	&   $2.993$   \\ \hline
$180$ &  $0.1651(-3)$  &  $2.000$  &  $0.6665(-7)$  &   $2.994$   \\ \hline
\end{tabular}
}
\\
\subfloat[$H^{1}$-errors]{
\begin{tabular}[h]{ | c | c | c | c | c | }\hline
$N$   &    $\eta$      &  $order$  &      $u$       &   $order$   \\ \hline
$40$  &  $0.3872$      &           &  $0.1048(-2)$  &            \\ \hline
$60$  &  $0.2581$      &  $1.000$  &  $0.4652(-3)$  &   $2.002$   \\ \hline
$80$  &  $0.1936$      &  $1.000$  &  $0.2616(-3)$  &   $2.001$   \\ \hline
$100$ &  $0.1549$      &  $1.000$  &  $0.1674(-3)$	&   $2.001$   \\ \hline
$120$ &	 $0.1290$      &  $1.000$  &  $0.1162(-3)$	&   $2.000$   \\ \hline
$140$ &	 $0.1106$      &  $1.000$  &  $0.8540(-4)$	&   $2.000$   \\ \hline
$160$ &	 $0.9678(-1)$  &  $1.000$  &  $0.6538(-4)$	&   $2.000$   \\ \hline
$180$ &  $0.8603(-1)$  &  $1.000$  &  $0.5166(-4)$  &   $2.000$   \\ \hline
\end{tabular}
}
\normalsize
\caption{Errors and orders of convergence. $(\ref{cb})$ system, nonstandard Galerkin
semidiscretization, $\eta_{h} \in S_{h}^{2}$, $u_{h} \in S_{h,0}^{3}$, uniform mesh.}
\label{tbl51}
\end{table}
\normalsize
Table \ref{tbl51} shows the resulting errors and orders of convergence in various norms. The optimal-order
$L^{2}$- and $L^{\infty}$ rates of convergence of Theorem $5.1$ are confirmed and the $H^{1}$ rates
are as expected. We obtained the same rates for the $(\ref{scb})$ system. Table \ref{tbl52} shows the
analogous errors and rates
\def\baselinestretch{1}
\captionsetup[subtable]{labelformat=empty,position=top,margin=1pt,singlelinecheck=false}
\scriptsize
\begin{table}[h]
\subfloat[$L^{2}$-errors]{
\begin{tabular}[h]{ | c | c | c | c | c | }\hline
$N$   &    $\eta$      &  $order$  &      $u$       &   $order$   \\ \hline
$40$  &  $0.2214(-4)$  &           &  $0.1529(-6)$  &            \\ \hline
$60$  &  $0.6628(-5)$  &  $2.975$  &  $0.3044(-7)$  &   $3.981$   \\ \hline
$80$  &  $0.2810(-5)$  &  $2.982$  &  $0.9670(-8)$  &   $3.986$   \\ \hline
$100$ &  $0.1443(-5)$  &  $2.986$  &  $0.3970(-8)$	&   $3.989$   \\ \hline
$120$ &	 $0.8370(-6)$  &  $2.989$  &  $0.1918(-8)$	&   $3.991$   \\ \hline
$140$ &	 $0.5279(-6)$  &  $2.990$  &  $0.1036(-8)$	&   $3.992$   \\ \hline
$160$ &	 $0.3540(-6)$  &  $2.992$  &  $0.6081(-9)$	&   $3.993$   \\ \hline
$180$ &  $0.2488(-6)$  &  $2.993$  &  $0.3799(-9)$  &   $3.994$   \\ \hline
\end{tabular}
}\qquad
\subfloat[$L^{\infty}$-errors]{
\begin{tabular}[h]{ | c | c | c | c | c | }\hline
$N$   &    $\eta$      &  $order$  &      $u$       &   $order$   \\ \hline
$40$  &  $0.4812(-4)$  &           &  $0.6323(-6)$  &            \\ \hline
$60$  &  $0.1422(-4)$  &  $3.006$  &  $0.1290(-6)$  &   $3.921$   \\ \hline
$80$  &  $0.5993(-5)$  &  $3.005$  &  $0.4145(-7)$  &   $3.946$   \\ \hline
$100$ &  $0.3066(-5)$  &  $3.004$  &  $0.1713(-7)$	&   $3.958$   \\ \hline
$120$ &	 $0.1773(-5)$  &  $3.003$  &  $0.8314(-8)$	&   $3.966$   \\ \hline
$140$ &	 $0.1116(-5)$  &  $3.003$  &  $0.4507(-8)$	&   $3.972$   \\ \hline
$160$ &	 $0.7474(-6)$  &  $3.003$  &  $0.2651(-8)$	&   $3.976$   \\ \hline
$180$ &  $0.5248(-6)$  &  $3.002$  &  $0.1659(-8)$  &   $3.979$   \\ \hline
\end{tabular}
}
\\
\subfloat[$H^{1}$-errors]{
\begin{tabular}[h]{ | c | c | c | c | c | }\hline
$N$   &    $\eta$      &  $order$  &      $u$       &   $order$   \\ \hline
$40$  &  $0.5999(-2)$  &           &  $0.3854(-4)$  &            \\ \hline
$60$  &  $0.2654(-2)$  &  $2.011$  &  $0.1152(-4)$  &   $2.978$   \\ \hline
$80$  &  $0.1490(-2)$  &  $2.008$  &  $0.4884(-5)$  &   $2.984$   \\ \hline
$100$ &  $0.9520(-3)$  &  $2.006$  &  $0.2509(-5)$	&   $2.988$   \\ \hline
$120$ &	 $0.6605(-3)$  &  $2.005$  &  $0.1454(-5)$	&   $2.990$   \\ \hline
$140$ &	 $0.4850(-3)$  &  $2.004$  &  $0.9168(-6)$	&   $2.991$   \\ \hline
$160$ &	 $0.3711(-3)$  &  $2.004$  &  $0.6148(-6)$	&   $2.992$   \\ \hline
$180$ &  $0.2931(-3)$  &  $2.003$  &  $0.4321(-6)$  &   $2.993$   \\ \hline
\end{tabular}
}
\normalsize
\caption{Errors and orders of convergence. $(\ref{cb})$ system, nonstandard Galerkin
semidiscretization, $\eta_{h} \in S_{h}^{3}$, $u_{h} \in S_{h,0}^{4}$, uniform mesh.}
\label{tbl52}
\end{table}
\normalsize
of convergence for the same problem but now discretized so that $\eta_{h} \in S_{h}^{3}$, i.e.
in the space of $C^{1}$ quadratic splines, and $u_{h} \in S_{h,0}^{4}$, i.e. in the space of cubic
splines that vanish at $x=0$ and at $x=1$. We observe e.g. that the associated $L^{2}$-rates
of convergence for this higher order pair of subspaces are, as expected, equal to $3$ for
$\eta$ and $4$ for $u$. The $(\ref{scb})$ system gave the same results. \par
The error analysis in paragraphs $5.1$ and $5.2$ depended strongly on the assumption of uniform mesh.
In the case of nonuniform mesh we expect some order reduction. For example, Table \ref{tbl53} shows the
$L^{2}$ errors and orders of convergence that we obtained for $(\ref{cb})$ by approximating
$(\eta_{h}, u_{h})$ in $S_{h}^{2}\times S_{h,0}^{3}$ and using for both subspaces the exact solution
and the nonuniform mesh with which Table \ref{tbl23}(a) was produced. {\em{Both}}  rates are apparently
reduced by one now.
\def\baselinestretch{1}
\scriptsize
\begin{table}[h]
\begin{tabular}[h]{ | c | c | c | c | c | }\hline
$N$   &    $\eta$      &  $order$  &      $u$       &   $order$   \\ \hline
$80$  &  $0.1372(-2)$  &           &  $0.3879(-5)$  &            \\ \hline
$160$ &  $0.6849(-3)$  &  $1.002$  &  $0.9569(-6)$  &   $2.019$   \\ \hline
$240$ &  $0.4564(-3)$  &  $1.001$  &  $0.4241(-6)$  &   $2.007$   \\ \hline
$320$ &  $0.3422(-3)$  &  $1.000$  &  $0.2383(-6)$	&   $2.004$   \\ \hline
$400$ &	 $0.2738(-3)$  &  $1.000$  &  $0.1524(-6)$	&   $2.002$   \\ \hline
$480$ &	 $0.2281(-3)$  &  $1.000$  &  $0.1058(-6)$	&   $2.002$   \\ \hline
$560$ &	 $0.1955(-3)$  &  $1.000$  &  $0.7775(-7)$	&   $2.001$   \\ \hline
$640$ &  $0.1711(-3)$  &  $1.000$  &  $0.5952(-7)$  &   $2.001$   \\ \hline
$720$ &  $0.1509(-3)$  &  $1.000$  &  $0.4702(-7)$  &   $2.001$   \\ \hline
\end{tabular}
\normalsize
\caption{$L^{2}$-errors and orders of convergence. $(\ref{cb})$ system, nonstandard Galerkin
semidiscretization, $\eta_{h} \in S_{h}^{2}$, $u_{h} \in S_{h,0}^{3}$, quasiuniform mesh
with $\tfrac{\max h_{i}}{\min h_{i}}=1.5$}
\label{tbl53}
\end{table}
\normalsize
\section{Remarks on standard Galerkin methods for related hyperbolic problems}
We conclude the paper with a section of remarks on the application of standard Galerkin methods
on some simple first-order hyperbolic problems. Our aim is to show examples of such problems
in which the techniques previously developed in this paper may be applied to some advantage.
\subsection{Initial-boundary-value problem for a single hyperbolic equation} For $T > 0$ consider
the initial-boundary-value problem
\begin{equation}
\begin{aligned}
& \eta_{t} +  \eta_{x} = 0\,, \qquad \,\,\,\, 0\leq x\leq 1\,, \quad 0\leq t\leq T\,,\\
& \eta(x,0) = \eta_{0}(x)\,, \quad 0\leq x\leq 1\,,\\
& \eta(0,t) = 0\,, \qquad \quad 0\leq t\leq T\,.
\end{aligned}
\label{eq61}
\end{equation}
It is clear that if $\eta_{0}\in C^{k}$, for integer $k\geq 0$, and $\eta_{0}^{(j)}(0)=0$,
$j=0,\dots,k$, then $(\ref{eq61})$ possesses a unique solution $\eta$ given for
$(x,t) \in [0,1]\times [0,T]$ by
\begin{equation*}
\eta(x,t) =
\begin{cases}
\eta_{0}(x-t) &, \quad x > t\,, \\
0 &, \quad x < t\,.
\end{cases}
\end{equation*}
The solution is classical if $k\geq 1$ and, for example, has the property that $\eta(\cdot,t) \in C^{k}$
for all $t \in [0,T]$. For integer $N\geq 2$ and $h=1/N$ we consider the uniform mesh given by
$x_{i}=ih$, $i=0,1,\dots,N$, and put $I_{j}=[x_{j-1},x_{j}]$. We let
\[
S_{h}^{2}\llap{\raise1.7ex\hbox{\scriptsize $\circ$}\hskip5pt}=
\{\phi \in C[0,1] : \phi\big|_{I_{j}} \in \mathbb{P}_{1},\,\,
1\leq j\leq N, \,\, \phi(0)=0\}\,,
\]
and denote by $\{\phi_{j}\}_{j=1}^{N}$ the basis of
$S_{h}^{2}\llap{\raise1.7ex\hbox{\scriptsize $\circ$}\hskip5pt}$ with the property that
$\phi_{i}(x_{j})=\delta_{ij}$. The standard Galerkin method for $(\ref{eq61})$  may be then defined as
follows: We seek $\eta_{h} : [0,T] \to S_{h}^{2}\llap{\raise1.7ex\hbox{\scriptsize $\circ$}\hskip5pt}$,
such that
\begin{equation}
\begin{aligned}
& (\eta_{ht},\phi) + (\eta_{hx},\phi) = 0\,, \quad \forall \phi \in
 S_{h}^{2}\llap{\raise1.7ex\hbox{\scriptsize $\circ$}\hskip5pt}\,, \,\, 0\leq t\leq T\,, \\
& \eta_{h}(0) = P\eta_{0}\,,
\end{aligned}
\label{eq62}
\end{equation}
where $P$ is the $L^{2}$-projection operator onto
$S_{h}^{2}\llap{\raise1.7ex\hbox{\scriptsize $\circ$}\hskip5pt}$. Obviously,
$(\ref{eq62})$ has a unique solution that satisfies e.g.
$\|\eta_{h}(t)\| \leq \|\eta_{h}(0)\|$, $t\in [0,T]$. Recall that we have
$\|v-Pv\| \leq Ch^{2}\|v\|_{2}$ for $v \in H^{2}$ with $v(0)=0$, and
$\|v-Pv\|_{\infty} \leq Ch^{2}\|v\|_{2,\infty}$ for $v \in C^{2}$ with $v(0)=0$. Our aim is to prove
an optimal-order $L^{2}$ error estimate for the solution of $(\ref{eq62})$, thus extending the
result of Dupont, \cite{d}, for the periodic initial-value problem for $\eta_{t}+\eta_{x}=0$ to the
case of the initial-boundary-value problem at hand. To do this, we first prove a lemma along the
lines of the proof of Lemma $5.6$.
\begin{lemma} Suppose that $v$ is $C^{2}$ and piecewise $C^{3}$ on $[0,1]$ and satisfies
$v(0)=v'(0)=v''(0)=0$. Then, if $\rho=v-Pv$ and
$\zeta \in S_{h}^{2}\llap{\raise1.7ex\hbox{\scriptsize $\circ$}\hskip5pt}$ is such that
\[
(\zeta,\phi) = (\rho,\phi') \quad \forall \phi \in
S_{h}^{2}\llap{\raise1.7ex\hbox{\scriptsize $\circ$}\hskip5pt}\,,
\]
we have $\|\zeta\| \leq Ch^{2}\|v^{(3)}\|_{\infty}$.
\end{lemma}
\begin{proof}
We will first show, as in Lemma $5.6$, that if $\ve_{i}=\int_{I_{i}}\rho dx$, $1\leq i\leq N$,
then $\ve_{i} = O(h^{4})$. Let $Pv=\sum_{j=1}^{N}c_{j}\phi_{j}$. Then, since
$(Pv,\phi_{i}) = (v,\phi_{i})$, $1\leq i\leq N$, we have $Gc=b$, where $G_{ij}=(\phi_{j},\phi_{i})$,
$1\leq i,j\leq N$, is the $N\times N$ tridiagonal matrix with $G_{ii}=2h/3$, $1\leq i\leq N-1$,
$G_{NN}=h/3$, and $G_{ij}=h/6$ if $|i-j|=1$, and $b_{i}=(v,\phi_{i})$, $1\leq i\leq N$. Hence,
\begin{equation}
\ve_{i} = \int_{I_{i}}\rho dx = \int_{I_{i}} v dx - \gamma_{i}\,, \quad
1\leq i\leq N\,,
\label{eq63}
\end{equation}
where
\begin{equation}
\gamma_{i} = \int_{I_{i}}Pv dx =
\begin{cases}
\tfrac{hc_{1}}{2} & \text{if}\quad i=1\,, \vspace{5pt} \\
\tfrac{h(c_{i-1}+c_{i})}{2} & \text{if} \quad 2\leq i\leq N\,.
\end{cases}
\label{eq64}
\end{equation}
The equations of the system $Gc=b$ are
\begin{align*}
& \tfrac{h}{6}(4c_{1} + c_{2}) = b_{1}\,, \\
& \tfrac{h}{6}(c_{i-1} + 4c_{i} + c_{i+1}) = b_{i}\,, \quad 2\leq i\leq N-1\,,\\
& \tfrac{h}{6}(c_{N-1} + 2c_{N}) = b_{N}\,,
\end{align*}
and may be rewritten as
\begin{align*}
& 3\cdot \tfrac{hc_{1}}{2} + \tfrac{h(c_{1}+c_{2})}{2} = 3b_{1}\,, \\
& \tfrac{hc_{1}}{2} + 4\cdot\tfrac{h(c_{1}+c_{2})}{2} + \tfrac{h(c_{2}+c_{3})}{2}
=3(b_{1}+b_{2})\,, \\
& \tfrac{h(c_{i-2}+c_{i-1})}{2} + 4\cdot\tfrac{h(c_{i-1}+c_{i})}{2}
+\tfrac{h(c_{i}+c_{i+1})}{2} = 3(b_{i-1}+b_{i})\,, \quad 3\leq i\leq N-1\,,\\
& \tfrac{h(c_{N-2}+c_{N-1})}{2} + 5\cdot\tfrac{h(c_{N-1} + c_{N})}{2}
= 3(b_{N-1} + 2b_{N})\,.
\end{align*}
Therefore, from $(\ref{eq64})$ we conclude that
\begin{equation}
\Gamma \gamma = \beta\,,
\label{eq65}
\end{equation}
where $\gamma=(\gamma_{1},\dots,\gamma_{N})^{T}$, $\Gamma$ is the $N\times N$ tridiagonal matrix with
$\Gamma_{11}=3$, $\Gamma_{ii}=4$, $2\leq i\leq N-1$, $\Gamma_{NN}=5$ and $\Gamma_{ij}=1$ if
$|i-j|=1$, and $\beta=(\beta_{1},\dots,\beta_{N})^{T}$ with $\beta_{1}=3b_{1}$,
$\beta_{i}=3(b_{i-1}+b_{i})$, $2\leq i\leq N-1$, and $\beta_{N}=3(b_{N-1}+2b_{N})$. Hence, from
$(\ref{eq63})$ and $(\ref{eq65})$ we have that
\begin{equation}
\Gamma\ve = r\,,
\label{eq66}
\end{equation}
where $\ve=(\ve_{1},\dots,\ve_{N})^{T}$, and $r=(r_{1},\dots,r_{N})^{T}$ is given by
\begin{align*}
& r_{1}=3\int_{0}^{h}vdx + \int_{h}^{2h}vdx - 3b_{1}\,, \\
& r_{i} = \int_{x_{i-2}}^{x_{i-1}}vdx + 4 \int_{x_{i-1}}^{x_{i}}vdx +
\int_{x_{i}}^{x_{i+1}}vdx - 3(b_{i-1}+b_{i})\,, \quad 2\leq i\leq N-1\,, \\
& r_{N} = \int_{x_{N-2}}^{x_{N-1}}vdx + 5\int_{x_{N-1}}^{x_{N}}vdx - 3(b_{N-1}+2b_{N})\,.
\end{align*}
For $r_{1}$ we have
\[
r_{1} = 3\int_{0}^{h}v(1-\tfrac{x}{h})dx + \int_{h}^{2h}v(\tfrac{3x}{h} - 5)dx\,,
\]
and by our hypothesis on $v$, using Taylor's theorem with remainder in integral form, we have
\begin{equation}
|r_{1}| \leq Ch^{4} \|v^{(3)}\|_{L^{\infty}[0,2h]}\,.
\label{eq67}
\end{equation}
As in the proof of Lemma $5.6$ we may show, using Taylor's theorem with remainder in
integral form, that
\begin{equation}
|r_{i}| \leq Ch^{4} \|v^{(3)}\|_{\infty}\,, \quad 2\leq i\leq N\,.
\label{eq68}
\end{equation}
Writing now the matrix $\Gamma$ in the form $\Gamma=4(I-E)$, we see that $E$ is the a
$N\times N$ tridiagonal matrix with $\|E\|_{\infty}=1/2$. Hence $\Gamma$ is invertible,
$\|\Gamma^{-1}\|_{\infty}\leq 1/2$, and $(\ref{eq66})$-$(\ref{eq68})$ give
\begin{equation}
\max_{1\leq i\leq N}|\ve_{i}| \leq Ch^{4}\|v^{(3)}\|_{\infty}\,.
\label{eq69}
\end{equation}
Since $(\zeta,\phi_{i})=(\rho,\phi_{i}')$, $1\leq i\leq N$, we conclude from $(\ref{eq69})$ that
\begin{equation}
\max_{1\leq i\leq N}|(\zeta,\phi_{i})| \leq Ch^{3}\|v^{(3)}\|_{\infty}\,.
\label{eq610}
\end{equation}
It is straightforward to check now that Lemma $2.1(i)$ and $(ii)$ hold for
$S_{h}^{2}\llap{\raise1.7ex\hbox{\scriptsize $\circ$}\hskip5pt}$ as well,
{\em{mutatis mutandis}}. The conclusion of the Lemma follows.
\end{proof}
\begin{proposition} Suppose that $\eta_{0}\in C^{3}$ with
$\eta_{0}(0)=\eta'_{0}(0)=\eta_{0}''(0)=0$.
Then, if $\eta_{h}$ is the solution of the semidiscrete problem $(\ref{eq62})$ we have that
\[
\max_{0\leq t\leq T}\|\eta(t) - \eta_{h}(t)\|\leq C(T)h^{2}\max_{0\leq t\leq T}\|\eta\|_{3,\infty}\,.
\]
\end{proposition}
\begin{proof} Let $\rho=\eta-P\eta$, $\theta=P\eta - \eta_{h}$. Then, for $0\leq t\leq T$
\[
(\theta_{t},\phi) + (\theta_{x},\phi) + (\rho_{x},\phi)=0 \quad
\forall \phi \in S_{h}^{2}\llap{\raise1.7ex\hbox{\scriptsize $\circ$}\hskip5pt}\,.
\]
Putting in this $\phi=\theta$ we obtain, using integration by parts, that
\begin{equation}
\tfrac{1}{2}\tfrac{d}{dt}\|\theta(t)\|^{2} + \tfrac{1}{2}\theta^{2}(1,t) =
-\rho(1,t)\theta(1,t) + (\rho,\theta_{x})\,.
\label{eq611}
\end{equation}
Now, $|\rho(1,t)|\leq \|\rho\|_{\infty}\leq Ch^{2}\|\eta\|_{2,\infty}$ and our hypotheses
on $\eta_{0}$ imply that $\eta(\cdot,t)\in C^{2}$ and piecewise $C^{3}$ for
$0\leq t\leq T$ and satisfies  $\partial_{x}^{j}\eta(0,t)=0$
for $j=0$, $1$, $2$, $0\leq t\leq T$. Therefore, by Lemma $6.1$ we have, if
$(\zeta,\phi_{i})=(\rho,\phi_{i}')$, $1\leq i\leq N$, that
$(\rho,\theta_{x})=(\zeta,\theta)\leq Ch^{2}\|\eta\|_{3,\infty}\|\theta\|$. Hence, by the
arithmetic-geometric mean inequality we obtain from $(\ref{eq611})$ that
\[
\tfrac{1}{2}\tfrac{d}{dt}\|\theta(t)\|^{2} \leq Ch^{4}\|\eta(t)\|_{3,\infty}^{2}
+ \|\theta(t)\|^{2}\,, \quad 0\leq t\leq T\,.
\]
Therefore, Gronwall's Lemma and $(\ref{eq62})$ give the conclusion of the Proposition.
\end{proof}
In Table \ref{tbl61} we show the results of a numerical experiment that we performed to investigate
the influence of the degree of compatibility of the initial data $\eta_{0}$ at $x=0$
on the order of convergence of the $L^{2}$- error of the standard Galerkin semidiscretization
of $(\ref{eq61})$ using the space
\def\baselinestretch{1}
\scriptsize
\begin{table}[h]
\begin{center}
\begin{tabular}[h]{ | c | c | c | c | c | c | c | c | c | }\hline
& \multicolumn{2}{c |}{$x\exp(x)$} & \multicolumn{2}{c |}{$x^{2}\exp(x)$}
& \multicolumn{2}{c |}{$x^{3}\exp(x)$} & \multicolumn{2}{c |}{$x^{4}\exp(x)$} \\ \hline
$N$ &  $L^{2}$-$error$  &  $order$ &  $L^{2}$-$error$ & $order$ &
$L^{2}$-$error$ & $order$ & $L^{2}$-$error$  &  $order$ \\ \hline
$50$      &  $0.9811(-3)$  &           &  $0.3014(-3)$ &         & $0.4786(-3)$
&         &  $0.6765(-3)$  &             \\ \hline
$100$     &  $0.4436(-3)$  &  $1.049$  &  $0.7583(-4)$ & $1.991$ & $0.1204(-3)$
& $1.991$ &  $0.1705(-3)$  &  $1.989$ \\ \hline
$150$     &  $0.2891(-3)$  &  $1.056$  &  $0.3379(-4)$ & $1.993$ & $0.5360(-4)$
& $1.995$ &  $0.7597(-4)$  &  $1.994$ \\ \hline
$200$     &  $0.2145(-3)$  &  $1.039$  &  $0.1904(-4)$ & $1.993$ & $0.3018(-4)$
& $1.997$ &  $0.4279(-4)$  &  $1.996$ \\ \hline
$250$     &  $0.1699(-3)$  &  $1.043$  &  $0.1221(-4)$ & $1.993$ & $0.1933(-4)$
& $1.997$ &  $0.2740(-4)$  &  $1.997$ \\ \hline
$300$     &  $0.1407(-3)$  &  $1.034$  &  $0.8489(-5)$ & $1.993$ & $0.1343(-4)$
& $1.998$ &  $0.1904(-4)$  &  $1.997$ \\ \hline
$350$     &  $0.1199(-3)$  &  $1.037$  &  $0.6244(-5)$ & $1.992$ & $0.9868(-5)$
& $1.998$ &  $0.1399(-4)$  &  $1.998$ \\ \hline
$400$     &  $0.1045(-3)$  &  $1.030$  &  $0.4786(-5)$ & $1.992$ & $0.7557(-5)$
& $1.998$ &  $0.1072(-4)$  &  $1.998$ \\ \hline
$450$     &  $0.9255(-4)$  &  $1.033$  &  $0.3785(-5)$ & $1.992$ & $0.5972(-5)$
& $1.999$ &  $0.8469(-5)$  &  $1.998$ \\ \hline
$500$     &  $0.8305(-4)$  &  $1.028$  &  $0.3069(-5)$ & $1.991$ & $0.4838(-5)$
& $1.999$ &  $0.6861(-5)$  &  $1.998$ \\ \hline
$550$     &  $0.7528(-4)$  &  $1.030$  &  $0.2538(-5)$ & $1.991$ & $0.3999(-5)$
& $1.999$ &  $0.5671(-5)$  &  $1.999$ \\ \hline
$600$     &  $0.6885(-4)$  &  $1.026$  &  $0.2135(-5)$ & $1.990$ & $0.3360(-5)$
& $1.999$ &  $0.4766(-5)$  &  $1.999$ \\ \hline
\end{tabular}
\end{center}
\caption{$L^{2}$ errors and convergence rates. Standard Galerkin semidiscretization
for problem $(\ref{eq61})$ with piecewise linear, continuous functions and four choices
of $\eta_{0}(x)$.}
\label{tbl61}
\end{table}
\normalsize
$S_{h}^{2}\llap{\raise1.7ex\hbox{\scriptsize $\circ$}\hskip5pt}$ for the space discretization
and $P\eta_{0}$ as  initial data. We considered $(\ref{eq61})$ with $T=0.5$ taking successively
$\eta_{0}(x)$ equal to $x\exp(x)$, $x^{2}\exp(x)$, $x^{3}\exp(x)$ and $x^{4}\exp(x)$. (The
temporal discretization was effected with the Crank-Nicolson implicit scheme with time step
$k=h/3$ in all cases; this time step was sufficiently small to ensure that the temporal error was
about two orders of magnitude less than the spatial error in all cases.) The $L^{2}$ errors
and associated rates of convergence of the numerical scheme at $t=T=0.5$ for diminishing
$h=1/N$ are shown in Table \ref{tbl61}. Recall that if $\eta_{0}(x)=x^{k}\exp(x)$, $k=1,2,3,\dots$,
the solution $\eta(x,t)$ of $(\ref{eq61})$ is in $C^{k-1}(0,1)$ (and piecewise $C^{\infty}$)
for all $t\in [0,T]$. The $L^{2}$ orders of convergence turn out to be practically equal to $2$ for
$k=4$ and $k=3$ in agreement with the error estimate that was proved.
The rate is apparently converging to one in the case $k=1$ in which $(\ref{eq61})$
possesses a generalized, piecewise smooth solution, while in the borderline case $k=2$ the convergence
rate has not stabilized. (When we increased the number of spatial intervals $N$, we observed that
this rate fell steadily and was e.g. equal to $1.9826$ at $N=1600$.)
\subsection{A hyperbolic equation with a variable coefficient that vanishes
at the boundary} Consider the problem
\begin{equation}
\begin{aligned}
& \eta_{t} + (u\eta)_{x} = 0\,, \quad 0\leq x\leq 1\,, \,\,\,\, 0\leq t\leq T\,,\\
& \eta(x,0) = \eta_{0}(x)\,, \quad 0\leq x\leq 1\,,
\end{aligned}
\label{eq612}
\end{equation}
where $u$ is a given smooth function that vanishes at $x=0$ and at $x=1$. No boundary conditions
are imposed on $\eta$. The problem is well posed and may be solved by the method of
characteristics. Uniqueness of solutions follows easily: If we multiply the p.d.e. in
$(\ref{eq612})$ by $\eta$ and integrate over $[0,1]$ by parts we see that for $0\leq t\leq T$
\[
\tfrac{1}{2}\tfrac{d}{dt}\|\eta\|^{2} + \tfrac{1}{2}\int_{0}^{1}\eta^{2}u_{x}dx=0\,,
\]
from which, by Gronwall's Lemma, there follows that
\[
\|\eta(t)\| \leq C \|\eta_{0}\|\,, \quad 0\leq t\leq T\,,
\]
where $C=C(T,\max_{0\leq t\leq T}\|u_{x}\|_{\infty})$. \par
An optimal-order $L^{2}$ error estimate for the standard Galerkin semidiscretization in $S_{h}^{2}$
on a uniform mesh may be proved for this problem using the methods previously developed in this
paper. (This was first pointed out in \cite{ant}.) The semidiscrete problem consists in finding
$\eta_{h} : [0,T]\to S_{h}^{2}$ so that
\begin{equation}
\begin{aligned}
& (\eta_{ht},\phi) + ((u\eta_{h})_{x},\phi) = 0\,, \quad \forall \phi \in S_{h}^{2}\,,
\,\,\,\, 0\leq t\leq T\,, \\
& \eta_{h}(0) = P\eta_{0}\,,
\end{aligned}
\label{eq613}
\end{equation}
where $P$ is the $L^{2}$ projection onto $S_{h}^{2}$.
\begin{proposition} Suppose that $u\in C_{0}^{2}$ and that the solution of $(\ref{eq612})$
satisfies $\eta \in C(0,T;C^{4})$. Then
\[
\max_{0\leq t\leq T}\|\eta - \eta_{h}\|_{\infty} \leq Ch^{2}\,.
\]
\end{proposition}
\begin{proof}
Putting $\theta=P\eta - \eta_{h}$, $\rho=\eta - P\eta$ we have for $0\leq t\leq T$
that
\begin{equation}
(\theta_{t},\phi) + ((u\theta)_{x},\phi) + ((u\rho)_{x},\phi)=0 \quad
\forall \phi \in S_{h}^{2}\,.
\label{eq614}
\end{equation}
Now, by Lemma $5.1$, if for each $t\in [0,T]$, $\zeta \in S_{h}^{2}$ is defined by
\[
(\zeta,\phi) = ((u\rho)_{x},\phi) \quad \forall \phi \in S_{h}^{2}\,,
\]
we have that $\|\zeta\|\leq Ch^{3}$. Therefore, putting $\phi=\theta$ in $(\ref{eq613})$
and using integration by parts we see that
\[
\|\theta\| \leq Ch^{3}\,.
\]
We conclude that
$\|\eta - \eta_{h}\|_{\infty} \leq \|\theta\|_{\infty} + \|\rho\|_{\infty}\leq
Ch^{-1/2}\|\theta\| + \|\rho\|_{\infty} \leq Ch^{2}$.
\end{proof}
This observation may prove useful in the error analysis of first-order problems with advection
terms of the form $(u\eta)_{x}$, where $u$ satisfies Dirichlet boundary conditions. However, as we
saw in Section $2$ in the case of the Boussinesq systems under study the presence of the linear
$u_{x}$ term in the first p.d.e. of $(\ref{cb})$ or $(\ref{scb})$ is the obstacle that
prevents achieving optimal-order estimates. For example, examining the proof of Theorem $2.1$
we note that if the $u_{x}$ term was missing, then the terms $(\xi_{x},\phi)$ and $(\sigma_{x},\phi)$
would not be present in $(\ref{eq221})$ and, consequently, by appeal to Lemma $2.2$ with
$w(0)=w(1)=0$, the estimates of the terms in $(\ref{eq224})$ would lead to an $O(h^{2})$ bound
in the right-hand side of $(\ref{eq225})$.
\subsection{Numerical experiments with nonuniform meshes} When we used a {\em{quasiuniform mesh}}
in the standard Galerkin semidiscretization $(\ref{eq62})$ of the initial-boundary-value problem
$(\ref{eq61})$, numerical results (not shown here) indicating that the $L^{2}$ error
$\|\eta - \eta_{h}\|$ is of $O(h)$; of the same order of accuracy is of course the upper bound
of the $L^{2}$ error predicted in a straightforward manner by the theory. \par
The same differences between the cases of uniform and nonuniform meshes apparently also persist
in the case of first-order hyperbolic {\em{systems}}. Consider, for example, the following
initial-boundary-value problem for the nonhomogeneous wave equation written in first-order
system form:
\begin{equation}
\begin{aligned}
\begin{aligned}
& \eta_{t} + u_{x} = f\,, \\
& u_{t} + \eta_{x} = g\,,
\end{aligned}
\qquad & 0\leq x\leq 1\,, \,\,\, 0\leq t\leq T\,,\\
\eta(x,0)=\eta_{0}(x)\,,\,\,\, & u(x,0) = u_{0}(x)\,, \,\,\,\,\, 0\leq x\leq 1\,,\\
\eta(0,t)=0\,, \quad \,\, u(& 1,t) = 0\,, \quad 0\leq t\leq T\,.
\end{aligned}
\label{eq615}
\end{equation}
For the purposes of the numerical experiment we took as exact solution of $(\ref{eq615})$ the pair
of functions $\eta(x,t)=x\exp(x(t+1))$, $u(x,t)=(x-1)\exp(xt)$ with appropriate initial conditions
and right-hand sides. We discretized the problem in space by the standard Galerkin method using
piecewise linear continuous elements on a uniform mesh with meshlength $h=1/N$ and also on a
quasiuniform mesh with $\Delta x=1.6/N$, $h_{2i-1}=0.75\Delta x$, $h_{2i}=0.5\Delta x$,
$1\leq i\leq N/2$. For the temporal discretization we used the classical, fourth-order explicit
Runge-Kutta scheme with $k=h$ and $k=\Delta x$, respectively. As initial values $\eta_{h}(0)$,
$u_{h}(0)$ we took the interpolants of $\eta_{0}$ and $u_{0}$. The $L^{2}$ errors at $T=0.4$
and the corresponding orders of convergence are shown, for increasing $N$, in Table \ref{tbl62}(a)
in the case of the quasiuniform mesh and in Table \ref{tbl62}(b) for the uniform mesh. The
experiment strongly suggests that the $L^{2}$ errors ore of $O(h^{2})$ for a uniform mesh and of
\def\baselinestretch{1}
\captionsetup[subtable]{labelformat=empty, position=bottom, singlelinecheck=true}
\scriptsize
\begin{table}[h]
\subfloat[(a)]{
\begin{tabular}[b]{ | c | c | c | c | c | }\hline
$N$   &  $L^{2}$-$errors (\eta)$   &  $order$  &    $L^{2}$-$errors (u)$  &  $order$   \\ \hline
$80$  &  $0.1577(-2)$  &           &  $0.1607(-2)$  &            \\ \hline
$160$ &  $0.7813(-3)$  &  $1.013$  &  $0.8002(-3)$  &  $1.006$   \\ \hline
$240$ &  $0.5194(-3)$  &  $1.007$  &  $0.5327(-3)$  &  $1.003$   \\ \hline
$320$ &  $0.3891(-3)$  &  $1.004$  &  $0.3992(-3)$	&  $1.002$   \\ \hline
$400$ &	 $0.3110(-3)$  &  $1.004$  &  $0.3193(-3)$	&  $1.002$   \\ \hline
$480$ &	 $0.2590(-3)$  &  $1.003$  &  $0.2660(-3)$	&  $1.001$   \\ \hline
$560$ &	 $0.2219(-3)$  &  $1.002$  &  $0.2279(-3)$	&  $1.001$   \\ \hline
$640$ &  $0.1941(-3)$  &  $1.002$  &  $0.1994(-3)$  &  $1.001$   \\ \hline
$720$ &  $0.1725(-3)$  &  $1.002$  &  $0.1772(-3)$  &  $1.001$   \\ \hline
\end{tabular}
}\qquad
\scriptsize
\subfloat[(b)]{
\begin{tabular}[b]{ | c | c | c | c | c | }\hline
$N$   &    $L^{2}$-$errors (\eta)$      &  $order$  &  $L^{2}$-$errors (u)$ &  $order$   \\ \hline
$80$  &  $0.8117(-4)$  &           &  $0.3932(-4)$  &            \\ \hline
$160$ &  $0.2029(-4)$  &  $2.000$  &  $0.9825(-5)$  &  $2.001$   \\ \hline
$240$ &  $0.9021(-5)$  &  $2.000$  &  $0.4369(-5)$  &  $1.999$   \\ \hline
$320$ &  $0.5074(-5)$  &  $2.000$  &  $0.2457(-5)$	&  $2.000$   \\ \hline
$400$ &	 $0.3248(-5)$  &  $2.000$  &  $0.1573(-5)$	&  $1.999$   \\ \hline
$480$ &	 $0.2255(-5)$  &  $2.000$  &  $0.1092(-5)$	&  $2.000$   \\ \hline
$560$ &	 $0.1657(-5)$  &  $2.000$  &  $0.8026(-6)$	&  $2.000$   \\ \hline
$640$ &  $0.1269(-5)$  &  $2.000$  &  $0.6145(-6)$  &  $2.000$   \\ \hline
$720$ &  $0.1002(-5)$  &  $2.000$  &  $0.4856(-6)$  &  $2.000$   \\ \hline
\end{tabular}
}
\normalsize
\caption{$L^{2}$-errors and orders of convergence, problem $(\ref{eq615})$. Standard Galerkin
discretization with piecewise linear, continuous functions. (a): quasiuniform mesh
(b): uniform mesh.}
\label{tbl62}
\end{table}
\normalsize
$O(\Delta x)$ for quasiuniform. \par As a final remark, we consider the following initial-boundary value
problem of a hyperbolic system with a viscous term
\begin{equation}
\begin{aligned}
\begin{aligned}
& \eta_{t} + u_{x} = 0\,, \\
& u_{t} + \eta_{x} - u_{xx} = 0\,,
\end{aligned}
\qquad  0 & \leq x\leq 1\,, \,\,\, 0\leq t\leq T\,,\\
\eta(x,0)=\eta_{0}(x)\,,\,\,\,\, u(x,0) & = u_{0}(x)\,, \,\,\,\,\, 0\leq x\leq 1\,,\\
u(0,t)=0\,,\,\,\,\, u(1,t) = 0\,,& \quad 0\leq t\leq T\,,
\end{aligned}
\label{eq616}
\end{equation}
and discretize it by the standard Galerkin method. We obtain similar results to those of Section $2$
in which the dispersive term $u_{xxt}$ is present instead of $u_{xx}$. (Cf.
Remark $2.3$) In a numerical example we took a nonhomogeneous version of $(\ref{eq616})$ with exact
solution $\eta(x,t)=\exp(2t)(\cos(\pi x) + x + 2)$, $u(x,t)=\exp(xt)(\sin(\pi x) + x^{3}-x^{2})$ and
discretized the problem in space with the standard Galerkin method using a uniform mesh with $h=1/N$
and the quasiuniform mesh of the previous example. We integrated up to $T=0.5$ with the classical,
fourth-order, explicit Runge-Kutta scheme with $k=h^{2}/25$ in the uniform mesh and
$k=(\Delta x)^{2}/25.6$ in the nonuniform mesh case. The errors and $L^{2}$ convergence rates at
$T=0.5$ shown in Table \ref{tbl63} suggest that the $L^{2}$ errors for $\eta$ and $u$ are of
$O(\Delta x)$ and $O(\Delta x)^{2}$, respectively, in the quasiuniform mesh case, and of $O(h^{3/2})$ and
$O(h^{2})$ in the uniform case, exactly as in the linear dispersive case.
\def\baselinestretch{1}
\captionsetup[subtable]{labelformat=empty, position=bottom, singlelinecheck=true}
\scriptsize
\begin{table}[h]
\subfloat[(a)]{
\begin{tabular}[b]{ | c | c | c | c | c | }\hline
$N$   &  $L^{2}$-$errors (\eta)$   &  $order$  &    $L^{2}$-$errors (u)$  &  $order$   \\ \hline
$20$  &  $0.2013(-1)$  &           &  $0.1771(-2)$  &            \\ \hline
$40$  &  $0.1004(-1)$  &  $1.003$  &  $0.4434(-3)$  &  $1.998$   \\ \hline
$60$  &  $0.6696(-2)$  &  $0.999$  &  $0.1971(-3)$  &  $1.999$   \\ \hline
$80$  &  $0.5023(-2)$  &  $0.999$  &  $0.1109(-3)$	&  $1.999$   \\ \hline
$100$ &	 $0.4019(-2)$  &  $0.999$  &  $0.7099(-4)$	&  $2.000$   \\ \hline
$120$ &	 $0.3350(-2)$  &  $0.999$  &  $0.4930(-4)$	&  $2.000$   \\ \hline
$140$ &	 $0.2872(-2)$  &  $0.999$  &  $0.3622(-4)$	&  $2.000$   \\ \hline
$160$ &  $0.2513(-2)$  &  $0.999$  &  $0.2774(-4)$  &  $2.000$   \\ \hline
\end{tabular}
}\qquad
\scriptsize
\subfloat[(b)]{
\begin{tabular}[b]{ | c | c | c | c | c | }\hline
$N$   &    $L^{2}$-$errors (\eta)$      &  $order$  &  $L^{2}$-$errors (u)$ &  $order$   \\ \hline
$20$  &  $0.3605(-2)$  &           &  $0.1396(-2)$  &            \\ \hline
$40$  &  $0.1149(-2)$  &  $1.649$  &  $0.3487(-3)$  &  $2.001$   \\ \hline
$60$  &  $0.6011(-3)$  &  $1.598$  &  $0.1550(-3)$  &  $2.000$   \\ \hline
$80$  &  $0.3823(-3)$  &  $1.574$  &  $0.8716(-4)$	&  $2.000$   \\ \hline
$100$ &	 $0.2700(-3)$  &  $1.559$  &  $0.5578(-4)$	&  $2.000$   \\ \hline
$120$ &	 $0.2035(-3)$  &  $1.549$  &  $0.3873(-4)$	&  $2.000$   \\ \hline
$140$ &	 $0.1605(-3)$  &  $1.542$  &  $0.2846(-4)$	&  $2.000$   \\ \hline
$160$ &  $0.1307(-3)$  &  $1.537$  &  $0.2179(-4)$  &  $2.000$   \\ \hline
\end{tabular}
}
\normalsize
\caption{$L^{2}$-errors and orders of convergence, problem $(\ref{eq616})$. Standard Galerkin
discretization with piecewise linear, continuous functions. (a): quasiuniform mesh
(b): uniform mesh.}
\label{tbl63}
\end{table}
\normalsize
\bibliographystyle{amsalpha} 

\end{document}